 \documentclass[11pt]{article}

\usepackage{mathtools}
\usepackage[colorlinks=true, linkcolor=black, citecolor=black, urlcolor=black]{hyperref}
\usepackage{amsmath}
\usepackage{amsthm}
\usepackage[alphabetic,nobysame]{amsrefs}
\usepackage{amssymb}
\usepackage{graphicx}
\usepackage{eucal}
\usepackage{dsfont}
\usepackage{enumitem}
\usepackage{cancel}
\usepackage[all]{xy}
\usepackage{amsrefs}

\usepackage{color}

\usepackage{tikz}

\usetikzlibrary{shapes,arrows}

\def\RR{\mathds R}

\DeclareMathOperator{\Ima}{\mathrm{Im}}
\DeclareMathOperator{\Diff}{\mathrm{Diff}}

\DeclareMathOperator{\dom}{\mathrm{dom}}

\DeclareMathOperator{\Hom}{\mathrm{Hom}}
\DeclareMathOperator{\rank}{\mathrm{rk}}
\DeclareMathOperator{\loc}{\mathrm{loc}}
\DeclareMathOperator{\rk}{\mathrm{rk}}
\DeclareMathOperator{\Lin}{\mathrm{Lin}}

\DeclareMathOperator{\X}{\mathfrak{X}}

\def\({\left(}
\def\){\right)}
\def\[{\left[}
\def\]{\right]}

\def\al{\alpha}
\def\be{\beta}
\def\de{\delta}

\def\ep{\epsilon}

\def\te{\theta}
\def\la{\lambda}
\def\ze{\zeta}
\def\om{\omega}
\def\si{\sigma}

\def\Ga{\Gamma}
\def\Te{\Theta}

\def\Om{\Omega}

\def\tors{\tau}

\def\g{\mathfrak{g}}

\theoremstyle{plain}
\newtheorem{theorem}{Theorem}[section]
\newtheorem{proposition}[theorem]{Proposition}
\newtheorem{corollary}[theorem]{Corollary}
\newtheorem{lemma}[theorem]{Lemma}
\newtheorem{definition}[theorem]{Definition}

\theoremstyle{definition}

\newtheorem{obsx}[theorem]{Remark}
\newenvironment{observation}
  {\pushQED{\qed}\obsx}
  {\popQED\endobsx}

\newtheorem{examplex}[theorem]{Example}
\newenvironment{example}
  {\pushQED{\qed}\examplex}
  {\popQED\endexamplex}
  
\def\del{\partial} 
\def\na{\nabla} 
\def\arr{\rightarrow} 

\def\Etemp{F} 
\def\Ftemp{E} 

\begin{document}

\title{From PDEs to Pfaffian fibrations}
\date{\today}


\author{Francesco Cattafi \footnote{Mathematical Institute, Universiteit Utrecht, The Netherlands, \texttt{f.cattafi@uu.nl} } 
\and Marius Crainic \footnote{Mathematical Institute, Universiteit Utrecht, The Netherlands, \texttt{m.crainic@uu.nl} } 
\and Mar\'ia Amelia Salazar \footnote{Departamento de Matem\'atica Aplicada, Universidade Federal Fluminense, Brazil, \texttt{mariasalazar@id.uff.br} } }

\maketitle

\begin{abstract}
 We explain how to encode the essential data of a PDE on jet bundle into a more intrinsic object called Pfaffian fibration. We provide motivations to study this new notion and show how prolongations, integrability and linearisations of PDEs generalise to this setting.
\end{abstract}
    
\begin{center}
\end{center}

\tableofcontents

\section{Introduction} 

The history and the importance of theory of Partial Differential Equations (PDEs) are themselves subjects of entire monographs. Very briefly, one of the central questions is that of integrability, i.e.\ the existence of local solutions of a PDE passing through each point. There are various techniques to handle this problem, each one with its own advantages. For instance, the Cartan-K\"{a}hler theorem can be applied in many instances but it is bound to the analytic setting. Another standard approach starts with the attempt to solve the PDE formally- and then one talks about formal integrability. One also discovers the notion of prolongations, which allows one to replace a given PDE with a new, ``larger'' one, but which may be easier to handle and, of course, has the same solutions as the original one. Another standard technique is that of linearising a PDE- the outcome is a PDE that is much easier to handle and which, although it usually has different solutions than the original one, often carries important informations about the behaviour of the solutions one is looking for. 

\medskip 

While the role of jets is clear already in the local study of PDEs, formalising it was important for a more geometric approach to PDEs; this was carried out by Charles Ehresmann \cite{Ehr51} in the 50's, leading to the the notion of jet bundle as the standard formalism to study PDEs on manifolds. Solutions of a PDE were then becoming sections of a bundle $R\rightarrow M$ over a manifold $M$, the PDEs themselves were becoming subspaces $P \subset J^k R$ of the bundles of jets of sections of $R$, and the condition for a section $s$ of $R$ to be a solution of $P$ was that $j_{x}^{k}s\in P$ for all $x\in M$.  Many of the notions and techniques known in the local study (e.g.\ prolongations, linearisations, etc) were then recast in this formalism; that process quickly revealed the notion of Cartan distribution(s), or Cartan form(s), on the jet bundles $J^kR$ and its central role to the entire geometric theory. The various ways of understanding these objects gave rise to different schools/approaches to the subject, e.g.\ depending on whether (and how) one works with vector fields or differential forms; see, among others, the monographs \cites{Boc99, Kra86, Olv86, Sau89, Sto00}. For instance, the Cartan-K\"{a}hler theorem mentioned above is now part of the standard material on Exterior Differential Systems \cite{Bry91}. Another example is the notion of diffiety, due to Vinogradov and his school \cite{Vin01}, which arises from the theory of differential equations in the same way the concept of algebraic variety arise from that of algebraic equations. It is important to mention that all these modern approaches to PDEs (including ours) have been greatly influenced by the pioneering works by Sophus Lie \cite{Lie88} and by \'{E}lie Cartan \cites{Car04, Car05}.

\medskip 

The aim of this paper is to emphasise and (hopefully) to clarify the importance of the Cartan distribution/form even further. The main message is that what is needed for the theory to work is not the jet bundles $J^k R$ but just the fibration $P \arr M$ together with the induced Cartan distribution; or, in our language, a {\it Pfaffian fibration}.  Of course, there are points at which the jet bundles are still important, but often they are just  ``noise'' in the background, giving rise to unnecessarily complicated formulae. Also, we are aware that this point may be, in principle, rather obvious to the specialists (and there are similar theories carried out at the level of infinite jet bundles), but we find it useful to spell it out in detail, taking care of the subtleties that arise along the way. We hope that, in this way, various techniques and notions that are often presented in a rather pragmatic way, via ``down to earth'' (but complicated) local formulae, become more transparent to people with a more geometric background/interests. 

\medskip 

On the other hand, our main motivation for carrying this out comes from the study of Lie pseudogroups and of geometric structures: the theory is now ready to be used right away to understand the main structures underlying the theory of Lie pseudogroups $\Gamma$ and, furthermore, of $\Gamma$-structures on manifolds. For instance, one may say that the Pfaffian groupoids of \cite{Sal13} are just the multiplicative version of the Pfaffian fibrations discussed in this paper. Again, while this may still seem rather abstract for someone whose interest on Lie pseudogroups comes from the study of symmetries of concrete PDEs, it reveals the theory from a more geometric perspective, pinpointing the actual structure that makes everything work, and uncovers rather unexpected bridges with other parts of Differential Geometry. For instance, the abstract (Pfaffian) groupoids arising from pseudogroups behave surprisingly similar to the symplectic groupoids of Poisson Geometry. This similarity can really be exploited: for instance, the analogues of the Hamiltonian spaces and of Morita equivalences of Poisson Geometry turn out to be precisely what is needed to study general geometric structures and their integrability - as carried out in \cite{Cat19}. In all of these, the notion of Pfaffian fibration that is being discussed in this paper has the role of building block. 

\medskip 

A few words on the structure of this paper. In section 2 we review the basics on PDEs: this include the notion of (finite-order) jet bundle and Cartan form, as well as its linear counterpart, the classical Spencer operator. Moreover, we recall the concepts of prolongation and of integrability of a PDE, and various important theorems in this area, together with the necessary technical tools, i.e.\ tableaux and Spencer cohomology.

In section 3 we introduce the definition of Pfaffian fibration in a double way, using either a differential form or a distribution. We define as well a number of objects naturally inspired from the theory of PDEs, such as symbol spaces and curvatures, and then we focus on the particular case of linear Pfaffian fibrations and the process of linearising Pfaffian fibration along a solution. We conclude with the discussion of the main examples that sparked our interest in this field.

Section 4 is the core of the paper: we use the definitions and the ideas from the previous section to develop a theory of prolongation in the context of Pfaffian fibrations. In particular, we present first the general notions of morphism and prolongation in the Pfaffian category, followed by the explicit construction of a prolongation which inspired from the classical notion of prolongation for PDEs, and which is ``universal'' in a certain sense. Since this process is not always possible, we show concrete criteria for the prolongability of a Pfaffian fibration, and then see how these results translate to the linear picture.

Last, in section 5 we apply the theorems from section 4 in order to tackle integrability of Pfaffian fibrations up to a finite order, as well as formal integrability. Borrowing ideas and terminology from the theory of $G$-structures, we associate inductively to any Pfaffian fibration certain obstructions to formal integrability, called the {\it torsions}. In this setting, we can prove fundamental result such as the Goldschmidt criterion for formal integrability, the integrability criterion for Pfaffian fibrations of finite type and the fact that analytic formally integrable Pfaffian fibrations are integrable.

\subsection*{Notations and conventions}

All manifolds and maps are smooth, unless stated explicitly otherwise. By a {\it fibration} between two manifolds $P$ and $M$ we  mean a surjective submersion $\pi: P \arr M$. Given a fibration $\pi: P \arr M$, by $T^\pi P$ we denote the {\it vertical bundle} $\ker (d\pi) \subset TP$ over $P$. By $\Om^k (P, \mathcal{N})$ we mean the space of differential $k$-forms on the manifold $P$ with coefficients in some vector bundle $\mathcal{N} \arr P$, i.e.\ $\Om^k (P, \mathcal{N}):= \Ga (\wedge^k T^*P \otimes \mathcal{N})$. We say that a form $\te \in \Om^k (P, \mathcal{N})$ is {\it (pointwise) surjective} if the linear map $\te_p: \wedge^k T_p P \arr \mathcal{N}_p$ is surjective for every $p \in P$.
Often we are given a vector bundle $E \arr M$, so that one can consider the pullback $\pi^*E \arr P$; when $\pi$ is clear from the context, we may omit the pullback notation. In particular, we often write $\Om^k (P, E)$ instead of $\Om^k (P,\pi^*E)$.

\subsection*{Acknowledgements}
 The authors would like to thank Luca Vitagliano for useful comments and suggestions. The first and second authors were supported by the NWO grant number 639.033.312. The third author was supported by CNPq Universal grant number 409552/2016-0; this study was financed in part by the Coordena\c{c}\~ao de Aperfei\c{c}oamento de Pessoal de N\'ivel Superior - Brasil (CAPES) - Finance code 001.

\section{PDEs on jet bundles}\label{sec:PDE}

The different notions which we will develop in the theory of Pfaffian fibrations arise as a way to geometrically encapsulate the fundamental properties of PDEs. In this section we review the various geometrical notions that motivated and inspired the analogous ones for Pfaffian fibrations.
In particular, we will restrict our attention to PDEs defined on jets of sections of a fibration, which are easier to deal with, more widely studied in the literature, and powerful enough for many applications. We will therefore not consider the more general setting of jets of submanifolds, even if we think that a suitable generalisation of Pfaffian fibrations could be introduced also in that case.

\subsection{Jets, PDEs, and the Cartan form}\label{sec:Cartan_form}

A PDE of order $k$ in the function $u=u(x_1,\ldots,x_n):\RR^n\to \RR^m$ is an equation of the form 
$$F \left( x_i,u, \frac{\partial^{|\alpha|} u}{\partial x_1^{\alpha_1}\cdots \partial x_m^{\alpha_m}} \right)=0$$ 
for all $m$-multi-indices $\alpha=(\alpha_1,\ldots,\alpha_m)$ with $|\alpha| =\alpha_1+\cdots+\alpha_m \leq k$.
However, in order to describe a conceptual theory of PDEs on manifolds, the language of jets will be very well suited, since it sees the PDE as a submanifolds of the $k$-jet bundle given by the zero locus of $F$ (see \cite{Kra86,Sau89} as references for jets).

More precisely, the $k$-jet of $u$ at $x\in\RR^n$ is encoded by all the partial derivatives of $u$ up to order $k$: this means that two such functions $u$ and $v$ have the same $k$-jet at $x$ if they have the same Taylor polynomial of degree $k$ at $x$. This defines an equivalence relation $\sim_x^k$ on the space of smooth maps $C^\infty (\RR^n,\RR^m)$; the induced equivalence class of $u$, called the {\it $k$-jet of $u$ at $x$}, is denoted by $j_x^ku$. Such an element of this quotient has coordinates $u^\alpha = \frac{\partial^{|\alpha|} u}{\partial x_1^{\alpha_1}\cdots \partial x_m^{\alpha_m}}$, with $\alpha$ as above.

More generally, given a fibration (by which we mean a surjective submersion) 
$$\pi:R\to M,$$
we denote by $\Gamma(R)$ the set of sections of $\pi$, and by $\Gamma_{\loc}(R)$ the local ones. For any integer $k\geq 0$, the space of {\bf $k$-jets of sections} of $\pi$ is defined as
$$J^kR:=\{j^k_x\beta\mid \beta\in\Gamma_{\loc}(R),\ x\in\mathrm{dom}(\beta)\}.$$
This set has a canonical manifold structure which fibres over $M$: indeed, the collection of $k$-jets of functions $u:\RR^n\to \RR^m$ coincides with $J^k S$, when $S = \RR^n \times \RR^m$ is the trivial bundle over $\RR^n$ with fibre $\RR^m$, hence the coordinates described above can be taken as local coordinates for $J^kR$ when $\dim(M) = n$ and $\rk(R)=m$.

In the case $k= 1$, a jet $j^1_x \be$ is completely encoded by $\be(x)\in R$ and the differential $d_x\be: T_xM\arr T_{\be(x)}R$. Actually, since $\be$ is a section of $\pi$, its differential is completely encoded by its image
$$ H_{\be(x)} := \textrm{Im}(d_x\be) \subset T_{\be(x)}R.$$
Indeed, $d_x\be$ will be the inverse of $d \pi |_{H}$. Of course, $H$ is not an arbitrary subspace: it is a complement in $T_{\be(x)}R$ of the vertical subspace $T^\pi_{\be(x)} R$. Such a complement is also called a {\bf horizontal subspace} for $\pi$. Therefore, one has
\begin{equation}\label{alternative_description_first_jet}
\begin{split}
J^1 R & \cong \{(p, H_p) \mid p\in R, H_p\subset T_pR\ \textrm{horizontal}\}\\
&  \cong \{ (p, \zeta) \mid p \in R, \zeta: T_x M \arr T_{\zeta(x)}R \text{ linear }, d\pi \circ \zeta = id \}.
\end{split}
\end{equation}

The various jet bundles are related to each other by the obvious projection maps
$$\cdots \to J^2R\to J^1R\to J^0R=R,$$
and each projection $J^kR\to J^{k-1}R$ is an affine bundle modelled on the pullback of $S^k(T^*M)\otimes T^\pi R$ (see for example Theorem 6.2.9 of \cite{Sau89}). To simplify the notation, we denote all the projections above by $\mathrm{pr}$, and the fibration of $J^kR$ over $M$ by $\pi.$ Having at hand the language of jets, we can naturally formalise the following definition (see \cite{Gol67b}): a PDE of order $k$ on $\pi$ is a (connected) submanifold
\begin{equation}\label{def_PDE_jet_bundle}
P\subset J^kR
\end{equation}
which fibres over $M$. Typically, a PDE is also asked to satisfy some mild regularity conditions. While one could develop most of the theory with no further assumptions, these conditions simplify the exposition and avoid unnecessary technicalities. Accordingly, in the rest of the thesis we will follow Section 1.4 of \cite{Yud16} and require that, if $P \subset J^k R$ is a PDE, then $\mathrm{pr}(P) \subset J^{k-1} R$ is a submanifold as well, and the projections $P \arr \mathrm{pr}(P)$ and $\mathrm{pr}(P) \arr \pi(P) \subset M$ are submersions.

A {\bf (local) solution} of a PDE $P$ is any (local) section $\beta$ of $R$ with the property that 
$$j_x^k\beta\in P \quad \forall x\in\mathrm{dom}(\beta);$$ 
this means that the (local) section $j^k\beta$ of $J^kR$ must be a (local) section of $P$. In other words, the set of solutions of $P$, denoted by $\mathrm{Sol}(P)$, is made up by all the sections $\alpha$ of $P$ which are {\it holonomic}, i.e.\ of the form $\alpha=j^k\beta$ for $\beta$ a section of $R$. Accordingly, in order to detect which sections are holonomic, we introduce the {\bf Cartan 1-form} 
\begin{equation*}
\te_{\mathrm{can}}\in\Omega^1(J^kR,\mathrm{pr}^*(T^\pi (J^{k-1}R)))\end{equation*}
with $T^\pi (J^{k-1}R):=\ker(d\pi)$ the vector bundle over $J^{k-1}R$ of vectors tangent to the fibres of $J^{k-1}R\to M$. For instance, in the case $k=1$, $\te_{\mathrm{can}}$ is defined as follows: if $p:=j^1_x\beta$, and $X\in T_pJ^1R$,
\begin{equation}\label{eq:Cartan_form}
(\te_{\mathrm{can}})_p (X) := d\mathrm{pr}(X)-d_x\beta (d\pi(X))\in T^\pi_{\beta(x)}R.
\end{equation}
In the general case, at level $k$, $\te_{\mathrm{can}}$ is defined analogously (it is the difference between the two canonical ways to move from the $k$- to the $(k-1)$-jet space). Moreover, we let $\mathcal{C}:=\ker(\te_{\mathrm{can}})$ be the kernel of the Cartan form, called the {\bf Cartan distribution} (see \cite{Boc99, Kra86, Olv86}).

The main property of this new object is the following:
\begin{lemma}\label{lemma:holonomic} 
A section $\alpha$ of $J^kR\to M$ is holonomic, i.e.\ of the form $\al = j^k\beta$, $\beta\in\Gamma(R)$, if and only if $\alpha^*\te_{\mathrm{can}}=0$ (equivalently, the section $d\alpha:TM\to TJ^kR$ takes values in $\mathcal{C}$).
\end{lemma}

Conceptually this means that we can characterise the solutions of $P$ only in terms of $P$ viewed as a bundle over $M$ (and not as a subbundle of $J^kR$), together with the restriction of $\te_{\mathrm{can}}$ to $P$:
$$\mathrm{Sol}(P)\cong \Gamma(P,\te_{\mathrm{can}}):=\{\alpha\in\Gamma(P)\mid \alpha^*\te_{\mathrm{can}}=0\}.$$
In other words, for the study of PDEs, the only relevant data is a fibration $P\to M$ endowed with an appropriate 1-form (or, equivalently, with its kernel): this will be our starting point for the definition of a Pfaffian fibration (which forgets the ambient jet space).

\subsection{Linear PDEs and Spencer operators}\label{sec:Spencer_operator}

If $R=E$ is a vector bundle over $M$, $J^kE$ is canonically a vector bundle over $M$ with fibrewise addition and multiplication by a scalar $\lambda \in \RR$ defined by 
$$j_x^k\beta+j_x^k\eta:=j^k_x(\beta+\eta), \quad \quad \lambda j^k_x \beta := j^k_x (\lambda \beta).$$
A linear PDE of order $k$ on $E$ is a vector subbundle $F\subset J^kE$ over $M$. As in the general case, solutions of $F$ are sections of $F$ that are holonomic; however, in this linear setting, the {\bf classical Spencer operator} of $E$ plays the role of the Cartan form \eqref{eq:Cartan_form}, i.e.\ detecting holonomic sections. As for the Cartan form, we will define explicitly this operator when $k=1$, using a very convenient way to describe sections of $J^1E$, known as the {\it Spencer decomposition}: it is the canonical isomorphism of vector spaces 
\begin{equation}\label{eq:Spencer_decomposition}\Gamma(J^1E) \cong \Gamma(E)\oplus \Omega^1(M,E).\end{equation}
This decomposition comes from the following short exact sequence of vector bundles over $M$
\begin{equation}\label{eq:exact_sec}
0\to \Hom(TM,E)\overset{i}{\to}J^1E \xrightarrow{pr} E\to 0,
\end{equation}
where $i$, at the level of sections, is defined as $i(df\otimes s):=fj^1s-j^1(fs)$. Although the sequence \eqref{eq:exact_sec} does not have a canonical right splitting, at the level of sections it does: $s\mapsto j^1s$. This gives the decomposition \eqref{eq:Spencer_decomposition}, so that the classical Spencer operator $D^{\mathrm{clas}}$ is by definition the projection to the second component:
\begin{equation}\label{eq:Spencer_operator}
D^{\mathrm{clas}}:\Gamma(J^1E)\to \Omega^1(M,E).
\end{equation}
This operator has been extensively studied, see for example \cites{Gol76a,Gol77,Spe69,Spe71, Ngo,Ngo2}. Moreover, it is clear from its description that holonomic sections of $F \subset J^1E$ are precisely the sections $\alpha$ with the property that $D^{\mathrm{clas}} (\alpha)=0$.

The same story can be also repeated for higher jets, obtaining classical Spencer operators of the form $D^{\mathrm{clas}}:\Gamma(J^kE)\to\Omega^1(M,J^{k-1}E)$. More precisely, since $J^k E$ is a vector subbundle of $J^1 (J^{k-1}E)$ (over $M$), we can consider the Spencer operator of the vector bundle $J^{k-1}E \arr M$ (where $J^{k-1}E$ now plays the role of $E$) and restrict it to space of sections $\Ga(J^k E)$.

This operator $D^\mathrm{clas}: \Gamma(J^kE)\to\Omega^1(M,J^{k-1}E)$ vanishes on the solutions of $k^{th}$-order linear PDEs $F \subset J^k E$; hence, in analogy with the Cartan form, we can characterise the solutions of $F$ only in terms of $F$ viewed as a vector bundle (and not as a subbundle of $J^k E$), together with the restriction of $D=D^\mathrm{clas}$ to $F$:
$$\mathrm{Sol}(F)\cong \Gamma(F,D):=\{\alpha\in\Gamma(F)\mid D(\alpha)=0\}.$$
After defining Pfaffian fibrations as generalisation of PDEs with their Cartan forms, their linear counterpart (the linear Pfaffian fibrations) will be in turn a generalisation of linear PDEs with their classical Spencer operators.

\begin{observation}
We will also show (see Proposition \ref{linear_Pfaffian_bundles_with_connections} and Remark \ref{linearisation_of_linear_pfaffian_bundle}) that the classical Spencer operator can be seen as the linearisation of the Cartan form in the sense of Definition \ref{def_linearisation_pfaffian_bundle}. Actually, the whole picture relating the two objects can be more clearly seen in the world of Lie groupoids endowed with multiplicative forms and Lie algebroids endowed with (non classical) Spencer operators: the linearisation of a Lie groupoid is its Lie algebroid, and the linearisation of a multiplicative forms is a Spencer operator. See \cite{Cra12} as a reference for this topic.  
\end{observation}

\subsection{Prolongations of PDEs}\label{sec:PDE_prolongation}

The theory of prolongations of a PDE is a powerful tool to find its solutions; the literature on this topic is very rich and dates back several decades: we mention \cites{Gol76a, Gol76b, Olv86, Boc99, Vin01, Sto00} and we briefly and informally recall here some of these notions.\\

A prolongation of a PDE $P$ of order $k$ on $\pi:R\to M$ can be thought as the $(k+1)$-order PDE on $\pi$ obtained by taking the first order differential consequences of $P$, with the fundamental property of having the same space of solutions. The first naive guess to define the prolongation of $P$ would be simply $J^1 P = \{ j^1_x \si \mid \si \in \Ga (P) \}$. However, one immediately sees that $J^1P$ fails to be a PDE of order $k+1$ on $\pi$, since $J^1P$ is by construction a subset of $J^1 (J^k R)$, not of $J^{k+1} R \subset J^1 (J^k R)$. The way to solve this (set-theoretical) problem is to define the {\bf prolongation} $P^{(1)}$ as
\begin{equation}\label{eq:PDE_prolongation}
P^{(1)}:= J^1 P \cap J^{k+1} R.
\end{equation}
However, $P^{(1)}$ may fail to be a subbundle of $J^{k+1} R$; even more, $P^{(1)}$ may fail to be smooth. The PDE $P$ is said to be {\it integrable up to order $k+1$} if $P^{(1)}$ happens to be ``nice enough'', meaning that it is indeed a new PDE, and the projection $P^{(1)}\to P$ is a surjective submersion. If $P$ is integrable up to any order, it is said to be {\it formally integrable}. In this case we obtain a tower of bundles over $M$
\begin{equation}\label{eq:tower}
\cdots\to P^{(2)}\to P^{(1)}\to P,
\end{equation}
each of them endowed with the restriction of the Cartan form at every order, and all the maps being surjective submersions. 

The study of formal integrability of a PDE is a very useful tool to prove the existence of its solutions. This can be best seen in the analytic case, where formal integrability becomes a sufficient condition for {\it integrability}, i.e.\ finding local solutions at every point.

\begin{theorem}[Theorem 9.1 of \cite{Gol67b}]\label{existence_analytic_solutions_PDE}
If $P\subset J^kR$ is an analytic formally integrable PDE, then for every $p \in P^{(l)} \subset J^{k+l} R$ over $x \in M$ there is an analytic local solution $\be$ of $P$ such that $j^{k+l}_x \be = p$ on a neighbourhood of $x \in \dom(\be)$.

In particular, through every $p \in P$ there exists a local (analytic) solution of $P$.
\end{theorem}

However, in the smooth category Theorem \ref{existence_analytic_solutions_PDE} is not always true, since there exist formally integrable PDEs admitting no solution: see the famous Lewy counterexample \cite{Lewy57}.

To understand better the structure of the prolongations and the notion of formal integrability, one arrives at the notion of a tableau (see \cite{Bry91,Gol67a} and Definition \ref{def_tableau} in the next section). The tableaux are linear spaces that provide the framework to handle the intricate linear algebra behind PDEs; they also provide (Spencer) cohomological criteria for integrability of PDEs.

In particular, the {\bf symbol space} $\g$ of the PDE $P \subset J^k R$ is the following tableau 
\begin{equation}\label{eq:symbol_pde}
\g := \ker (d \mathrm{pr}: T^\pi P\to T J^{k-1}R ) \subset \ker (d \mathrm{pr}: T^\pi J^kR\to T J^{k-1}R )\cong S^k(T^*M)\otimes T^\pi R .
\end{equation}
This last isomorphism comes from the following short exact sequence:
\begin{equation}\label{eq:jets}
0\to S^kT^*M\otimes T^\pi R\to T^\pi J^kR\overset{d\mathrm{pr}}{\to} T^\pi J^{k-1}R\to 0,
\end{equation}
where we assume that all vector bundles sit on top $J^kR$ as pullback by the obvious maps (which we omit).

Using the definition of the Cartan form $\te_{\mathrm{can}}$, one checks that
\begin{equation*}\label{symbol_space}
\g = \{ v \in T^\pi P | \te_{\mathrm{can}}(v) = 0\} = T^{\pi} P \cap \ker(\te_{\mathrm{can}})\cong T^{\pi} P\cap  (S^k(T^*M)\otimes T^\pi R).
\end{equation*}

We can use the symbol space to provide a sufficient criterion for formal integrability of PDEs in terms of the prolongations and the Spencer cohomology of $\g$, which we recall in the next section (see \cite{Gol67b} for the original result and \cite{Yud16} for a more careful and modern proof):
\begin{theorem}[\bf Goldschmidt formal integrability criterion]\label{Goldschmidt_criterion_PDE}
Let $P$ be a PDE whose symbol space $\g$ is 2-acyclic, i.e.\ its Spencer cohomology $H^{k,2}(\g)$ vanishes for every $k \geq 0$.
If, moreover, $P^{(1)} \arr P$ is surjective and the prolongation $\g^{(1)} := \{ \eta \in S^{k+1} (T^* M) \otimes T^\pi R \mid \iota_X \eta \in \g \ \forall X \in \mathfrak{X}(M) \}$ is of constant rank, then $P$ is formally integrable.
\end{theorem}

\begin{observation}
In the same way that the theory of Pfaffian fibrations (developed in Section \ref{section_pfaffian_bundle}) is inspired from the theory of PDEs (recalled in Section \ref{sec:Cartan_form}), the notion of prolongation of a Pfaffian fibration (developed in section 4) comes as a geometrical way to describe the prolongation of a PDE only in terms of $P$ and the Cartan form, i.e.\ it isolates the properties that each map of \eqref{eq:tower} has in terms of $\te_{\mathrm{can}}$, forgetting the ambient jet space where $P$ lived.
\end{observation}

\subsection{Tableaux and Spencer cohomology}\label{section:tableaux}

As stated in Theorem \ref{Goldschmidt_criterion_PDE}, Goldschmidt provides in \cite{Gol67b} a cohomological criterion for formal integrability of a PDE in terms of its tableau. In this section we recall the general notions of tableau and Spencer cohomology, and state some facts relevant to the theory of PDEs. We also describe a small variant of the Spencer cohomology which will appear in the theory of Pfaffian fibrations, when dealing with a slightly more general notion of tableau.

\begin{definition}\label{def_tableau}
Let $E,F$ be vector spaces. A {\bf tableau} on $(E,F)$ is a linear subspace 
$$\g\subset \mathrm{Hom}(E,F).$$
We define the {\bf $1^{\text{st}}$ prolongation of} $\g$ as 
$$\g^{(1)}:=\{\eta\in \mathrm{Hom}(E, \mathfrak{g}): \eta(X)(Y)= \eta(Y)(X)\ \forall\ X, Y\in E\}=\Hom(E,\g)\cap S^{2}E^*\otimes F,$$
and we define inductively the {\bf $i^{\text{th}}$ prolongation} of $\g$ by 
 $$\mathfrak{g}^{(i)}:= \left(\mathfrak{g}^{(i-1)}\right)^{(1)}=\mathrm{Hom}(E, \mathfrak{g}^{(i-1)})\cap S^{i+1}E^*\otimes F.$$
\end{definition}

Next, we recall from Section 6 of \cite{Gol67b} that the following operator on $E$,
$$\delta:S^kE^*\to E^*\otimes S^{k-1}E^*,\quad \delta(\eta)(v)=\iota_v\eta\in S^{k-1}E^*,$$
extends to a linear map 
$$\delta:\wedge^jE^*\otimes S^kE^*\to \wedge^{j+1}E^*\otimes S^{k-1}E^*,\quad \delta(\omega\otimes \eta)=(-1)^j\omega\wedge\delta(\eta).$$
The resulting sequence of complexes (i.e.\ $\delta\circ\delta=0$) is of the form
\begin{equation}\label{eq:complex}
0\to S^kE^*\overset{\delta}{\to}E^*\otimes S^{k-1}E^*\overset{\delta}{\to}\cdots\overset{\delta}{\to}\wedge^{n}E^*\otimes S^{k-n}E^*\to 0
\end{equation}
for each $k$ (we set $S^lE^*=0$ for $l<0$). We tensor then the sequence \eqref{eq:complex} by $F$, and the operator $\delta$ by $id_F$, keeping still the same notation $\delta$. Note that, for a tableau $\g\subset \mathrm{Hom}(E,F)$, each prolongation $\mathfrak{g}^{(i)}$ can be described as the kernel of the restriction of the appropriate $\delta$ to $\textrm{Hom}(E, \mathfrak{g}^{(i-1)})$:
\begin{equation}\label{eq:differential}
 \delta=\delta_i: \textrm{Hom}(E, \mathfrak{g}^{(i-1)})\rightarrow \textrm{Hom}(\wedge^2E, \mathfrak{g}^{(i-2)}), \quad \delta(\eta)(X, Y)= \eta(X)(Y)- \eta(Y)(X). 
 \end{equation}
Therefore, it is not difficult to see that the sequence of complex \eqref{eq:complex} tensored with $F$ contains the subsequence of complexes
$$0\to \g^{(i)}\overset{\delta}{\to} E^*\otimes \g^{(i-1)}\overset{\delta}{\to}\wedge^2E^*\otimes \g^{(i-2)}\overset{\delta}{\to}\cdots\overset{\delta}{\to}\wedge^{i}E^*\otimes \g\overset{\delta}{\to} \wedge^{i+1}E^*\otimes F,$$
for each $i$. At $\wedge^mE^*\otimes \g^{(l)}$, the cocycles are denoted by
$$Z^{l,m}(\g):=\ker(\delta:\wedge^mE^*\otimes\g^{(l)}\to \wedge^{m+1}E^*\otimes\g^{(l-1)}),$$
and the coboundaries by
$$B^{l,m}(\g):=\mathrm{Im}(\delta:\wedge^{m-1}E^*\otimes\g^{(l+1)}\to \wedge^{m}E^*\otimes\g^{(l)});$$
the resulting cohomology groups are denoted by
\begin{equation}\label{eq:Spencer_cohomology}
H^{l,m}(\g):=Z^{l,m}(\g)/B^{l,m}(\g).
\end{equation}
Note that by construction $H^{l,1}(\g)=0$ for all $l\geq0$. The resulting cohomology is called the {\bf Spencer cohomology of the tableau $\g$.}

\begin{definition}\label{def_acyclic}
Let $r\geq 1$ be an integer. A tableau $\g$ is said to be {\bf $r$-acyclic} if 
$$H^{l,m}(\g)=0,\quad \forall 1\leq m\leq r,\ l\geq0,$$
and it is {\bf involutive} if it is $r$-acyclic for all $r\geq 1$, i.e.
$$H^{l,m}(\g)=0,\quad \forall m \geq 1,\ l\geq0.$$
\end{definition}

Later on, in the theory of Pfaffian fibrations, we will need a small variant of the Spencer complex of a tableau $\g\subset \Hom(E,F)$, in which the inclusion $\g\hookrightarrow \Hom(E,F)$ is replaced by a linear map 
$$\partial:\g\to \Hom(E,F).$$
In this case we define the {\bf $1^{\text{st}}$ prolongation of} $\g$ {\bf (with respect to $\partial$)} by
\begin{equation}\label{eq:prolongation_tableau}
\g^{(1)}(\partial):=\{\eta\in\Hom(E,\g)\mid \partial(\eta(X))(Y)=\partial(\eta(Y))(X),\ \forall X,Y\in E\}.
\end{equation} 
We can regard $\g^{(1)}(\partial)$ as a (classical) tableau on $(E,\g)$ and prolong it repeatedly, giving rise to the higher prolongations
$$\g^{(i)}(\partial):= S^iE^*\otimes\g\cap\Hom(E,\g^{(i-1)}), \quad i>1.$$
The Spencer sequence for $\g^{(1)}(\partial)$ can be extended in the following way: we extend $\partial$ to the linear map 
$$\delta_\partial:\wedge^jE^*\otimes\g\to\wedge^{j+1}E^*\otimes F,\quad \delta_\partial(\omega\otimes v)=(-1)^j\omega\wedge\partial(v).$$
A simple computation shows that the sequence of Spencer complexes of $\g^{(1)}(\partial)$ extends to the sequence of complexes
\begin{equation}\label{eq:partial_Spencer}0\to \g^{(i)}\overset{\delta}{\to} E^*\otimes \g^{(i-1)}\overset{\delta}{\to}\cdots \overset{\delta}{\to}\wedge^{i-1}E^*\otimes \g^{(1)}\overset{\delta}{\to}\wedge^{i}E^*\otimes \g\overset{\delta_\partial}{\to} \wedge^{i+1}E^*\otimes F,\end{equation}
for each $i$. We call the {\bf $\partial$-Spencer cohomology of $\g$} the cohomology of the sequence \eqref{eq:partial_Spencer}.

\ 

Now, when dealing with vector bundles $E,F$ over $M$ instead of vector spaces, all the notions discussed above extend naturally. In particular, a {\bf tableau bundle on $(E,F)$} is a bundle $\g \subset \Hom(E,F)$ of linear subspaces $\{\g_x \subset\Hom(E_x;F_x)\}_{x \in M}$, whose rank may vary; $\g$ is therefore a (smooth) vector subbundle over $M$ only when it is of constant rank. However, let us point out that the prolongations $\g^{(i)}$ may fail to be smooth even if we start with a smooth tableau bundle $\g$; at certain points the rank of some prolongations may not be constant anymore. One of the roles of the acyclicity condition from Definition \ref{def_acyclic} is to ensure the smoothness of the prolongations (see \cite{Gol67b,Yud16}):

\begin{lemma}\label{lemma:2-acyclic}
Let $\g\subset\Hom(E,F)$ be a tableau bundle over a connected manifold $M$. If $\g$ is 2-acyclic and $\g^{(1)}\subset\Hom(E,\g)$ is a vector bundle of constant rank, then $\g^{(i)}\subset\Hom(E,\g^{(i-1)})$ is also a vector bundle of constant rank for all $i\geq0.$
\end{lemma}

 \begin{observation}\label{rmk:lemma}
 Lemma \ref{lemma:2-acyclic} above also holds when dealing with a tableau bundle defined by a vector bundle map $\partial:\g\to \Hom(E,F)$ over $M$; in that case we are considering of course the $1^{\text{st}}$ prolongation $\g^{(1)}(\partial)$ w.r.t.\ $\partial$ from equation \eqref{eq:prolongation_tableau}. The proof follows the same lines as the proof of Lemma \ref{lemma:2-acyclic}.
 \end{observation}

A fundamental result in the theory of prolongations of PDEs states that, even if a tableau bundle is not involutive, it becomes so after a finite number of prolongations (see \cite[Lemma 2]{Gol68}):

\begin{theorem}
Let $\g$ be a tableau bundle. There exists an integer $l_0$ such that $\g^{(l)}$ is involutive for all $l\geq l_0$.
\end{theorem}

\section{Pfaffian fibrations and their geometry}\label{section_pfaffian_bundle}

We present now the central object of this paper, which we obtain by replacing the jet bundles with their hidden ``PDE structures". Furthermore, we explain how to recover in this new formalism many concepts from the theory of PDEs. As anticipated in the introduction (and discussed in the section of examples), we stress that the leading idea in this picture is not to give an abstract generalisation of the notion of PDE, but to shed light on its geometry.

\subsection{Pfaffian fibrations}

\begin{definition}\label{def_Pfaffian_bundle}
A {\bf Pfaffian fibration} $(P,\theta)$ over $M$ is a fibration $\pi:P\to M$ together with a pointwise surjective form $\te \in \Om^1(P,\mathcal{N})$ with coefficients in some vector bundle $\mathcal{N} \arr P$ such that
\begin{itemize}
\item $\te$ is $\pi$-regular, i.e.\ the restriction of $d\pi$ to $\ker(\theta)$ is pointwise surjective, or equivalently, $\ker(\theta)$ is transversal to the $\pi$-fibres:
$$T^\pi P + \ker(\te) = TP$$
\item $\te$ is $\pi$-involutive, i.e.\ the following distribution is involutive (in the sense of Frobenius)
\begin{equation}\label{eq:symbol_space}
\g(\theta):=T^\pi P\cap \ker\theta
\end{equation}
\end{itemize}
The form $\theta$ satisfying the properties above is called a {\bf Pfaffian form}, the vector bundle $\mathcal{N}$ the {\bf coefficient bundle}, and the distribution $\g(\theta)$ the {\bf symbol space} of $\theta$. 
\end{definition}

From the $\pi$-regularity of the Pfaffian form $\theta$ it follows that $\te$ has constant rank, hence it defines a vector subbundle $\g(\theta)\subset TP$ over $P$, i.e.\ a regular distribution (therefore it makes sense to ask it to be Frobenius-involutive).


\begin{observation}{\bf (Pfaffian distributions)}\label{rmk:correspondence}
We can look at pointwise surjective $\pi$-regular 1-forms from the equivalent point of view of distributions transversal to the $\pi$-fibres (or {\it$\pi$-transversal} distributions). In particular, starting with a $\pi$-transversal distribution $H\subset TP$
$$TP=H+T^\pi P,$$ one defines the {\it symbol space} of $H$
$$\g(H):=T^\pi P\cap H.$$
and the normal bundle
$$ \mathcal{N}_H:= TP/H \cong T^\pi P/\g(H)$$ 
If, moreover, the symbol space of $H$ is Frobenius-involutive, we call $H$ a {\it Pfaffian distribution}. We can then produce the surjective 1-form $\te_H$ (and say that $\theta_H$ is induced by $H$) given by the projection $TP\to \mathcal{N}_H$: by construction $\te_H$ satisfies $\ker (\te_H) = H$, is $\pi$-regular, and its symbol space coincides with that of $H$. 

Viceversa, if some distribution $H_\te \subset TP$ is already the kernel of a surjective $\pi$-regular 1-form $\theta\in\Omega^1(P,\mathcal{N})$, then its normal bundle becomes isomorphic to the coefficient bundle $\mathcal{N}$ via the map $\mathcal{N}_H\ni [u]\mapsto \theta(u)\in \mathcal{N}$. Under this isomorphism $\theta$ can be trivially written as the projection map $TP\to\mathcal{N}_H$. Clearly, $H_\te$ is $\pi$-transversal and its symbol space coincides with that of $\te$.
\end{observation}

\begin{proposition}\label{correspondence_pfaffian_forms_distribution}
The previous construction (of Remark \ref{rmk:correspondence}) gives a 1-1 correspondence: 
$$ \left\{   \begin{array}{c}
    \text{Pfaffian distributions}  \\
    H\subset TP
    \end{array} \right\} 
\tilde{\longleftrightarrow}
\left\{   \begin{array}{c}
    \text{(equivalence classes) of Pfaffian forms}  \\
    \theta\in\Omega^1(P,\mathcal{N})
    \end{array} \right\}. $$
where two forms $\te_1, \te_2$ are equivalent if there exists a vector bundle isomorphism $\phi: \mathcal{N}_1 \arr \mathcal{N}_2$ between their coefficients such that $\phi (\te_1 (v)) = \te_2 (v)$ $\forall v \in TP$.
\end{proposition}
              
 Accordingly, we have the equivalent notion of a {\it Pfaffian fibration} $(P,H)$ over $M$ when dealing with a {\it Pfaffian distribution}; in the following, we will switch freely between these two definitions (with forms or with distributions).

As we will see later (Proposition \ref{PDE_is_Pfaffian_bundlle}), PDEs on jet bundles are the main example of Pfaffian fibrations. With this in mind, the correspondence from Proposition \ref{correspondence_pfaffian_forms_distribution} recovers the correspondence between the Cartan form and the Cartan distribution.

\begin{observation}{\bf (Pfaffian systems)} Pfaffian fibrations are related to another way of studying differential equations, namely exterior differential systems (EDSs): every Pfaffian fibration induces a special kind of EDS.

An EDS is differential ideals of the exterior algebra of a manifold (see \cite{Bry91} for an introduction). In particular, a {\it Pfaffian system} is an EDS $\mathcal{I}\subset \Omega^*(P)$, generated as an exterior differential ideal in degree one, together with a transversal (or independence) condition.
It can be proved that a $\pi$-transversal distribution $H\subset TP$ induces such kind of Pfaffian systems,
and moreover, if $H$ is also $\pi$-involutive, the induced Pfaffian system turns out to be {\it linear} (another notion from the theory of EDSs, different from that of linear Pfaffian fibration in section \ref{sec:linear_Pfaffian}).


In conclusion, the framework of Pfaffian fibrations fits nicely in between two classical ways of studying differential equations:
\begin{itemize}
\item The formalism of jet bundles becomes a particular case (we give up the jets and retain the main structure given by the Cartan form).
\item The formalism of exterior differential systems is a more general case (we concentrate only on Pfaffian systems which have a transversal condition and are linear). \qedhere
\end{itemize}
\end{observation}

In both cases outlined above, a (local) solution of a PDE (i.e.\ a holonomic section in the jet bundle language, an ``integral manifold'' in the EDS language) corresponds to a (local) section of the Pfaffian fibration which pullbacks the Pfaffian form to zero:

\begin{definition} Given a Pfaffian fibration $(P,\theta)$, a {\bf holonomic (local) section} of $(P,\theta)$ is any (local) section $\beta$ of $P$ with the property that $\beta^*\theta=0$. The set of holonomic sections is denoted by $\Gamma(P,\theta)$ and that of local ones by $\Gamma_{\loc}(P,\theta)$. 

Analogously, a holonomic section of a Pfaffian fibration $(P,H)$ is any section $\beta$ of $P$ tangent to $H$ (i.e.\ $d\beta:TM\to TP$ takes values in $H$). We denote by $\Gamma(P,H)$ the set of holonomic sections, and by $\Gamma_{\loc}(P,H)$ that of local ones. 
\end{definition}

One of the main questions for Pfaffian fibrations is the integrability from the PDE point of view:

\begin{definition}\label{def_integrability}
A Pfaffian fibration $(P,\theta)$ (or $(P,H)$) is {\bf PDE-integrable} if through each point $p \in P$ there is a local holonomic section $\beta\in\Gamma_{\loc}(P,\theta)$ (or $\beta\in\Gamma_{\loc}(P,H)$), i.e.\ $\beta(\pi(p))=p$. 
\end{definition}

\begin{observation}
Of course the notion of holonomic section makes sense for any 1-form $\te$ on a fibration $P \arr M$, without any {\it a priori} relation with $T^\pi P$; however, PDE-integrability implies $\pi$-regularity of $\theta$, which is therefore {\it a posteriori} meaningful condition to ask in the definition. This can be more easily seen using $H=\ker\theta$: if for any $p$ there is a local section $\beta:M\to P$ passing through $p$ which is tangent to $H$, then 
$$T_xM=d(\pi\circ \beta)(T_xM)=d\pi(d\beta(T_xM))\subset d\pi (H_p),$$
where $x=\pi(p)$. This means that $d\pi$ is surjective when restricted to $H$, i.e.\ $H$ is $\pi$-transversal (or $\theta$ is $\pi$-regular).
\end{observation}  

A natural notion that comes into play when studying PDE-integrability is that of {\bf integral element} (see \cite{Bry91} for the analogous notion for an EDS). Intuitively, an integral element of $(P,H)$ is a linear subspace $V\subset T_pP$, $p\in P$, which is a ``good'' candidate to be the tangent space of a holonomic (local) section $\be$ that passes through $p$. Suppose that $V$ is indeed tangent to $\be$ i.e.\ $V=d\beta(T_xM),$ $x=\pi(p)$: this immediately implies that the dimension of $V$ is the dimension of $M$ and that $T_pP$ can be written as the direct sum $V\oplus T^\pi_pP$. Due to the holonomicity of $\be$, one further obtains that 
\begin{equation*}\label{eq:integral_element}V\subset H_p,\quad \text{and} \quad [u,v]_p\in V,\end{equation*}
for any $u=d\beta(X), v=d\beta(Y)$ with $X,Y\in\X(M)$.

In order to rewrite this last condition independently of the extensions of $u_p$ and $v_p$, we introduce the {\bf curvature map} of $H$,
\begin{equation}\label{eq:curvature}
\kappa_H:H\times H\to \mathcal{N}_H,
\end{equation}  
which is the $C^\infty(P)$-bilinear map defined at the level of sections by $(U,V)\mapsto [U,V]\mod H$. The Leibniz identity of the Lie bracket of vector fields implies that $\kappa_H$ is indeed well defined. Alternatively, if $H=\ker\theta$, the curvature map is denoted by $\kappa_\theta:H\times H\to \mathcal{N}$ and can be described by $(U,V)\mapsto \theta([U,V])$; therefore, it coincides with the restriction of $d_\nabla\theta$ to $\ker(\theta)$, where $d_\nabla$ is the De Rham-like differential associated to any linear connection $\nabla$ on $P$.

\begin{definition}\label{def:integral_elements}
Given a Pfaffian fibration $(P,H)$ (or $(P,\theta)$), a linear subspace $V\subset T_pP$ of dimension equal to the dimension of $M$ is called a {\bf partial integral element} if 
\begin{equation*}\label{condition_integral_element}
V\subset H_p \quad \text{and} \quad  T_pP=V\oplus T^\pi_pP.
\end{equation*}
If, moreover, the restriction of the curvature map $(\kappa_H)_p$ to $V\times V$ is zero, then $V$ is called an {\bf integral element}. 
\end{definition}

\subsection{Linear Pfaffian fibrations and relative connections}\label{sec:linear_Pfaffian}

In this section we discuss the notion of Pfaffian fibrations in the linear case, i.e.\ when the fibration $P \arr M$ is a vector bundle. We will also introduce an equivalent description in terms of {\it relative connections}.

Let $\pi:\Ftemp\to M$ be a vector bundle with zero section $\textbf{0} (x)= (x,0)$, fibrewise addition $a(e,f)=e+f$ and multiplication by a scalar $m_\la (e) = \la e$, for $\la \in \RR$. Its tangent vector bundle is the vector bundle $T\Ftemp$ over $TM$ defined as follows: the fibrewise projection is the differential $d \pi: T\Ftemp \arr TM$, the zero section is $d \textbf{0}$, the fibrewise addition is given by the differential $da: T\Ftemp \times_{TM} T\Ftemp \arr T\Ftemp$ and the fibrewise multiplication by $\la \in \RR$ is given by the differential $dm_\la: TE \arr TE$.
\begin{itemize}
\item  A differential form $\theta\in\Omega^1(\Ftemp,\pi^*\Etemp)$ with values in the (pullback of the) {\it coefficient bundle} $\Etemp\to M$ is called {\bf linear} if $a^*\theta=\mathrm{pr}_1^*\theta+\mathrm{pr}_2^*\theta$, where $\mathrm{pr}_1,\mathrm{pr}_2:\Ftemp\times _M\Ftemp\to \Ftemp$ denote the canonical projections
\item A distribution $H\subset T\Ftemp$ is called {\bf linear} if it is a vector subbundle of $T\Ftemp$ over the same base $TM$.
\end{itemize}

\begin{lemma}\label{rmk:linear}
Let $H$ be a linear distribution on a vector bundle $\Ftemp \arr M$. Then the distribution $H \cap T^\pi \Ftemp$ satisfies
\begin{equation}\label{eq:pullback_symbol}
H \cap T^\pi \Ftemp \cong \pi^*( (H \cap T^\pi \Ftemp) \mid_M).
\end{equation}
Similarly, the normal bundle $TE/H$ can be recovered from the $\pi$-pullback of the vector bundle
\begin{equation*}\label{eq:coefficient_bundle}
\Etemp_H:= (T\Ftemp/H) \mid_M \arr M.
\end{equation*}
Moreover, $H$ is $\pi$-transversal.
\end{lemma}

\begin{proof}
First, we notice that we can right translate vectors tangent to the fibres to the zero section. Indeed, any vector $V$ at $e\in \Ftemp$ tangent to the fibre $\Ftemp_x$, $x=\pi(e)$, moves to a vector based at $\textbf{0}(x)=(x,0)$ by taking the differential of right translation $a_{e}(\cdot):=a(\cdot,-e)$ by $-e$:
\begin{equation}\label{eq:translation}
da_e:T_e(\Ftemp_x)\to T_x(\Ftemp_x),\ V\mapsto da(V,0_{-e}).
\end{equation}
The advantage of this is that $da_e$ takes $\g(H)_e$ to $\g(H)_x$ because $H$ is linear, hence we get  \eqref{eq:pullback_symbol}.

Second, as $H$ is linear, $TM=d\textbf{0}(TM)\subset H|_M$ and this shows that $H$ is $\pi$-transversal on $M$. This, together with the identification \eqref{eq:pullback_symbol}, implies the $\pi$-transversality of $H$:
\begin{equation}\label{eq:transversal}
T\Ftemp=H+T^\pi \Ftemp.
\end{equation} 
Indeed, it is enough to compute $\rk(H_e+T^\pi_e \Ftemp)=\rk(H_e)+\rk(T^\pi_e \Ftemp)-\rk(\g(H)_e)$ and compare it with the ranks at $x=\pi(e)$.

Condition \eqref{eq:transversal} implies in turn that the normal bundle can be rewritten as
$$T\Ftemp/H=T^\pi \Ftemp/ (H \cap T^\pi \Ftemp).$$ 
Using \eqref{eq:translation} and \eqref{eq:pullback_symbol}, and passing again to the normal bundle, we obtain the isomorphism 
\begin{equation*}
\pi^* \Etemp_H\cong T\Ftemp/H. \qedhere
\end{equation*}
\end{proof}

\begin{proposition}\label{rmk:equivalence}{\bf (Equivalence between linear forms and distributions)}
Any pointwise surjective linear form $\te \in \Om^1 (E, \pi^* F)$ induces a linear distribution $H_\te := \ker(\te) \subset T\Ftemp$.

Conversely, any linear distribution $H$ on $\Ftemp$ arises as $\ker(\te_H)$, for $\te_H \in \Om^1 (\Ftemp, \pi^* \Etemp_H)$ the linear form defined by the canonical projection $T\Ftemp\to T\Ftemp/H$ followed by the isomorphism $T E/H \cong \pi^*\Etemp_H$ of Lemma \ref{rmk:linear}. 
\end{proposition}

Analogously to Proposition \ref{correspondence_pfaffian_forms_distribution}, the result above defines a 1-1 correspondence
$$ \left\{   \begin{array}{c}
    \text{Linear distributions}  \\
    H\subset T\Ftemp
    \end{array} \right\} 
\tilde{\longleftrightarrow}
\left\{   \begin{array}{c}
    \text{(equivalence classes) of pointwise surjective linear forms}  \\
    \theta\in\Omega^1(E, \pi^*\Etemp)
    \end{array} \right\}. $$

\begin{proof}
It is immediate to see that $H_\te$ is linear. 
Conversely, let us prove that $\te_H$ is linear (we omit the subscript on $H$ for simplicity). 
Due to the transversality of $H$ one writes $\theta_e(V)=\theta_e(V-\bar V)$, with $\bar V\in H_e = \ker (\te_e)$ any vector such that $d\pi(V)=d\pi(\bar V)$. Hence, for any other vectors $W\in T_f\Ftemp$ with $d\pi(V)=d\pi(W)$, and $\bar W\in H_f$ with $d\pi(W)=d\pi(\bar W)$, we have
\begin{align*}
\theta_e(V)+\theta_f(W)&= \te_0 (da(da(V-\bar V,0_{-e}),da(W-\bar W,0_{-f})) ) \\
&=\te_0 (da(da(V-\bar V,W-\bar W),da(0_{-e},0_{-f})) )= \te_0 ( da(da(V,W)-da(\bar V,\bar W),0_{-e-f}) )\\
&=\theta_{e+f}(da(V,W)-da(\bar V,\bar W))=\theta_{e+f}(da(V,W)),
\end{align*} 
where in the last line we used that $da$ takes $H_e\times_{TM}H_f$ to $H_{e+f}$ by linearity of $H$. 
\end{proof}

Proposition \ref{rmk:equivalence} implies that the following definition is well given:

\begin{definition}
A {\bf linear Pfaffian fibration} is a vector bundle $\pi: \Ftemp\to M$, together with either a pointwise surjective linear form $\theta$ or a linear distribution $H\subset T\Ftemp$.
\end{definition}

\begin{proposition}
 If $(\Ftemp, \te)$ is a linear Pfaffian fibration, then it is a Pfaffian fibration in the sense of Definition \ref{def_Pfaffian_bundle}. Analogously for a linear Pfaffian fibration $(\Ftemp,H)$.
\end{proposition}

\begin{proof}
We say that a vertical vector field $X \in \Ga (T^\pi E) \subset \mathfrak{X}(E)$ is {\it constant along the fibres of $\pi$} if, for every $x \in M$, the vector $d a_e (X) \in T_x (E_x)$ (see equation \eqref{eq:translation}) does not depend on $e \in E_x$. It can be easily seen that such vertical vector fields constant along the fibre of $\pi$ commute.

Moreover, given a linear distribution $H$ on $\pi$, we can write any vector field tangent to $\g(H) \subset \Ga (T^\pi E)$ as a $\mathcal{C}^{\infty}(E)$-linear combination of vector fields tangent to $\g(H)$ and constant along the fibres; it follows that $\g(H)$ is Frobenius-involutive. Together with Remark \ref{rmk:linear}, this concludes the proof. Using Proposition \ref{rmk:equivalence}, the same holds for a linear Pfaffian fibration $(\Ftemp,H)$.
\end{proof}

As promised, we explain now that linear forms and linear distributions can be encoded by a generalised version of linear connections, called {\it relative connections}. Starting from the well-known correspondence between linear connections $\na$ on $E \arr M$ and distributions $H \subset E$ which are horizontal and linear, relative connections will turn out to be in correspondence with distributions which are linear, but not necessarily horizontal.

\begin{definition}\label{def_relative_connection}
Let $E$ and $F$ be two vector bundles over $M$; a {\bf connection} on $E$, {\bf relative} to a surjective vector bundle map $\sigma:\Ftemp\to \Etemp$, is an $\RR$-linear map   
$$D:\Gamma(\Ftemp)\to \Omega^1(M,\Etemp),$$
satisfying, for any section $s \in \Ga(E)$ and function $f\in C^\infty(M)$, the Leibniz-type identity
\begin{equation}\label{eq:Leibniz}
D(fs)(X)=fD(s)(X)+L_X(f)\sigma(s)\quad \forall X\in\X(M).
\end{equation}  
We also say that $(D,\sigma)$ is a {\bf relative connection} and $\si$ is its {\bf symbol map}.  
\end{definition}

In particular, any linear form $\theta\in\Omega^1(M,\Etemp)$ is fully encoded by the operator 
\begin{equation}\label{eq:relative_connection}
D:\Gamma(\Ftemp)\to \Omega^1(M,\Etemp),\quad s\mapsto s^*\theta.
\end{equation}
together with the vector bundle map $\sigma:\Ftemp\to \Etemp,\ \sigma(v)=\theta(v).$
Indeed, we have the following:

\begin{proposition}\label{linear_Pfaffian_bundles_with_connections}
The above procedure induces a 1-1 correspondence between pointwise surjective linear 1-forms on a vector bundle $\pi:\Ftemp \arr M$ and relative connections on $\pi$. 
\end{proposition}

\begin{proof}
The linearity of $\theta$ is translated into the fact that $D$ as in \eqref{eq:relative_connection} is $\RR$-linear and satisfies the Leibniz-type identity \eqref{eq:Leibniz}, where $\sigma:\Ftemp\to \Etemp$ is the vector bundle map over $M$ defined by
\begin{equation*}\sigma_x(u)=\theta_f(u)\end{equation*}
under the canonical identification $T^\pi _f\Ftemp=\Ftemp_x$, for $f \in \Ftemp$, $x = \pi(f) \in M$. Conversely, if $D$ is a connection relative to $\sigma$, there is a well defined linear form $\te \in \Om^1 (E, \pi^* F)$ uniquely determined by $s^*\theta=D(s)$ (for any $s\in \Gamma(\Ftemp)$) and $\theta(v)=\sigma(v)$ (for any $v\in \Ftemp=T^\pi \Ftemp|_M$).
\end{proof}

When there is no confusion, we denote a {\bf linear Pfaffian fibration} by $(\Ftemp,D)$.
Of course, all definitions and properties can be translated from the point of view of linear forms to the one of relative connections and viceversa.  Accordingly, we call $$\g(D):=\ker(\sigma)$$
the {\bf symbol space of $D$}, we say that a section $s$ is holonomic if $D(s)=0$, and we denote by $\Gamma(\Ftemp,D)$ the set of {\bf holonomic sections}. As in the case of linear distributions, the linearity of the form $\theta$ associated to $D$ implies that the natural identification between $T^\pi \Ftemp$ and the pullback $\pi^*\Ftemp$ restricts to the symbol spaces: 
\begin{equation*}\label{eq:symbol_spaces}
\g(\theta)\cong \pi^*\g(D).
\end{equation*}

\begin{observation}{\bf (Relative connections induced by linear distributions)}\label{rmk:operator_distribution}
We describe directly the correspondence between linear distributions and relative connections, bypassing Proposition \ref{linear_Pfaffian_bundles_with_connections} and Remark \ref{rmk:linear}. As we anticipated, this can be also thought as a generalisation of the well-known correspondence between linear connections $\nabla:\X(M)\times\Gamma(\Ftemp)\to \Gamma(\Ftemp)$, and transversal linear distributions, given by the horizontal distribution of $\nabla$. 

For any linear distribution $H$ on $\Ftemp$, one produces a connection 
$$D:\Gamma(\Ftemp)\to \Omega^1(M,\Ftemp/\g),$$
relative to the projection $\mathrm{pr}:\Ftemp\to \Ftemp/\g$, for $\g\subset \Ftemp$ the subbundle defined by
$$\g:=\g(H)|_{M}\subset T^\pi \Ftemp|_M\cong \Ftemp,$$ 
where we are identifying canonically $T^\pi \Ftemp$ with $\pi^*\Ftemp$. The connection $D$ is given by the formula
\begin{equation*}D_X(s)(x):=[\bar s,\bar X](x)\mod H\end{equation*}
where $X \in \X(M)$, $\bar X\in \X(\Ftemp)$ is any $\pi$-projectable extension of $X$, tangent to $H$, and $\bar s$ is the vertical vector field constant along the fibres induced by $s$. Of course, the above formula coincides with \eqref{eq:relative_connection} when $\theta_H$ is the canonical projection $T\Ftemp\to \pi^*\Etemp_H$. More generally, for any linear form $\te$, one can write the associate relative connection \eqref{eq:relative_connection} as 
$$D_X(s)(x)=\theta([\bar s, \bar X]_x).$$
To check this formula one uses the flow of $\bar s$ to compute the bracket, and the linearity of $\theta$. This equation will play a role in the theory of prolongations of a linear Pfaffian fibration.
\end{observation}

\begin{observation}[\bf Relative connections as Spencer operators]\label{rmk:Spencer_operators} 
Any vector bundle $\Ftemp$ can be thought as a Lie algebroid with zero bracket and zero anchor. The appropriate generalisation of relative connections in the world of algebroids is the notion of Spencer operators: these are relative connections compatible with the Lie bracket and the anchor; they play the infinitesimal counterpart of multiplicative distributions (see \cite{Cra12}). These compatibility conditions are trivially satisfied when the Lie algebroid is a vector bundles, so in this case the notions of Spencer operator and relative connection coincide.
\end{observation}

\subsection{Linearisation of Pfaffian fibrations along holonomic sections}\label{sec:linearisation}

In this section we discuss a natural process of linearisation in the context of Pfaffian fibrations, which can be sketched as the following map:
\begin{center}
Pfaffian fibrations and holonomic sections $\stackrel{\Lin}{\Longrightarrow}$ linear Pfaffian fibrations
\end{center}
$$ ((P,\te), \be ) \mapsto (\mathrm{Lin}_{\beta}(P, \theta), D^\beta).$$
Let us describe this application $\Lin$. 


%

\begin{definition}\label{def_linearisation_pfaffian_bundle}
 Let $(P,\te)$ be a Pfaffian fibration over $M$ and $\beta \in \Ga(P,\te)$ a holonomic section, i.e.\ $\beta^*\theta=0$. The {\bf linearisation of $(P,\te)$ along $\beta$} is the pair
 $$(\mathrm{Lin}_{\beta}(P, \theta), D^\beta),$$
 where $\mathrm{Lin}_{\beta}(P, \theta)$ is the vector bundle over $M$
\begin{equation*} \mathrm{Lin}_{\beta}(P, \theta):= \beta^* T^\pi P, \end{equation*}
 and $D^\be$ is the operator
 $$D^\be: \Ga (\Lin_\be (P,\te) ) \arr \Om^1 (M, \be^*\mathcal{N})$$
 defined as follows. For any section $s\in \Gamma(\beta^* T^\pi P)$, choose a smooth family $\beta_t$ of sections of $P$ such that
$$\beta_0= \beta, \quad \left.\frac{d}{dt}\right|_{t= 0} \beta_t(x)= s(x).$$
 For $X_x\in T_xM$, the family $\beta_t^*(\theta)(X_x)\in \mathcal{N}_{\beta_{t}(x)}$ defines a curve starting at $0_{\beta(x)}$. Accordingly, its speed is a vector in $T_{0_{\beta(x)}}\mathcal{N}\cong T_{\beta(x)}P\oplus\mathcal{N}_{\beta(x)}$. We define
 $$ D^{\beta}_{X}(s)(x) := \textrm{pr}_{\mathcal{N}_{\beta(x)}}\left(\left.\frac{d}{dt}\right|_{t= 0} (\beta_t)^*(\theta)(X_x)\right) \in \mathcal{N}_{\beta(x)},$$
\end{definition}

It is straightforward to check that the operator $D^\be$ defined above is a connection on $\mathrm{Lin}_{\beta}(P, \theta)$ relative to $\si = \te \mid_{T^\pi P}$ (Definition \ref{def_relative_connection}), hence $(\mathrm{Lin}_{\beta}(P, \theta), D^\beta)$ is a linear Pfaffian fibration. Moreover, its symbol space coincides with the pull-back via $\beta$ of the symbol space $\g$ of $(P, \theta)$:
$$ \mathfrak{g}(\textrm{Lin}_{\beta}(P, \theta))= \beta^* \mathfrak{g}.$$


\begin{observation}[\bf Linearisation of a linear Pfaffian fibration]\label{linearisation_of_linear_pfaffian_bundle}
When a Pfaffian fibration is already linear, linearising along the zero section becomes the identity, i.e.\ $\mathrm{Lin}_{\textbf{0}} (\bullet) = \bullet$ (of course, the zero section \textbf{0} is always holonomic for any linear form $\theta$).

Indeed, the linearisation of $(\Ftemp,\theta)$ along $\textbf{0}$ recovers the vector bundle $\Ftemp=\Ftemp^\textbf{0}=\textbf{0}^*(T^\pi \Ftemp)$ and the relative connection $D$ associated to $\theta$ as in \eqref{eq:relative_connection}. To check this, note that a section $s$ of $\Ftemp$ can by written as 
$$s= \left. \frac{d}{d\epsilon} \right|_{\ep= 0} (\textbf{0}+\epsilon s),$$
hence
$$D^\textbf{0}(s)= \left. \frac{d}{d\epsilon} \right|_{\ep= 0} (\textbf{0}+\epsilon s)^*(\theta)= \left. \frac{d}{d\epsilon} \right|_{\ep= 0} \epsilon(s^*(\theta))=s^*(\theta)=D(s),$$
where in the second equality we used again the linearity of $\theta$ to write $(\textbf{0}+\epsilon s)^*(\theta)=\textbf{0}^*\theta+\epsilon(s^*\theta)=\epsilon(s^*\theta)$. As $\theta$ and $D$ encode the same Pfaffian fibration (see Remark \ref{rmk:Spencer_operators}), we see that linearising a linear Pfaffian fibration along the zero section does not do anything; we end up recovering the same linear Pfaffian fibration.
\end{observation}

\begin{observation}[\bf Linearisation of a Pfaffian groupoid]\label{obs_pfaffian_groupoids}
 Intuitively, a Pfaffian groupoid is a Pfaffian fibration together with a multiplicative (group-like) structure; such multiplicativity translates into a richer geometrical content and simpler objects. Passing to the infinitesimal counterpart, we found Lie algebroids endowed with Spencer operators (see Remark \ref{rmk:Spencer_operators}): the linearisation of a Pfaffian groupoid along its unit map coincides precisely with the Spencer operator associated to a multiplicative form as in \cite{Cra12}. 
\end{observation}

%
%
%

\begin{observation}[\bf Heuristics of the linearisation procedure]
In this remark we aim to give an intuitive explanation of the linearisation phenomenon, for which we will use an infinite-dimensional picture in a heuristic way, without providing precise details.

Let $(P,\te)$ be a Pfaffian fibration over $M$, with $\te \in \Om^1 (P,\mathcal{N})$, and consider the (infinite-dimensional) vector bundle $\mathcal{\Etemp}$ over the (infinite-dimensional) manifold $\mathcal{P}:= \Gamma(P)$ by setting the fibres
$$\mathcal{\Etemp}_\beta:=\Omega^1(M,\beta^*\mathcal{N}), \quad \be \in \mathcal{P}$$
and consider its global section
\begin{equation*}\label{eq:holonomic}
\Theta:\mathcal{P}\to \mathcal{F}, \quad \beta\mapsto \beta^*\theta.
\end{equation*}



The holonomic sections of $(P,\te)$ are now the zeroes of $\Te$, hence $\Te$ can be called {\it holonomator}. The linearisation of $(P,\te)$ around a holonomic section $\be \in \mathcal{P}$ becomes then the usual linearisation of the section $\Theta$ at the zero $\beta$, i.e.\ the $\mathcal{F}_\beta$-component of the differential  
$$ d_\beta \Theta: T_\beta\mathcal{P}\rightarrow T_0\mathcal{F}=T_\beta\mathcal{P}\oplus\mathcal{\Etemp}_{\beta}.$$
Since a vector tangent to $\mathcal{P}$ at $\be$ is realised as the velocity of a path $t\mapsto\be_t \in \mathcal{P}$ starting at $\be$, i.e.\ $T_\beta\mathcal{P}= \Gamma(\beta^* T^\pi P)$, then the linearisation becomes an operator
$$D^{\beta}:= d_\beta \Theta: \Gamma(\beta^* T^\pi P)\rightarrow \Omega^1(M, \beta^*\mathcal{N}).$$
Together with $\si^{\beta}$ given by $\theta$ restricted to $T^\pi P$, we obtain a relative connection $(D^{\beta}, \si^{\beta})$ on $\beta^* T^\pi P$ with coefficients in $\beta^*\Etemp$.
This is precisely the linearisation of $(P, \theta)$ along $\beta$ from Definition \ref{def_linearisation_pfaffian_bundle}.
\end{observation}

\subsection{Examples}


\begin{example}[\bf PDEs]
As we anticipated, jet bundles and PDEs are the prototypical examples of Pfaffian fibrations.

\begin{proposition}\label{PDE_is_Pfaffian_bundlle}
Let $R \arr M$ be a fibration; any PDE $P \subset J^k R$, together with the restriction of the Cartan form $\te_{\mathrm{can}}$, is a Pfaffian fibration on $M$. Moreover, its symbol space (Definition \ref{def_Pfaffian_bundle}) coincides with the symbol space of $P$ as a PDE (equation \eqref{eq:symbol_pde}).
\end{proposition}

\begin{proof}
 By the regularity conditions asked on $P$ (see the discussion after equation \eqref{def_PDE_jet_bundle}), the projection $\pi: P \arr \pi(P) \subset M$ is a surjective submersion. Moreover, since also $\mathrm{pr}: P \arr \mathrm{pr}(P)$ is a submersion, we can choose a splitting $\xi: T \mathrm{pr}(P) \arr \mathrm{pr}^*TP$ of $d \mathrm{pr}$. It follows that, for every $p = j^k_x \phi \in P$, we can consider the map
 $$ T_x M \arr \ker(\te_p) \subset T_p P, \quad v \mapsto \xi (d_x (j^{k-1} \si) (v)), $$
which is a splitting of $d_p \pi \mid_{\ker(\te_p)}: \ker(\te_p) \arr d_p \pi (P)$; this proves that $\te$ is $\pi$-transversal.
 
 Moreover, one notices that the Cartan form $\te_{\mathrm{can}}$ restricted to $\ker(d\pi)$ is simply the differential of the projection $\mathrm{pr}: P \arr \mathrm{pr}(P) \subset J^{k-1} P$, hence
\begin{equation}\label{symbol_space_in_proof}
 \ker(\te_{\mathrm{can}}) \cap \ker(d\pi) = \ker (d \mathrm{pr}: T P \arr T \mathrm{pr} (P) ).
\end{equation}
 Since, by definition of PDE, we assume that $\mathrm{pr}: P \arr \mathrm{pr}(P)$ is a submersion, its kernel is a smooth submanifold and $\ker(\te_{\mathrm{can}}) \cap \ker(d\pi)$ is an involutive regular distribution on $P$, i.e.\ $\te_{\mathrm{can}}$ is $\pi$-involutive.
 
 We conclude that $(P,\te) \arr \pi(P)$ is a Pfaffian fibration. In particular, by equation \eqref{symbol_space_in_proof}, the symbol space of $(P,\te)$ as a Pfaffian fibration coincides with the symbol space of $P$ as a PDE.
\end{proof}


Here is a partial converse of the previous result; any Pfaffian fibration which is ``nice enough" can be realised from a jet bundle.

\begin{proposition}\label{immersion_theorem_pfaffian_bundle}
Let $\pi: (P,\te) \arr M$ be a Pfaffian fibration, with $\te \in \Om^1 (P, \mathcal{N})$, and assume that the foliation on $P$ defined by the symbol space is simple, i.e.\ $\g(\te) = \ker(df)$ for some fibration $f: P \arr Q$.
Then there exist
\begin{itemize}
\item a fibration $\tau: Q \arr M$ such that $\tau \circ f = \pi$
\item a vector bundle isomorphism $\Phi: f^* (T^\tau Q) \arr \mathcal{N}$ 
\item a unique bundle map $(P, \te) \xrightarrow{i} (J^1 Q, \te_{\mathrm{can}})$ such that
$$\Phi \circ i^* \te_{\mathrm{can}} = \te,$$
for $\te_{\mathrm{can}} \in \Om^1 (J^1 Q, \mathrm{pr}^*T^\tau Q )$ the canonical Cartan form on $J^1 Q$.
\end{itemize}
\end{proposition}
\begin{proof}
We define $Q:= P/\mathord\sim$ as the leaf space of the foliation $\g(\te)$. Then the projection
$$ \tau: Q \arr M, \quad [p]=f(p) \mapsto \pi(p)$$
is well defined, since $d \pi$ vanishes on $\g(\te)$, hence $\pi$ is constant on each leaf. Moreover, $\tau$ is a fibration since $\pi$ is so.

The linear isomorphism $\Phi_p: T_{f(p)}^\tau Q \arr \mathcal{N}_p$ is defined as the composition of the inverse of the isomorphism
$$ d_p f: T^\pi_p P/\g(\te_p) \arr T^\pi_{f(p)} Q, \quad [v] \arr df(v)$$
with the isomorphism
$$ \te_p: T^\pi_p P/\g(\te_p) = T_p P /\ker(\te_p) \arr \mathcal{N}_p, \quad [v] \arr \te_p (v).$$

The bundle map $i$ is defined as
$$ i(p) := (f(p), \xi_{[p]} )$$
where we interpret $J^1 Q$ as in equation \eqref{alternative_description_first_jet}. Here $\xi_{[p]}$ is defined as the composition of the isomorphisms $T_{\pi(p)} M \cong \ker(\te_p)/\g(\te_p) \subset T_p P /\g (\te_p)$ and $T_p P/ \g(\te_p) \cong T_{[p]} Q$, i.e.\
$$ \xi_{[p]}: T_{\pi(p)} M \arr T_p P, \quad v \mapsto d_p f (\bar{v}), $$
where $\bar{v}$ is any vector in $\ker(\te_p)$ such that $d_p \pi(\bar{v}) = v$.

To prove that $\Phi \circ i^*\te_{\mathrm{can}} = \te$, we compute, for every $v \in T_p P$,
$$ \Phi \circ (i^*\te_{\mathrm{can}})_p (v) = \Phi \circ (\te_{\mathrm{can}})_{i(p)} (d_p i (v)) = \Phi \circ ( d (\mathrm{pr} \circ i) (v) - \xi_{[p]} (d (\pi \circ i) (v) ) ) = $$
$$ = \Phi \circ ( d_p f (v) - \xi_{[p]} (d_p \pi (v) ) ) = \Phi \circ  (d_pf (v) - d_p f (\bar{v}) ) = \Phi (d_p f (v - \bar{v})) = \te_p (v - \bar{v}) = \te_p (v).$$

Last, for the uniqueness of $i$, assume there is another bundle map $j: P \arr J^1 Q, \quad p \mapsto (f(p), \widetilde{\xi}_{[p]})$ with the same properties; then, for every $v \in TP$,
$$ (\te_{\mathrm{can}})_{i(p)} (di (v) ) = (\te_{\mathrm{can}})_{j(p)} (dj (v)).$$
The previous computations tells us that
$$ d_p f (v - \bar{v}) = d_p f (v) - \widetilde{\xi}_{[p]} (d_p \pi (v) ),$$
which implies that $\xi = \widetilde{\xi}$, i.e.\ that $j$ must coincide with $i$.
\end{proof}

Proposition \ref{immersion_theorem_pfaffian_bundle} will be improved in the next section (see Corollary \ref{corollary_immersion_theorem}). 
\end{example}

\begin{example}[\bf Linear PDEs]\label{example_linear_pfaffian_bundle}
Let $E \arr M$ be a vector bundle; any linear PDE $F \subset J^kE$, together with the restriction of the Cartan form $\te_{\mathrm{can}}$, is a linear Pfaffian fibration on $M$. Indeed, a simple computation shows the linearity of $\te_{\mathrm{can}}$. 

Note that the coefficient bundle of $\te_{\mathrm{can}}$ is $J^{k-1}E$ because we have the canonical identification $\mathrm{pr}^*T^\pi(J^{k-1}E)\cong\pi^*J^{k-1}E$, with $\mathrm{pr}$ the projection $J^kE\to J^{k-1}E$. This explains also why the Cartan form and the classical Spencer operator play the same role in the theory of linear PDEs. More precisely, the classical Spencer operator $D^{\mathrm{clas}}:\Gamma(J^kE)\to \Omega^1(M,J^{k-1}E)$ is just the connection relative to the projection $J^k E \arr J^{k-1} E$ and defined by equation \eqref{eq:relative_connection} via $\te_{\mathrm{can}}$:
$$D(s)=s^*\te_{\mathrm{can}}.$$
In other words, the Cartan form on a linear jet space is fully encoded by the classical Spencer operator (see also sections \ref{sec:Cartan_form} and \ref{sec:Spencer_operator}).

Note also that, applying Remark \ref{linearisation_of_linear_pfaffian_bundle}, the linearisation of the Cartan form on a linear jet bundle $J^k E$ is precisely the classical Spencer operator of $J^k E \arr M$.
\end{example}

\begin{example}[\bf Lie Pseudogroups]\label{pseudogroups_examples}
 An important source of examples of Pfaffian fibrations comes from Lie pseudogroups. Recall from \cite{Yud16} that a pseudogroup on a manifold $X$ is a set $\Ga \subset \Diff_{\mathrm{loc}}(X)$ of diffeomorphisms between opens of $X$, which is closed under composition, inversion, restriction and glueing. A Lie pseudogroup is a pseudogroup $\Ga$ satisfying further regularity conditions, namely the subspace
 $$ J^k \Ga := \{ j^k_x \phi \mid \phi \in \Ga, x \in \dom(\phi) \} \subset J^k (X, X):= J^k (\mathrm{pr}_1: X \times X \arr X)$$
must be a smooth submanifold for every $k$.
 
 In particular, $J^k \Ga$ is endowed with the restriction of the Cartan form $\te_{\mathrm{can}}$ of $J^k (X,X)$, denoted by $\te$, as well as with two fibrations:
 $$s: J^k \Ga \arr X, \quad j^k_x \phi \mapsto x,$$
 $$t: J^k \Ga \arr X, \quad j^k_x \phi \mapsto \phi(x).$$
 We claim that $(J^k \Ga, \te)$ is a Pfaffian fibration w.r.t.\ both fibrations.
 
 Indeed, $s: J^k \Ga \arr X$ is a PDE on the fibration $X \times X \arr X$, hence is a Pfaffian fibration by Proposition \ref{PDE_is_Pfaffian_bundlle}. On the other hand, it is easy to check that the two maps $s$ and $t$ are related to the Cartan form $\te$ by the folllowing equation:
\begin{equation}\label{lie_pfaffian_equation}
 \ker(\te) \cap \ker(ds) = \ker(\te) \cap \ker(dt)
\end{equation} 
The fact that $\te$ is $t$-transversal follows then by a dimensional argument: for every $g \in J^k \Ga$,
$$ \dim(T_g J^k \Ga) = \dim (\ker(d_gs)) + \dim(\te_g) - \dim (\ker(d_gs) \cap \ker (\te_g) ) = $$
$$ = \dim (\ker (d_gt)) + \dim (\te_g) - \dim (\ker (d_gt) \cap \ker (\te_g) ).$$
Moreover, since $\te$ is $s$-involutive and \eqref{lie_pfaffian_equation} holds, $\te$ is also $t$-involutive, hence $t: (J^k \Ga, \te) \arr X$ is a Pfaffian fibration as well.

Here is an important property of Pfaffian fibrations of the kind $J^k \Ga$: they are all PDE-integrable (Definition \ref{def_integrability}). Indeed, for every $g = j^k_x \phi \in J^k \Ga$, there exists the local section $j^k \phi \in \Ga_{\mathrm{loc}} (s)$, which is holonomic by construction and sends $x$ to $g$; similarly, the local section $j^k \phi \circ \phi^{-1} \in \Ga_{\mathrm{loc}} (t)$ is holonomic and sends $\phi(x)$ to $(j^k \phi) (\phi^{-1}(\phi(x))) = (j^k \phi) (x) = g$.

Last, we remark that equation \eqref{lie_pfaffian_equation} establishes a compatibility between the two structures of Pfaffian fibrations on $J^k \Ga$. This becomes more meaningful if we realise that $J^k \Ga$ possesses a Lie groupoid structure compatible with $\te$ in an appropriate sense, i.e.\ $J^k \Ga$ is an example of Pfaffian groupoid (see Remark \ref{obs_pfaffian_groupoids}). The fact that $\ker(\te) \cap \ker(ds) = \ker(\te) \cap \ker(dt)$ says that the Pfaffian groupoid $(J^k \Ga, \te)$ is of a special kind, called of {\it Lie type}; we will however not discuss here the consequence of this property, for which we refer to  \cites{Sal13, Cra20}.
\end{example}

\begin{example}[\bf $G$-structures]\label{example_G_structure}
Many geometric structures defines a Pfaffian fibration: this happens with Riemannian metrics, almost symplectic structures, almost complex structures, etc. More precisely, let $P \subset Fr(M)$ be any $G$-structure on $M^n$, i.e.\ $P$ is a reduction of the structure group of $Fr(M)$ to a Lie subgroup $G \subset GL(n,\RR)$; then $P$ defines a Pfaffian fibration over $M$ as follows. Consider
$$\smash{\widetilde{P}}:= \{ (x,y,\xi) \mid x,y \in M, \xi: T_x M \arr T_y M \text{ linear isomorphism preserving frames in } P \}, $$
and the projections $\pi_1$ and $\pi_2$ on the first and second component. Then $\pi_1: (\smash{\widetilde{P}}, \om) \arr M$ is Pfaffian fibration, where the form $\om \in \Om^1 (\smash{\widetilde{P}}, \pi_2^* TM )$ is defined by
$$ \om_{(x,y,\xi)} (v) := d\pi_2 (v) - \xi (d\pi_1 (v) ).$$
This follows easily by realising $\smash{\widetilde{P}}$ as a subbundle of $J^1 (M,M): = J^1 (\mathrm{pr}_1: M \times M \arr M)$ via equation \eqref{alternative_description_first_jet}, and noticing that $\om$ is the restriction of the Cartan form of $J^1 (M,M)$. Of course, swapping $\pi_1$ and $\pi_2$ and replacing $\xi$ with $\xi^{-1}$ would yield another form $\om' \in \Om^1 (\smash{\widetilde{P}}, \pi_1^* TM)$ which makes $\pi_2: (\smash{\widetilde{P}}, \om') \arr M$ a Pfaffian fibration.

Here is an interesting application: the PDE-integrability of $\smash{\widetilde{P}}$ as a Pfaffian fibration is a necessary condition for the integrability of $P$ as a $G$-structure (e.g.\ the flatness of a Riemannian metric, the closedness of an almost symplectic form, etc.). Recall that a $G$-structure $P$ is integrable if it admits an atlas of charts  ``adapted" to $P$, meaning that their induced diffeomorphisms between opens of $M$ preserve the frames of $P$. In particular, using such an atlas, for every $(x,y,\xi) \in \smash{\widetilde{P}}$ one finds adapted charts $\chi_x: U \arr \RR^n$ around $x$ and $\chi_y: V \arr \RR^n$ around $y$ such that $f:= (\chi_y)^{-1} \circ \chi_x$ is a local diffeomorphism of $M$, sending $x$ to $y$ and such that $d_x f = \xi$.

On the other hand, a section of $\smash{\widetilde{P}}$ is a function $\si: U \arr \smash{\widetilde{P}}$ of the type $\si(x) = (x, f(x), \xi_x)$, for $f: U \arr V$ some smooth map (not necessarily a diffeomorphism). By the definition of $\om$, the section $\si$ is holonomic precisely if and only if $\xi_x = d_x f$. It follows that, if $P$ is integrable, for every $(x,y, \xi) \in \smash{\widetilde{P}}$ there is a holonomic section through it, i.e.\ $\smash{\widetilde{P}}$ is PDE-integrable.

As for Example \ref{pseudogroups_examples}, one can also notice that $\smash{\widetilde{P}}$ has a structure of Lie groupoid; this is more transparent by establishing the isomorphism $\smash{\widetilde{P}} \cong (P \times P)/G$, where $P \times P$ is quotiented by the diagonal action of $G$ (this is also known as the {\it gauge groupoid} of the principal bundle $P$). Then $\smash{\widetilde{P}}$ is also a Pfaffian groupoid (see Remark \ref{obs_pfaffian_groupoids}), which is of Lie type since it clearly satisfies $\ker(\om) \cap \ker(d \pi_1) = \ker(\om) \cap \ker(d \pi_2)$.
\end{example}


\section{Prolongations}

The purpose of this section is to understand geometrically and intrinsically the notion of prolongation of a Pfaffian fibration and its fundamental properties. We start by exploring the type of morphisms between Pfaffian fibrations which induce maps on the set of holonomic sections, and then move forward to study morphisms with more specific requirements. These extra conditions extract, in a sense, all the fundamental properties of the prolongations of a PDE (see section \ref{sec:PDE_prolongation}), in the same way that the conditions of a Pfaffian fibration extract the fundamental properties of the solutions of a PDE.

\subsection{Morphisms of Pfaffian fibrations}

Given two Pfaffian fibrations over the same manifolds, the most natural notion of morphism between them consists of a bundle map preserving the two Pfaffian forms.

\begin{definition}\label{def_pfaffian_morphism}
A {\bf weak Pfaffian morphism} between two Pfaffian fibrations $(P',\theta')$, $(P,\theta)$ over $M$ is a smooth fibre bundle map $\phi:P'\to P$ with the property that 
\begin{equation}\label{eq:forms_commute}\phi^*\theta=\Phi\circ \theta'\end{equation}
for some vector bundle map $\Phi:\mathcal{N}'\to \phi^*\mathcal{N}$ between the coefficient bundles. \end{definition} 

Note that, since $\te'$ and $\te$ are surjective, the map $\Phi$ in the previous definition is unique.

\begin{observation}\label{properties_pfaffian_morphism}
It follows immediately from the definition that a weak Pfaffian morphism $\phi$ induces a map on the sections which preserves the holonomic ones:
\begin{equation}\label{eq:holonomic_holonomic}
\phi:\Gamma_{\loc}(P',\theta')\to \Gamma_{\loc}(P,\theta).
\end{equation} 

Moreover, since $\pi' = \pi \circ \phi$, the differential $d\phi$ maps the symbol space $\g(\te')$ to $\g(\theta)$.
\end{observation}

\begin{example}\label{example_pfaffian_morphism}
An example of weak Pfaffian morphism is given by a PDE $P \xrightarrow{i} J^k R$: in this case, the form $\te = i^* \te_{\mathrm{can}}$ on $P$ is just the pullback of the Cartan form $\te_{\mathrm{can}}$ on $J^k R$ by the injection $i$.

Similarly, if a PDE $P \subset J^k R$ is integrable up to order $k+1$ (see section \ref{sec:PDE_prolongation}), the projection
$$\mathrm{pr}:(P^{(1)},\theta^{(1)})\to (P,\theta)$$
is a weak Pfaffian morphism, where $\theta^{(1)}$ is the restriction of the Cartan form of $J^{k+1}R$, and $\theta$ the restriction of the Cartan form of $J^kR$.

Note that, in both cases, $\Phi$ is the identity and the results from Remark \ref{properties_pfaffian_morphism} hold trivially.
\end{example}

\begin{example}\label{example_pfaffian_morphism_2}
Given any Pfaffian fibration $(P,\te)$ whose symbol space satisfies the hypothesis of Proposition \ref{immersion_theorem_pfaffian_bundle}, the induced bundle map
$$ i: (P,\te) \arr (J^1 Q, \te_{\mathrm{can}})$$
is a weak Pfaffian morphism, with $\Phi$ the inverse of the isomorphism between the coefficients.
\end{example}

However, there are a number of reasons to add some constraints to the above definition of weak Pfaffian morphism. First, such notion does not behave well with respect to important objects associated to Pfaffian fibrations, such as curvature or integral elements. Second, given a bundle map $\phi: P' \arr (P,\te)$, we cannot always produce a weak Pfaffian morphism by endowing $P'$ with the form $\phi^* \te$ (as we did in Example \ref{example_pfaffian_morphism} for PDEs), since $\phi^* \te$ might not be $\pi$-involutive, $\pi$-regular, or even pointwise surjective. In conclusion, even if Definition \ref{def_pfaffian_morphism} is very natural, it reveals to be too weak for our further study of prolongations of Pfaffian fibrations; we are therefore going to introduce the following notion.

\begin{definition}\label{def_pfaffian_fibration}
A {\bf Pfaffian morphism} between two Pfaffian fibrations $(P',\theta')$, $(P,\theta)$ over $M$ is a surjective submersion $\phi: P' \arr P$ which is also a weak Pfaffian morphism.
\end{definition}

A Pfaffian morphism satisfies many properties, which we list below for future reference:

\begin{proposition}\label{prop:curvature_commute}
Given a Pfaffian morphism $\phi:(P',\theta')\to (P,\theta)$,
\begin{enumerate}
\item $\phi$ sends holonomic sections of $(P', \te')$ to holonomic sections of $(P,\te)$.
\item $d \phi$ sends the symbol space $\g(\te')$ to the symbol space $\g(\theta)$.
\item If $(P',\te')$ is PDE-integrable, $(P,\te)$ is PDE-integrable.
\item The curvature maps $\kappa_{\te'}$ and $\kappa_{\te}$ are related by the equation
\begin{equation}\label{eq:curvature_commute_2}
\phi^*\kappa_\theta=\Phi\circ \kappa_{\theta'}.
\end{equation}
\item $d \phi$ sends (partial) integral elements of $(P',\theta')$ to (partial) integral elements $(P,\theta)$.
\end{enumerate}
\end{proposition}

\begin{proof} 
The first two properties holds for any weak Pfaffian morphism, as we noticed in Remark \ref{properties_pfaffian_morphism}.

The third property requires the surjectivity of $\phi$. Indeed, under such assumption, consider any $p \in P$; then we can pick a point $p' \in \phi^{-1}(p) \subset P'$, around which there exists a holonomic section $\si'$ of $P'$, and check that $\si := \phi \circ \si'$ is a holonomic section of $P$ around $p$.

For the fourth property, we use equation \eqref{eq:forms_commute} and $\phi$-projectable vector fields to conclude that $\phi^*\kappa_\theta=\kappa_{\Phi\circ\theta'}$. Then we choose two linear connections $\nabla'$ and $\nabla$, respectively on the coefficient bundles $\mathcal{N}'$ and $\mathcal{N}$, and we show that $d_{\phi^*\nabla}(\Phi\circ \theta')=\Phi\circ d_{\nabla'}(\theta')$. Last, we argue that the restrictions of $\kappa_{\Phi\circ\theta'}$ and $\Phi\circ \kappa_{\theta'}$ to $\ker(\theta')$ coincide (see the discussion after equation\ \eqref{eq:curvature}).

For the fifth property, it is enough to use the relations \eqref{eq:forms_commute} and \eqref{eq:curvature_commute_2}, which imply that $d\phi$ preserves (partial) integral elements (Definition \ref{def:integral_elements}).
\end{proof}

\begin{example}[\bf Pullback Pfaffian fibration]
Let $\pi: (P,\te) \arr M$ be a Pfaffian fibration, $\pi': P' \arr M$ a fibration and $\phi: P' \arr P$ a surjective submersive bundle map. Then $P'$ can be endowed with the pullback
$\te' := \phi^* \te$, so that $(P', \te')$ becomes a Pfaffian fibration (the {\it pullback Pfaffian fibration}) and $\phi$ becomes a Pfaffian morphism (where $\Phi$ is just the identity).

In order to prove this claim, as anticipated above, the hypothesis that $\phi$ is a submersion is crucial. One checks immediately that the pullback $\phi^*\theta$ is pointwise surjective. Moreover, $\phi^*\theta$ is $\pi'$-regular: indeed, for every $p \in P'$, the maps $d_p\phi:T_pP'\to T_{\phi(p)}P$ and $d_{\phi(p)} \pi: T_{\phi(p)} P \arr T_{\pi'(p)} M$ are surjective when restricted to $\ker(\phi^*\te)_p$ and $\ker(\te)_{\phi(p)}$, and the diagram
\begin{equation*}
\xymatrix{
\ker (\phi^*\theta)_p \ar[r]^{d\phi} \ar[d]_{d\pi'} & \ker(\theta)_{\phi(p)} \ar[ld]^{d\pi} \\
T_{\pi'(p)}M
}
\end{equation*}
commutes, hence $d_p \pi'$ is surjective as well when restricted to $\ker(\phi^*\te)_p$. Last, to prove the $\pi'$-involutivity of $\phi^*\theta$, consider any two vector fields $X,Y$ tangent to $\g(\phi^*\theta)$; then we have
$$\kappa_{\phi^*\theta}(X,Y)=\kappa_\theta(d\phi(X),d\phi(Y))=0$$
thanks to properties 2 and 4 of Proposition \ref{prop:curvature_commute} and because $\g(\theta)$ is Frobenius-involutive. This says on one hand that the bracket $[X,Y]$ belongs to $\ker(\phi^*\theta)$; on the other hand, since $T^{\pi'}P'$ is Frobenius-involutive, that the bracket $[X,Y]$ is also tangent to $T^{\pi'}P'$, hence to $\g(\phi^*\theta)$, proving that $\phi^*\theta$ is $\pi'$-involutive. 
\end{example}

\begin{example}\label{example_pfaffian_fibration}
A PDE $P \xhookrightarrow{i} J^k R$, which is a weak Pfaffian morphism by Example \ref{example_pfaffian_morphism}, is not a Pfaffian morphism, since $i$ is not a surjective submersion; similarly for the morphism from Example \ref{example_pfaffian_morphism_2}.

On the other hand, given a PDE $P \subset J^k R$ integrable up to order $k+1$, its prolongation $\mathrm{pr}:(P^{(1)},\theta^{(1)})\to (P,\theta)$ is a Pfaffian morphism. In fact, this projection has a richer geometrical structure, which is manifested in the properties of a {\it normalised prolongation} (see Definition \ref{def:prolongation} and Example \ref{example_normalised_prolongation} below).
\end{example}

\begin{observation}{\bf (weak Pfaffian morphisms between Pfaffian distributions)}\label{rmk:morphism_distribution} Paraphrasing this section in the language of Pfaffian distributions $H'\subset TP'$, $H\subset TP$, one obtains the corresponding conditions of weak Pfaffian morphisms only in terms of the distributions, when applied to the associated Pfaffian forms $\theta = \te_H$ and $\theta' = \te_{H'}$. First of all, \eqref{eq:forms_commute} corresponds to 
\begin{equation}\label{eq:distributions_commute}
d\phi(H')\subset H.
\end{equation}
The map $\Phi:TP'/H'\to\phi^* TP/H$ is forced to be $[u]\mapsto [d\phi(u)]$ and it is well defined by \eqref{eq:distributions_commute}; in this case we denote $\Phi$ by $[d\phi]$.
Hence, in this setting, a {\bf weak Pfaffian morphism} is a bundle map $\phi:P'\to P$ satisfying \eqref{eq:distributions_commute}; as in \eqref{eq:holonomic_holonomic}, $\phi$ preserves holonomic sections.

A weak Pfaffian morphism $\phi$ is called a {\bf Pfaffian morphism} when it is also a surjective submersion. Again, such condition will imply an equation on the curvatures analogous to \eqref{eq:curvature_commute_2}:
\begin{equation*}
\phi^*\kappa_H=[d\phi]\circ \kappa_{H'}.
\end{equation*}
Moreover, as in Proposition \ref{prop:curvature_commute}, $\phi$ sends (partial) integral elements to (partial) integral elements, and the PDE-integrability of $(P',H')$ implies the PDE-integrability of $(P,H)$.
\end{observation}


\subsection{Abstract prolongations}

Going back to the definition of prolongation of a PDE $P\subset J^kR$ (see equation \eqref{eq:PDE_prolongation}), one finds that, for $P$ integrable up to order $k+1$, the projection $P^{(1)} \arr P$ maps $\ker(\theta^{(1)})$ at a given point $p\in P^{(1)}$ to a {\it single} integral element of $(P,\theta)$ (Definition \ref{def:integral_elements}), where both $\theta^{(1)}$ and $\te$ are restrictions of the Cartan forms of $J^{k+1}R$ and $J^k R$. This will be explained in Example \ref{example_normalised_prolongation}; the following definition extracts the right properties so that the phenomenon described above happens in general for a Pfaffian morphism:

\begin{definition}\label{def:prolongation}
An (abstract) {\bf prolongation} of a Pfaffian fibration $(P,\theta)$ over $M$ consists of a Pfaffian fibration $(P',\theta')$ over $M$ together with a Pfaffian morphism $\phi: (P',\te')\to (P,\te)$, such that 
\begin{equation}\label{eq:vert_2}
\g(\theta')\subset \ker(d\mathrm{\phi}),
\end{equation}
and, for any $u,v\in\ker(\theta')$,
\begin{equation}\label{eq:curv_2}
\kappa_{\theta}(d\mathrm{\phi}(u),d\mathrm{\phi}(v))=0.
\end{equation}
We say that $\phi:P'\to P$ is a {\bf normalised prolongation} if $\g(\theta')= \ker(d\mathrm{\phi})$.
\end{definition}

As already mentioned, we obtain a practical criterion to test when a Pfaffian morphism is a prolongation in terms of integral elements (Definition \ref{def:integral_elements}).

\begin{proposition}\label{criterion_abstract_prolongation}
A Pfaffian morphism $\phi: (P',\te') \arr (P,\te)$ is an abstract prolongation if and only if, for every point $p' \in P'$, the subspace $d\phi(\ker(\te'_{p'})) \subset T_{\phi(p')}P$ is an integral element.
\end{proposition}
\begin{proof}
Assume that $\phi$ is an abstract prolongation and choose any partial integral element $V\subset \ker(\theta_{p'})$ of $P'$. By property 4 of Proposition \ref{prop:curvature_commute} $d\phi(V)\subset d\phi(\ker\theta'_{p'})$ is a partial integral element. Since $d\phi(V)$ is transversal to the $\pi$-fibres, then $d\phi(\ker\theta'_{p'})$ is also transversal. Condition \eqref{eq:vert_2} says that $d\phi(\ker\theta'_{p'})=d\phi(V\oplus\g(\te')_{p'})=d\phi(V)$, implying that $d\phi(\ker\theta'_{p'})$ is a partial integral element. With equation \eqref{eq:curv_2} we conclude that it is actually an integral element.

Conversely, if $d\phi(\ker\theta'_{p'})$ is an integral element then equation \eqref{eq:curv_2} follows. To show \eqref{eq:vert_2} we use that $ d\phi(\ker\theta'_{p'})$ is, in particular, a partial integral element. As before, choose any partial integral element $V\subset \ker(\theta_{p'})$; then we obtain by Proposition \ref{prop:curvature_commute} that $d\phi(V)\subset d\phi(\ker\theta'_{p'})$ is a partial integral element. By dimensional reasons $d\phi(V)= d\phi(\ker\theta'_{p'})$, hence 
$$d\phi(V)= d\phi(\ker\theta'_{p'})=d\phi(V\oplus\g(\te')_{p'}), \quad \text{and} \quad d\phi(\g(\te_{p'}))\subset \g(\te)_p$$
for $p=\phi(p')$. The last equation holds again by Proposition \ref{prop:curvature_commute}. This implies that $d\phi(\g(\te_{p'}))\subset \g(\te)_p\cap d\phi(V)\subset T^\pi P\cap d\phi(V)=\emptyset$, hence it shows \eqref{eq:vert_2}.
\end{proof}

\begin{example}\label{example_normalised_prolongation}
As anticipated in Example \ref{example_pfaffian_fibration}, given a PDE $P\subset J^kR$ integrable up to order $k+1$, the projection $d\mathrm{pr}: (P^{(1)}, \te^{(1)}) \arr (P,\te)$ is a normalised prolongation. In fact this is the content of Proposition \ref{prop:PDE_prolongation}, together with the discussion at the beginning of this section. Moreover, it is immediate to see that
\begin{equation*}
\ker(d\mathrm{pr})=\g(\theta^{(1)}).
\end{equation*}
Indeed, by definition of $\te^{(1)}$ as the restriction of the Cartan form \eqref{eq:Cartan_form}, we see that $\te^{(1)}|_{T^\pi P^{(1)}}=d\mathrm{pr}$; therefore $\g(\theta^{(1)})=\ker(\te^{(1)}|_{T^\pi P^{(1)}})=\ker(d\mathrm{pr})$.
\end{example}

\begin{observation}[\bf Cartan-Ehresmann connections]\label{rmk:Cartan-Ehresmann} 
Consider an abstract prolongation $\phi:(P',\te') \to (P,\te)$; as a consequence of Proposition \ref{criterion_abstract_prolongation}, any section $\sigma:P\to P'$, induces the following distribution $H_\sigma\subset \ker(\theta)$ on $P$, which is made of integral elements of $\theta$ and is $\pi$-horizontal:
$$H_{\sigma,p}:=d_{\sigma(p)}\phi(\ker(\theta'_{\si(p)})) \quad \text{ for each } p\in P.$$
Such a distribution $H_\si$ is also called a {\bf Cartan-Ehresmann connection} of $(P,\te)$; in this paper it will be only used once as a technical tool (in the proof or Proposition \ref{prop:torsion}), so we refer to \cite{Yud16} for more details.
\end{observation}

\begin{observation}[\bf Alternative definition of prolongation]
Because $\phi$ is a Pfaffian morphism, the relation \eqref{eq:curvature_commute_2} between the curvatures of $\te$ and $\te'$ holds, hence we can replace condition \eqref{eq:curv_2} for the following equivalent one:
\begin{equation*}
\Phi(\kappa_{\theta'}(u,v))=0 \quad \forall u,v\in \ker(\theta^{(1)}). \qedhere
\end{equation*}
\end{observation}

Again, as in Remark \ref{rmk:morphism_distribution}, Definition $\ref{def:prolongation}$ can be reformulated using distributions instead of forms: we say that $\phi:(P',H')\to (P,H)$ is a Pfaffian prolongation if it is a Pfaffian morphism (i.e.\ $d\phi(H')\subset H$) and
\begin{equation}\label{eq:cond_prol}
\g(H')\subset \ker(d\phi), \quad \text{and } \quad \kappa_H(d\phi(u),d\phi(v))=0 \quad \text{ for all } u,v\in H'.
\end{equation}
The second equation can be equivalently written as $[d\phi](\kappa_{H'}(u,v))=0.$  The prolongation $\phi$ is normalised when 
\begin{equation}\label{eq:normalised}
\g(H')=\ker(d\phi). \end{equation}
where $[d\phi]:TP'/H'\to\phi^*(TP/H)$ is the induced map on the quotient. In this picture, the name {\it normalised} has a natural explanation: 

\begin{lemma}\label{lemma:normalised}
A Pfaffian prolongation $\phi:(P',H')\to (P,H)$ is normalised if and only if its differential $d\phi$ descends to an isomorphism between $TP'/H'$ and the pullback via $\phi$ of $T^\pi P$:
\begin{equation}\label{eq:iso_normalised}
T_pP'/H'_{p}\cong T^\pi_{\phi(p)}P, \quad [u]\mapsto d\phi(u-v),
\end{equation}
where $v\in H'_{p}$ is any vector with the property that $d\pi'(u)=d\pi'(v).$ 
\end{lemma} 

\begin{proof}
 The $\pi'$-transversality of $H'$ implies that its normal bundle is isomorphic to $T^{\pi'}P'/\g(H')$:
\begin{equation}\label{eq:1}
TP'/H'\cong T^{\pi'}P'/\g(H), \quad [u]\mapsto [u-v],
\end{equation}
where $v\in H$ is as in the Lemma \ref{lemma:normalised}. On the other hand, $d\phi(\g(H'))=0$ implies that map $d\phi$ induces 
\begin{equation}\label{eq:2}
T^{\pi'}P'/\g(H') \to T^\pi P,  \quad [w]\mapsto d\phi(w).
\end{equation}
The fact that $\phi$ is a prolongation implies that the map \eqref{eq:2} is well defined and surjective. Then, the map \eqref{eq:iso_normalised} comes from composing the maps \eqref{eq:2} and \eqref{eq:1}, and it is an isomorphism if and only if \eqref{eq:2} is injective, which is equivalent to condition \eqref{eq:normalised}.
\end{proof}

The lemma above suggests that, if $\phi$ is not normalised, we could ``fatten'' $H'$ by $\ker(d\phi)\subset T^{\pi'}P'$ to a new distribution
\begin{equation}\label{eq:canonical_normalised}
\bar H':=H'+ \ker(d\phi).
\end{equation}

\begin{proposition}
Let $\phi:(P',H')\to (P,H)$ be a prolongation of Pfaffian fibrations; then $(P',\bar H')$, for $\bar H'$ as in equation \eqref{eq:canonical_normalised}, is a Pfaffian fibration which makes $\phi:P'\to P$ into a normalised prolongation. 
\end{proposition}

We call $(P', \bar H')$ from the previous proposition the {\bf canonical normalised prolongation}.
 
 \begin{proof}
We prove first that $\g(H')=H'\cap\ker(d\pi')=H'\cap \ker(d\phi)$. Indeed, on the one hand, $\phi$ is a bundle morphism, hence $\ker(d\phi)\subset \ker(d\pi')$; on the other hand, the first condition for the prolongation $\phi$ is the inclusion $\g(H')\subset H'\cap \ker(d\phi).$

Then, the fact that $\bar H'$ has constant rank follows from dimension counting:
 \begin{align*}
 \rank(\bar H')=\rank(H')+\rank(\ker(d\phi))-\rank (H'\cap \ker(d\phi)) = \rank(H')+\rank(\ker(d\phi))-\rank(\g(H')).
 \end{align*} 
The $\pi'$-transversality of $\bar H'$ follows from the transversality of $H'\subset \bar H'$, and its $\pi'$-involutivity is just the Frobenius-involutivity of $\ker(d\phi)$.
 
Last, the prolongation is normalised by Lemma \ref{lemma:normalised}, since \eqref{eq:2} becomes injective when we replace $\g(H')$ by $\g(\bar{H'})=\ker(d\phi)$.
 \end{proof}

 \begin{observation}[\bf Normalised prolongations in terms of Pfaffian forms]\label{rmk:normalised}
 If we look at normalised prolongations in terms of 1-forms, we have various identifications that put us in the following case. Lemma \ref{lemma:normalised} identifies the quotient $TP'/\ker(\theta')$ with the pullback of $T^\pi P$ via $\phi$ on the one hand, and $\theta'$ identifies this quotient with its coefficient bundle $\mathcal{N}'$; hence, we can think that the coefficient bundle is $T^\pi P$:
 $$\mathcal{N}'=\phi^*(T^\pi P).$$
 Moreover, under this identification, the maps $d\phi: T^{\pi'}P'\to T^\pi P$ and $\theta': T^{\pi'}P'\to \mathcal{N}'$ coincide; it follows that a prolongation $\phi:(P',\theta')\to(P,\theta)$ is normalised if $\theta'$ takes values on $T^\pi P$, i.e.\
 $$\theta'\in\Omega^1(P',\phi^*(T^\pi P)),$$
 and the differential $d\phi$ coincides with $\theta'$, seen as a map on $T^{\pi'}P'$. The remaining conditions for a prolongations of course remain the same, namely 
 $$\phi^*\theta=\theta\circ \theta',\quad \textrm{and}\quad \theta(\kappa_{\theta'}(u,v))=\kappa_{\theta}(d\phi(u), d\phi(v))=0$$ 
 for all $u,v \in \ker(\theta')$.
  \end{observation}

\subsection{The partial prolongation}

To simplify the exposition, we will adopt from now on the point of view of distributions; at the end of the next section (Remark \ref{rk_classical_prol_for_forms}), we will make the appropriate comments about how this picture is adapted using 1-forms.

\

In analogy with the classical notion of prolongation of a PDE (Section \ref{sec:PDE_prolongation}), the {\it classical prolongation} of a Pfaffian fibration $\pi: (P,H) \arr M$ may be thought of as the space of its first order differential consequences; more precisely, the prolongation consists of all the integral elements of $(P,H)$ (Definition \ref{def:integral_elements}). Those can be reinterpreted, using equation \eqref{alternative_description_first_jet}, as the images of all linear splittings $\zeta: T_{\pi(p)}M\to T_pP$ of $d_p\pi$ such that
$$\mathrm{Im}(\zeta)\subset H_p,\quad \zeta^*(\kappa_H)=0.$$
The {\it partial} prolongation of $(P,H)$ takes care of the first condition.

\begin{definition}\label{def_partial_prolongation}
\index{Pfaffian fibration!partial prolongation}
 The {\bf partial prolongation of a Pfaffian fibration $\pi: (P,H) \arr M$}, denoted by $J^1_H P$, is the set of all its partial integral elements. In other words, modulo the identification \eqref{alternative_description_first_jet}, it is the subset of $J^1 P$ defined by 
 \begin{equation*}
 J^1_H P:=\{(p,\zeta) \in J^1 P \mid \zeta (T_{\pi(p)}M) \subset H_p\}. \qedhere
  \end{equation*}
\end{definition}

The classical prolongation of $(P,H)$ will sit inside $J^1_HP$, hence many of its properties are inherited from $J^1_HP$. In particular, we will prove later that both the partial and the classical prolongation can be seen as universal, the first in the world of Pfaffian morphisms (Proposition \ref{prop:classical_universal}), and the second in the world of Pfaffian prolongations (Proposition \ref{prop:partial_universal}). 



We begin by discussing the bundle structure of $J_H^1 P$.

\begin{proposition}
The partial prolongation $J^1_H P$ from Definition \ref{def_partial_prolongation} is a smooth manifold and $\mathrm{pr}:J^1_HP\to P$ is an affine bundle modelled on $\textrm{Hom}(\pi^*TM, \mathfrak{g}(H))$.
\end{proposition}

\begin{proof}
As explained above and in equation \eqref{alternative_description_first_jet}, we represents the points of $J^1P$ as pairs $(p, \xi)$ with $p\in P$ and $\xi: T_xM\rightarrow T_pP$ splitting of $d_p \pi$, where $x= \pi(p)$. Recall from Section \ref{sec:Cartan_form} that $\textrm{pr}: J^1P\rightarrow P$ is an affine bundle over $P$ with underlying vector bundle $\textrm{Hom}(\pi^*TM, T^{\pi}P)$. Indeed, any two points $(p, \xi)$ and $(p, \xi')$ in the same fibre of $J^1 P$ above $p\in P$ differ by 
$$ \overrightarrow{\xi'\, \xi}:= \xi'- \xi: T_xM\rightarrow T^{\pi}_{p}P,$$
which can be arbitrary. We remark also that $J^{1}_HP$ is the kernel of the map 
\begin{equation*}\label{prol-c1} 
e: J^1P\to \Hom(\pi^*{TM},\mathcal{N}_H),\quad e(j^1_x\beta): v\mapsto d_x\beta(v)\mod H_{\beta(x)},
\end{equation*}
and that $e$ is an affine map with underlying vector bundle map
$$ \overrightarrow{e}:  \textrm{Hom}(\pi^*TM,  T^{\pi}P)\rightarrow \textrm{Hom}(\pi^*TM,  \mathcal{N}_H), \quad \xi \mapsto \xi \mod H. $$
Since $H$ is $\pi$-transversal and therefore $pr: T^{\pi}P\rightarrow T^\pi P/\g(H)=\mathcal{N}_H, \ v\mapsto v\mod H$ is surjective, it follows that $\textrm{pr}: J^{1}_{H}P\rightarrow P$ is an affine bundle with underlying vector bundle
\begin{equation*}
 \textrm{ker}(\overrightarrow{e})= \textrm{Hom}(\pi^*TM, \mathfrak{g}(H)). \qedhere
\end{equation*}
\end{proof}

We study now the ``Pfaffian structure'' of $J^1_HP$, as well as its main properties.

\begin{theorem}\label{thm:partial_prolongation} 
The partial prolongation $J^1_HP$ of a Pfaffian fibration $(P,H)$ is the largest subbundle of $J^1P$ such that, when endowed with the restriction of the Cartan distribution
\begin{equation}\label{eq:partial_distribution}
H^{(1)}:=\mathcal{C}\cap TJ^1_HP,\  (\text{where } \mathcal{C} \text{ is the kernel of }\te_{\mathrm{can}} \text{ of equation \eqref{eq:Cartan_form}}  )
\end{equation}
the restriction of the projection $\mathrm{pr}:J^1_HP\to P$ becomes a Pfaffian morphism (Definition \ref{def_pfaffian_fibration}).
\end{theorem}

\begin{proof} 
Let us prove first that $(J^1_H P, H^{(1)})$ is a Pfaffian fibration. To see that $H^{(1)}$ is $\pi$-transversal, we compute its vertical part $H^{(1)}\cap T^\pi J^1_HP$, which is the same as the kernel of the Cartan form $\theta_{\mathrm{can}}$ when restricted to $T^\pi J^1_HP$. From the explicit definition \eqref{eq:Cartan_form} of $\theta_{\mathrm{can}}$, we see that the Cartan form restricted to $T^\pi J^1_HP$ is precisely $d\mathrm{pr}:T^\pi J^1_HP\to T^\pi P$. However, the kernel of $d\mathrm{pr}$ is the first term of the exact sequence over $J^1_HP$,
\begin{equation}\label{eq:res_sec}
0\to\g(H^{(1)})=\Hom(\pi^*TM,\mathrm{pr}^*(\g(H)))\to T^{\pi}J^1_HP\overset{d\mathrm{pr}}{\to} \mathrm{pr}^*(T^\pi P)\to 0,
\end{equation}
where this sequence comes from restricting
\begin{equation}\label{eq:ex_sec}
0\to\Hom(\pi^*TM,\mathrm{pr}^*(T^\pi P))\to T^\pi J^1P\overset{d\mathrm{pr}}{\to} \mathrm{pr}^*(T^\pi P)\to 0.
\end{equation}
 to $TJ^1_HP$. This also shows that, since $d\mathrm{pr}:T^\pi J^1_HP\to T^\pi P$ is pointwise surjective, $\theta_{\mathrm{can}}$ on $TJ^1_HP\supset T^\pi J^1_HP$ is surjective as well; hence $H^{(1)}=\ker(\theta_{\mathrm{can}}|_{T J^1_HP})$ is a distribution and 
\begin{equation}\label{eq:rank}
\rank(H^{(1)})=\rank(TJ^1_HP)-\rank(T^\pi P).
\end{equation}
The $\pi$-transversality of $H^{(1)}$ follows from dimension counting using \eqref{eq:res_sec} and \eqref{eq:rank}:
\begin{align*}
\rank(H^{(1)}+T^{\pi}J^1_HP)=\rank(H^{(1)})+\rank(T^{\pi}J^1_HP)
-\rank(\Hom(\pi^*TM,\g(H)))=\rank(TJ^1_HP).
\end{align*} 
The Frobenius-involutivity of the vertical part of $H^{(1)}$ is immediate as it is the intersection of the tangent space of a submanifold with the Frobenius-involutive distribution $\mathcal{C} \cap T^\pi J^1P.$

We have proved that $(J^1_H P, H^{(1)})$ is a Pfaffian fibration; now we see that it is also the biggest submanifold of $J^1P$ so that $\mathrm{pr}$ becomes a Pfaffian morphism. Indeed, a vector $v\in T_{j^1_x\beta}J^1 P$ belongs to the Cartan distribution if and only if
$$0=\theta_{\mathrm{can}}(v)=d\mathrm{pr}(v)-d_x\beta(d\pi(v)).$$
As the image of $d_x \be$ is in $H_{\be(x)}$ by definition of $J^1_H P$, then $d\mathrm{pr}(v)=d_x\beta(d\pi(v))\in H_{\beta(x)}$; hence $d\mathrm{pr} (H^{(1)}) \subset H$, i.e.\ $\mathrm{pr}$ is a Pfaffian morphism.

Conversely, if  $P' \subset J^1 P$ is a Pfaffian morphism over $P$, with $H':= \mathcal{C} \cap T P'$, then any $v\in H'_{j^1_x \beta}$ satisfies
$$  d_x \be (d \pi (v) ) = d\mathrm{pr} (v) \in H_{\be(x)}.$$
This implies that $j^1_x \be \in J^1_H P$, hence $P' \subset J^1_H P$.
\end{proof}

\begin{observation}\label{rmk:kernel}
From the proof above we see that the symbol space $\g(H^{(1)})$ of the partial prolongation is precisely the kernel of the differential of the projection $\mathrm{pr}: J^1_HP\to P$,
$$\g(H^{(1)})=\ker(d\mathrm{pr})=\Hom(\pi^*TM,\mathrm{pr}^*(\g(H))).$$
This condition is shared with normalised prolongations (see Definition \ref{def:prolongation}) and it means that we have an isomorphism for each $p\in J^1_HP$,
$$T_pJ^1_HP/H^{(1)}_p\cong T^\pi_{\mathrm{pr}(p)} P,\quad [u]\mapsto d\mathrm{pr}(u-v),$$
where $v\in H^{(1)}_p$ is any vector with $d\pi(u)=d\pi(v)$; compare this with Lemma \ref{lemma:normalised}.
\end{observation}

\begin{observation}\label{rmk:solutions}
Being a Pfaffian morphism, the projection $\mathrm{pr}:J^1_HP\to P$ induces a map between holonomic sections (Proposition \ref{def_pfaffian_fibration})
$$\Gamma(J^1_HP,H^{(1)})\to \Gamma(P,H),\quad \xi\mapsto \mathrm{pr}^*(\xi).$$
In fact, this map defines a 1-1 correspondence with inverse given by $\Gamma(P,H)\ni \beta\mapsto j^1\beta$. Indeed, by Lemma \ref{lemma:holonomic}, $j^1\beta$ is a section of $J^1P$ tangent to the Cartan distribution $\mathcal{C}$. Moreover, since $\beta$ is holonomic, $d_x\beta(T_xM)\subset H_{\beta(x)}$ for all $x\in\mathrm{dom}(\beta)$, i.e.\ $j^1\beta$ actually takes values in $J^1_HP$, and therefore it is tangent to $H^{(1)}=TJ^1_HP\cap \mathcal{C}$.
\end{observation}

As anticipated above, another possible characterisation of the partial prolongation $J_H P$ is that it is ``universal'' among the world of Pfaffian morphisms with target $(P,H)$.

\begin{proposition}\label{prop:partial_universal}
Any Pfaffian morphism $\phi:(P',H')\to (P,H)$ with the property that $\g(H')\subset \ker(d\phi)$ factors through a unique bundle morphism $\varphi:P'\to J^1_HP$ over $P$ so that
$$d\varphi(H')\subset H^{(1)} \text{ and  }\ [d\mathrm{pr}]\circ \varphi^*\kappa_{H^{(1)}}=[d\phi]\circ \kappa_{H'},$$
where $[d\mathrm{pr}]:\mathcal{N}_{H^{(1)}}\to \mathrm{pr}^*\mathcal{N}_H,\ [u]\mapsto [d\mathrm{pr}(u)]$ and $[d\phi]:\mathcal{N}_{H'}\to \phi^*\mathcal{N}_H,\ [u]\mapsto [d\phi(u)]$ are the induced maps on the normal bundles.
\end{proposition}

\begin{proof} 
The condition $d\varphi(H')\subset H^{(1)}$ forces the definition of $\varphi$ to be as follows: for $v\in H'_{p}$, $d\varphi(v)$ is an element of $H^{(1)}_{\varphi(p)}$. This means that for $j^1_{\pi'(p)}\beta=\varphi(p)$,
$$0=d\mathrm{pr}(d\varphi(v))-d_{\pi'(p)}\beta(d\pi(d\varphi(v)))=d\phi(v)-d_{\pi'(p)}\beta(d\pi'(v)),$$
where in the second equality we are using that $\varphi$ is a bundle map over $P$ (and hence, over $M$), thus $\mathrm{pr}\circ\varphi=\phi$ and $\pi\circ\varphi=\pi'$. This defines uniquely $\varphi(p)$ as the linear splitting $\varphi(p):T_{\pi'(p)}M \to H_{\phi(p)}$ of $d\pi$ given by $X\mapsto d\phi(v)$, where $v$ is any vector tangent to $H'_{p}$ with the property that $d\pi'(v)=X$. Of course, we still need to check that $\varphi$ is indeed well-defined, but this is a direct consequence of $\g(H')\subset \ker(d\phi)$, as one can see easily.

The equality involving the curvatures is a direct consequence of the relations between the curvatures of the Pfaffian morphisms $\phi$ and $\mathrm{pr}$, with the curvature of $H$ (Remark \ref{rmk:morphism_distribution}):
$$\phi^*\kappa_{H}=[d\phi]\circ \kappa_{H'}, \text{ and }\ \mathrm{pr}^*\kappa_H=[d\mathrm{pr}]\circ \kappa_{H^{(1)}}.$$ 
We apply then $\varphi^*$ to the second equation and use $\mathrm{pr}\circ\varphi=\phi$ to substitute in the first equation.
\end{proof}

\subsection{The classical prolongation}

Recall that the classical prolongation of a Pfaffian fibration $(P,H)$ may be thought as the space of first order consequences of the Pfaffian fibration, in analogy with the notion of prolongation of a PDE. More precisely, it is defined as the set of integral elements of $(P,H)$ (Definition \ref{def:integral_elements}), and hence it sits inside the partial prolongation $J^1_HP\subset J^1P$ (Definition \ref{def_partial_prolongation}) as the subset where the second part of condition \eqref{eq:cond_prol} holds, i.e.\
$$\mathrm{pr}^*\kappa_H=0.$$ 
Indeed, if $j^1_x\beta$ is an element of $J^1_HP$ such that for any $u,v\in H^{(1)}_{j^1_x\beta}$, $\kappa_H(d\mathrm{pr}(u),\mathrm{pr}(v))=0$, then  
$$\kappa_H(d_x\beta(d\pi(u)),d_x\beta(d\pi(v)))=0$$
because $d\mathrm{pr}(u)-d_x\beta(d\pi(u))=0$ (i.e.\ $u\in H^{(1)}_{j^1_x\beta}$), and analogously for $v$. This is exactly saying that $j_x^1\beta$ is an integral element.

\begin{definition}\label{def:classical_prolongation}
The {\bf classical prolongation} of a Pfaffian fibration $(P,H)$, denoted by $\mathrm{Prol}(P,H)$, is the set of all its integral elements. In other words, it is the subset of the partial prolongation $J^1_HP$ (Definition \ref{def_partial_prolongation}) given by 
\begin{equation*}
\mathrm{Prol}(P,H):=\{(p,\zeta) \in J^1_H P \mid \zeta^*(\kappa_H)=0\},
\end{equation*}
where $\zeta^*(\kappa_H) (u,v) := \kappa_H (\ze(u), \ze(v)) \quad \forall u, v \in T_{\pi(p)}M.$ 
\end{definition}

Studying the smooth structure of $\mathrm{Prol}(P,H)$ is a bit more subtle than in the case of the partial prolongation. The classical prolongation is the zero-set of the map
\begin{equation}\label{prol-c2} 
\widetilde{\kappa}_H: J^{1}_{H}P\rightarrow \textrm{Hom}(\pi^*\wedge^2TM, \mathcal{N}_H), \quad
(p,\zeta) \mapsto \zeta^*\kappa_{H},
\end{equation}
hence the smoothness of 
$\textrm{Prol}(P, H)$ can be studied by understanding $\widetilde{\kappa}_H$. Indeed, $\widetilde{\kappa}_H$ is an affine map, and a simple computation reveals that the underlying vector bundle morphism is precisely the map 
\begin{equation*}\label{eq:delta}
\delta_H:\Hom(\pi^*TM,\g(H))\to\Hom(\pi^*(\wedge^2TM),\mathcal{N}_H)
\end{equation*}
$$\delta_H(\eta_p)(X,Y)=\partial_H(\eta_p(X))(Y)-\partial_H(\eta_p(Y))(X).$$ 
Here $\partial_H$, called the {\bf symbol map} of $(P,H)$, is given by 
\begin{equation}\label{eq:symbol_map}
\partial_H:\g(H)\to\Hom(\pi^*TM,\mathcal{N}_H),\quad \partial_H(v)(Y)=\kappa_H(v,\bar Y)
\end{equation} 
with $\bar Y$ any vector tangent to $H_p$ that projects to $Y$, i.e.\ $d\pi(\bar Y)=Y$. One can check that $\partial_H$ is well-defined because $\g(H)$ is Frobenius-involutive. 
We deduce that:

\begin{lemma}\label{lemma:smooth} $\mathrm{Prol}(P, H)
$ is a smooth affine subbundle of $J^1P$ if and only if: 
\begin{itemize}
\item[(1)] $\delta_H$ has constant rank, and
\item[(2)] $\mathrm{pr}: \mathrm{Prol}(P,H) 
\rightarrow P$ is surjective.
\end{itemize}
\end{lemma}

Related to (1) in the previous lemma, we see that the kernel of $\delta_H$ is the first prolongation 
$$\g(H)^{(1)}:= \g^{(1)}(\partial)$$
 of the generalised tableau bundle $\partial_H:\g(H)\to\Hom(\pi^*TM,\mathcal{N}_H)$, in the sense of equation \eqref{eq:prolongation_tableau}. Accordingly, $\g(H)^{(1)}\subset \Hom(\pi^*TM,\g(H))$ is a bundle of vector spaces whose rank may vary; of course, $\delta_H$ has constant rank if and only if $\g(H)^{(1)}$ is of constant rank.
 
Now, related to (2), we see that for any two $(p,\zeta),(p,\zeta')\in J^1_HP$, the difference $\eta:=\zeta-\zeta'$ lies in $\Hom(T_{\pi(p)}M,\g(H)_p)$ and 
$$\zeta^*{\kappa}_H-\zeta'^*\kappa_H=\delta_H(\eta).$$ 
Therefore, $\widetilde{\kappa}_H$ descends to the following map, called \textbf{the torsion of $(P,H)$}:
\begin{equation}\label{eq:higher_curvature}
\tors:P \to \Hom(\pi^*(\wedge^2TM),\mathcal{N}_H)/\mathrm{Im}(\delta_H),\quad p\mapsto [ \widetilde{\kappa}_H (p, \zeta)]
\end{equation}
It is now a simple exercise to check that the zero-set of $\tors$ is precisely the image of $\textrm{pr}: \mathrm{Prol}(P,H) \rightarrow P$. In particular:

\begin{theorem}\label{thm:classical_prolongation} For any Pfaffian fibration $\pi: (P,H)\rightarrow M$, the following are equivalent: 
\begin{enumerate}
\item The prolongation $\mathrm{Prol}(P,H)$ 
is a smooth affine subbundle of $J^1R$.
\item The prolongation $\g(H)^{(1)}$ of $\mathfrak{g}(H)$ is of constant rank, and $\tors=0$ (or, equivalently, $\mathrm{pr}: \mathrm{Prol}(P,H)\to P$ is surjective). 
\end{enumerate}
Moreover, in this case:
\begin{itemize}
\item the vector bundle underling the affine bundle $\mathrm{Prol}(P,H)$ 
is precisely $\g(H)^{(1)}$.
\item if we denote the restriction of the Cartan distribution $\mathcal{C}=\ker(\theta_{\mathrm{can}})$ (see equation \eqref{eq:Cartan_form}) of $J^1P$ to $\mathrm{Prol}(P,H)$ by
\begin{equation}\label{eq:classical_distribution}
H^{(1)}:=\mathcal{C}\cap T\mathrm{Prol}(P,H),
\end{equation}
then $(\mathrm{Prol}(P,H),H^{(1)})$ becomes a Pfaffian fibration over $M$ with symbol space 
$\mathrm{pr}^*\g(H)^{(1)}\subset \mathrm{Hom}(\pi^*TM, \mathrm{pr}^*\mathfrak{g}(H))$. 
\item $\mathrm{Prol}(P,H)$ is the biggest submanifold of $J^1P$ such that, when endowed with the restriction of the Cartan distribution $\mathcal{C}$, the projection $\mathrm{pr}$ becomes a normalised prolongation.
\end{itemize}
\end{theorem} 

\begin{proof} From Lemma \ref{lemma:smooth} and the discussion thereafter we know that the first two items are equivalent. Checking that $H^{(1)}$ as in \eqref{eq:classical_distribution} is a Pfaffian distribution is completely analogous to the proof given for the partial prolongation (see Theorem \ref{thm:partial_prolongation}). 

Let us prove that $\ker(d\mathrm{pr})$ restricted to the vertical tangent of the classical prolongation $T^\pi\mathrm{Prol}(P,H)$ coincides with $\g(H)^{(1)}$. We know that $\g(H)^{(1)}\subset \mathrm{Hom}(\pi^*TM,\g(H))$ is the vector bundle that models the affine bundle $\mathrm{pr}:\mathrm{Prol}(P,H)\to P$, and hence it can be computed as the kernel of
$$d\mathrm{pr}:T^\pi\mathrm{Prol}(P,H)\to T^\pi P$$
(see sequences \eqref{eq:res_sec} and \eqref{eq:ex_sec}). On the other hand, $$\g(H^{(1)})=\ker(d\mathrm{pr}:T^\pi\mathrm{Prol}(P,H)\to T^\pi P)$$
by the very definition of $H^{(1)}$ as the kernel of the Cartan form $\theta_{\mathrm{can}}$ when restricted to $\mathrm{Prol}(P,H)$. In conclusion, $\g(H^{(1)})=\mathrm{pr}^*\g(H)^{(1)}$.

To prove that $\mathrm{pr}:(\mathrm{Prol}(P,H),H^{(1)})\to (P,H)$ is a normalised prolongation, note that $\mathrm{Prol}(P,H)$ is a subbundle of $J^1_HP$ and recall from Theorem \ref{thm:partial_prolongation} that the projection from $(J^1_H P,H^{(1)})$ to $(P,H)$ is a Pfaffian morphism. The only thing left to see is that $\mathrm{pr}^*\kappa_H=0$, which holds by construction of $\mathrm{Prol}(P,H)$ (see the discussion previous to the Definition \ref{def:classical_prolongation}).

Last, if $P' \subset J^1 P$ is another normalised prolongation over $(P,H)$, together with $H' := \mathcal{C} \cap T P'$, then $P' \subset J^1_H P$ by Theorem \ref{thm:partial_prolongation}. Moreover, since $(P',H')$ is a Pfaffian fibration, for any $j^1_x \be \in P'$ and $u_1, u_2 \in T_x M$, there exist $v_1, v_2 \in H'_{j^1_x \be}$ such that $d \pi (v_i) = u_i$. In particular, $v_i \in \mathcal{C}$, so that $d_x \be (d \pi (v_i)) = d \mathrm{pr} (v_i)$; we conclude therefore that
$$ (d_x \be)^* \kappa_H (u_1, u_2) = \kappa_H ( d_x \be (d \pi (v_1)), d_x \be (d \pi (v_2)) ) = \kappa_H (d \mathrm{pr}(v_1), d \mathrm{pr}(v_2) ) = 0 \quad \forall u_1, u_2 \in T_x M,$$
where the last equality holds by condition \eqref{eq:curv_2}. This implies that $j^1_x \be \in \mathrm{Prol} (P,H)$, i.e.\ $P' \subset \mathrm{Prol} (P,H)$.
\end{proof}

\begin{observation}\label{bijection_solutions_prolongations} A Remark analogous to \ref{rmk:solutions} goes here. More precisely, whenever $\mathrm{pr}:\mathrm{Prol}(P,H)\to P$ is a smooth bundle map, there is 1-1 correspondence between holonomic sections
$$\Gamma(\mathrm{Prol}(P,H),H^{(1)})\to \Gamma(P,H),\quad \xi\mapsto \mathrm{pr}^*(\xi),$$
with inverse $\Gamma(P,H)\ni\beta\mapsto j^1\beta$.

To check this, recall from Remark \ref{rmk:solutions} that $j^1\beta\in\Gamma(J^1_HP,H^{(1)})$. As $\beta$ is tangent to $H$ and
$$[\smash{\widetilde{d\beta(X)}},\smash{\widetilde{d\beta(Y)}}]=d\beta([X,Y])\subset H|_{\beta(M)} \quad \text{ for } X,Y\in\X(M)$$
(where the tildes indicate  $\pi$-projectable extensions of the vectors), then $(d\beta)^*\kappa_H(X,Y)=d\beta([X,Y])\mod H =0.$ This implies that $j^1\beta$ is a section of $\mathrm{Prol}(P,H)$.
\end{observation}

Again, the classical prolongation can be thought as ``universal'' among prolongations. Let us assume that $\mathrm{pr}:\mathrm{Prol}(P,H)\to P$ is a (smooth) bundle map.

\begin{proposition}\label{prop:classical_universal} Any Pfaffian prolongation $\phi:(P',H')\to (P,H)$ factors through a unique bundle map $\varphi:P'\to \mathrm{Prol}(P,H)$ over $P$, so that 
\begin{equation}\label{eq:distributions}
d\varphi(H')\subset H^{(1)},\text{ and  }\ [d\mathrm{pr}]\circ \varphi^*\kappa_{H^{(1)}}=[d\phi]\circ \kappa_{H'}=0,
\end{equation}
where $[d\mathrm{pr}]:\mathcal{N}_{H^{(1)}}\to \mathrm{pr}^*\mathcal{N}_H,\ [u]\mapsto [d\mathrm{pr}(u)]$, and $[d\phi]:\mathcal{N}_{H'}\to \phi^*\mathcal{N}_H,\ [u]\mapsto [d\phi(u)]$, are the induced maps on the normal bundles.
\end{proposition}

\begin{observation}\label{rmk:non_smooth}
Actually the above proposition can be stated in a slightly greater generality. Even if $\mathrm{Prol}(P,H)$ is not smooth, any prolongation factors through the map $\varphi:P'\to J^1_HP$ given in Proposition \ref{prop:partial_universal}. We can slightly modify the above statement by saying that this map takes values in the subset $\mathrm{Prol}(P,H)$, and that the relations with the distributions, and the curvatures hold when we take $H^{(1)}$ as the Pfaffian distribution \eqref{eq:partial_distribution} of $J^1_HP.$

As a consequence we obtain that when $(P,H)$ admits a prolongation then the projection $\mathrm{pr}:\mathrm{Prol}(P,H)\to P$ is surjective. Accordingly, we will give the proof of the above proposition without the smoothness assumption.
\end{observation}

\begin{proof} We let $\varphi: P'\to J^1_HP$ defined as in the proof of Proposition \ref{prop:partial_universal}, and we show that it takes values in $\mathrm{Prol}(P,H)$. A closer look to $\varphi(p):T_{\pi'(p)}M\to H_p$ shows that its image $\varphi(p)(T_{\pi'(p)}M)$ coincides with $d\phi(H'_{p})$, because $d\phi(\g(H'))=0$. By Proposition \ref{criterion_abstract_prolongation}$, d\phi(H'_{p})$ is an integral element, hence $\varphi(p)$ belongs to $\mathrm{Prol}(P,H).$

The left hand side condition \eqref{eq:distributions} for the distributions is immediately implied by the same condition in Proposition \ref{prop:partial_universal} for the partial prolongation, and the right hand side condition \eqref{eq:distributions} also follows from the commutativity of the curvatures in the same proposition taking into account that on $J^1_HP$, $\mathrm{pr}^*\kappa_H=[d\mathrm{pr}]\circ \kappa_{H^{(1)}}$ is zero at points of $\mathrm{Prol}(P,H)$, and that $\phi$ satisfies $\phi^*\kappa_{H}=[d\phi]\circ \kappa_{H'}=0.$
\end{proof}

Again, the motivating and inspiring example comes from the classical definition \eqref{eq:PDE_prolongation} of prolongation of a PDE $P\subset J^kR$; the next result states that it coincides with our definition of classical prolongation.

\begin{proposition}\label{prop:PDE_prolongation}
Let $P\subset J^kR$ be a $PDE$, so that $(P,H)$ is a Pfaffian fibration by Proposition \ref{PDE_is_Pfaffian_bundlle}, for $H=\mathcal{C}\cap TP$. Then,
$$\mathrm{Prol}(P,H)=P^{(1)}:=J^1P\cap J^{k+1}R,\quad \g(H)^{(1)} = \g^{(1)},$$
where $\g^{(1)}$ is as in Theorem \ref{Goldschmidt_criterion_PDE}. Moreover, if $P$ is integrable up to order $k+1$, then $\mathrm{pr}:(P^{(1)},H^{(1)})\to (P,H)$ is a normalised prolongation with $H^{(1)}=\mathcal{C}\cap TP^{(1)},$ and $\mathrm{pr}:P^{(1)}\to P$ is an affine subbundle modelled on $\g^{(1)}$.
\end{proposition}

\begin{proof}
We first recall that $J^{k+1}R$ sits inside $J^1(J^kR)$ as the splitting $\sigma:T_xM\to T_qJ^kR$ of $d\pi$ tangent to the Cartan distribution $\mathcal{C}\subset T(J^kR)$. It follows that (it can be checked in local coordinates) that  
$$\kappa_\mathcal{C}(\sigma(X),\sigma(Y))=0, \quad \text{ for all }X,Y\in T_xM.$$
Since $P^{(1)}$ is the intersection of $J^1(J^kR)$ with $J^1P$, then the splittings $\sigma$ that belong to $P^{(1)}$ are the ones satisfying the previous conditions plus the fact that its image $\sigma(T_xM)$ lies in $T_qP$. Putting all these conditions together, we see that $\sigma$ is an element of $P^{(1)}$ if and only if it belongs to the classical prolongation $\mathrm{Prol}(P,H)$.

To conclude, we observe that the definition of integrability up to order $k+1$ is saying precisely that $\mathrm{pr}:P^{(1)}\to P$ is a bundle map, hence, by Theorem \ref{thm:classical_prolongation}, $\mathrm{pr}$ is a normalised prolongation. Moreover, in this case, we have the inclusion $\g(H)\subset \ker(d\mathrm{pr}:T^\pi J^kR\to J^{k-1}R)\cong S^kT^*M\otimes T^\pi R$ (see the exact sequence \eqref{eq:jets}), and $\partial_H$ is precisely the restriction of
$$\partial_{\mathcal{C}}:S^kT^*M\otimes T^\pi R\to \mathrm{Hom}(TM,S^{k-1}T^*M\otimes T^\pi R), \ \eta\mapsto \partial_{\mathcal{C}}(\eta)(X)=\iota_X\eta.$$
Therefore, $\g(H)^{(1)}=\g^{(1)}$, and the rest follows from Theorem \ref{thm:classical_prolongation}.
\end{proof}

Coming back to Pfaffian fibrations using the language of forms we have the following remark:

\begin{observation}[\bf Classical prolongation for forms]\label{rk_classical_prol_for_forms}
Let us go back to the picture of Pfaffian fibrations $(P,\theta)$ in terms of 1-forms: all the definitions related to the partial and classical prolongation can be written directly in terms of $\te$. For example, instead of considering the distribution $H^{(1)}$ as in \eqref{eq:partial_distribution} and \eqref{eq:classical_distribution}, we look at the dual 1-form denoted by $\theta^{(1)}$, given by the restriction of the Cartan form $\te_{\mathrm{can}}$ on $J^1P$ to the partial or classical prolongation. Similarly, all the results go through in this setting with the appropriate modifications. For Theorems \ref{thm:partial_prolongation} and \ref{thm:classical_prolongation}, since the projection $\mathrm{pr}$ in both cases is a weak Pfaffian morphism, then the forms $\theta^{(1)}$ and $\theta$ are related by 
$$\mathrm{pr}^*\theta=\bar\theta\circ\theta^{(1)},$$
where $\bar\theta:\mathrm{pr}^*(T^\pi P)\to \mathrm{pr}^*\mathcal{N},\ v\mapsto \theta(v)$ is the vector bundle map between the coefficient bundle of $\theta^{(1)}$, and $\theta$.
In Propositions \ref{prop:partial_universal} and \ref{prop:classical_universal}, the condition for the distributions translate into
$$\varphi^*\theta^{(1)}=[d\phi]\circ\theta',$$ 
where $[d\phi]:\mathcal{N}'\to \varphi^*T^{\pi}P$ is the composition between the identification $T^{\pi'}P'/\g(H')$ with $\mathcal{N}'$ via $\theta'$ and the map $T^{\pi'}/\g(H') \to T^\pi P,\ [v]\mapsto [d\phi(v)]$. In the same Propositions, the relation between the curvatures becomes
$$\bar\theta\circ\varphi^* \kappa_{\theta^{(1)}}=\Phi\circ \kappa_{\theta'},$$
where $\Phi:\mathcal{N}'\to\phi^*\mathcal{N}$ is the vector bundle map between the coefficient bundles, associated to the Pfaffian morphism $\phi$ (see $\Phi$ in Definition \ref{def_pfaffian_morphism}). Of course, in Proposition \ref{prop:classical_universal} this last expression is equal to zero.
\end{observation}

\subsubsection*{Other results about prolongations}

There are some other nice consequences about the Pfaffian distributions and the prolongations involving the curvature and the prolongation of the symbol space; we list some of them.

\begin{corollary}\label{cor_prolongation_iff_kappa_zero}
Assume that $\g(H)^{(1)}$ has constant rank; then $(P,H)$ admits a Pfaffian prolongation if and only if the torsion $\tors$ vanishes.
\end{corollary}

\begin{proof}
If $(P,H)$ admits a prolongation, then Remark \ref{rmk:non_smooth} says that the projection $\mathrm{pr}:\mathrm{Prol}(P,H)\to P$ is surjective, hence $\tors=0$ by part of Theorem \ref{thm:classical_prolongation}. The converse is Theorem \ref{thm:classical_prolongation}.
\end{proof}

\begin{corollary}\label{cor:3}
The Pfaffian distribution $H\subset TP$ is Frobenius-involutive if and only if $\mathrm{Prol}(P,H)$ coincides with $J^1_HP$ and the symbol map $\partial_H$ from equation \eqref{eq:symbol_map} vanishes. 
\end{corollary} 

\begin{proof}
If $H$ is Frobenius-involutive then all partial integral elements are integral elements, hence $\mathrm{Prol}(P,H)=J^1_HP$; moreover, $\partial_H$ vanishes trivially.

Conversely, if we let $p\in P$, we can split $H_p$ as a direct sum $V\oplus \g(H)_p$, where $V$ is a partial integral element. Because $\mathrm{Prol}(P,H)=J^1_HP$, $V$ is actually an integral element. In conclusion, we compute the bracket modulo $H$ using the direct sum: for $v+u,v'+u'\in V\oplus \g(H)_p$,
\begin{align*}
\kappa_H(v+u,v'+u')&=\kappa_H(v,v')+\kappa_H(v,u')+\kappa_H(u,v')+\kappa_H(u,u')\\
&=-\partial_H(u')(d\pi(v))-\partial_H(u)(d\pi(v'))=0,
\end{align*}
where we used the Frobenius-involutivity of $\g(H).$
\end{proof}

\begin{corollary}
Let $H\subset TP$ be a Pfaffian distribution whose curvature $\tors$ vanishes; then, if two of the following three conditions hold, the third holds as well:
\begin{enumerate}
\item $\mathrm{pr}:\mathrm{Prol}(P,H)\to P$ is a bijection;
\item $\g(H)$ is zero;
\item $H$ is Frobenius-involutive.
\end{enumerate}
\end{corollary}

\begin{proof} That (1) and (2) imply (3) follows from a computation similar to that of Corollary \ref{cor:3}. Assuming (1) and (3), we have that (3) implies that $\mathrm{Prol}(P,H)=J^1_HP$ by Corollary \ref{cor:3}, and by (1) we have that for the fibre bundle $\mathrm{pr}:J^1_H P\to P$, the kernel $\ker(d\mathrm{pr})=\Hom(\pi^*TM, \mathrm{pr}^*\g(H))$ (Remark \ref{rmk:kernel}) is zero because $\mathrm{pr}$ is a bijection, hence (2). Last, to show that (2) and (3) imply (1), we see that $H$ is a horizontal distribution if and only if $\g(H)$ is zero; in this case $\mathrm{pr}:J^1_H P\to P$ is a bijection. If, moreover, $H$ is Frobenius-involutive, then $J^1_H P=\mathrm{Prol}(P,H)$ by Corollary \ref{cor:3}.
 \end{proof}

 \begin{corollary}\label{corollary_immersion_theorem}
In the setting of Proposition \ref{immersion_theorem_pfaffian_bundle}, assume that the symbol map $\g(\te) \arr \Hom(\pi^*TM,\mathcal{N})$ from equation \eqref{eq:symbol_map} is injective; then the bundle map $i: P \arr J^1 Q$ is an immersion.
\end{corollary}

\begin{proof}
It is enough to show that $di$ is injective when restricted to $\ker(df) = \g(\te)$. In turn, this follows after noticing that $di_{\mid \g(\te)}$ coincides with the symbol map, which is injective by hypothesis.
\end{proof}

\subsection{Abstract prolongations in the linear case}

In this section we discuss the theory of abstract prolongations for linear Pfaffian fibrations (introduced in Section \ref{sec:linear_Pfaffian}). In order to do that, we will use the equivalent approach using relative connections (see Proposition \ref{linear_Pfaffian_bundles_with_connections}).

\

Let $(\Ftemp',D')$, $(\Ftemp,D)$ be linear Pfaffian fibrations over $M$, with $(D',\sigma')$ a relative connection taking values in $\Ftemp$, and $(D,\sigma)$ a relative connection taking values in $\Etemp$:
\begin{equation}\label{eq:compatible}D':\Gamma(\Ftemp')\to \Omega^1(M,\Ftemp),\quad D:\Gamma(\Ftemp)\to\Omega^1(M,\Etemp).\end{equation}
The following definition will play the role of normalised prolongations between Pfaffian fibrations in the non-linear case.

\begin{definition}\label{def:compatible} The relative connections $(D',\sigma')$ and $(D,\sigma)$ as in \eqref{eq:compatible} are {\bf compatible} if 
\begin{enumerate}
\item $D\circ \sigma'=\sigma\circ D'$;
\item $D_X\circ D'_{Y}-D_Y\circ D'_{X}-\sigma\circ D'_{[X,Y]}=0$ for all $X,Y\in \mathfrak{X}(M).$
\end{enumerate}
\end{definition}

The two conditions of Definition \ref{def:compatible} above have a clear cohomological interpretation, which appeared already in \cite{Gol76a, Ngo}. For a relative connection $(D,\sigma)$ there exists a linear operator, denoted by the same letter $D$,
\begin{equation}\label{eq:higher}
D:\Omega^*(M,\Ftemp)\to \Omega^{*+1}(M,\Etemp),
\end{equation}
uniquely defined by the following two properties: it coincides with the connection $D$ on $\Gamma(\Ftemp)=\Omega^0(M,\Ftemp)$, and it satisfies the Leibniz identity relative to $\sigma$,
$$D(\omega\otimes s)=d\omega\otimes \sigma(s)+(-1)^k\omega\wedge D(s),$$
for any $k$-form $\omega\in\Omega^k(M)$, and any section $s\in\Gamma(\Ftemp)$. This operator $D$ can be given explicitly by the Koszul formula
\begin{align*}
D \eta (X_0,\ldots, X_k) =&\sum_i(-1)^i D_{X_i}(\eta(X_0,\ldots,\hat{X}_i,\ldots, X_k))\\
&+\sum_{i<j}(-1)^{i+j}\sigma(\eta([X_i,X_j],X_0,\ldots,\hat{X}_i,\ldots,\hat{X}_j,\ldots,X_k))
\end{align*}
for any $\eta\in\Omega^k(M,\Ftemp)$. A direct check shows the following lemma:

\begin{lemma}\label{lemma:composition} Let $(D',\sigma')$, and $(D,\sigma)$ be relative connections as in \eqref{eq:compatible}. If $\dim M>0$, then the relative connections are compatible if and only if the composition
$$\Omega^*(M,\Ftemp')\overset{D'}{\longrightarrow}\Omega^{*+1}(M,\Ftemp)\overset{D}{\longrightarrow}\Omega^{*+2}(M,\Etemp)$$
is zero.
\end{lemma}

For compatible relative connections $(D',\sigma')$ and $(D,\sigma)$ as above, the first condition of Definition \ref{def:compatible} implies that $\sigma'$ preserves holonomic sections. In general, the resulting map
$$\Gamma(\Ftemp',D')\to \Gamma(\Ftemp,D),\quad s\mapsto \sigma'(s) $$
is not necessarily surjective; its surjectivity is measured, in the sense of Proposition \ref{prop:surjective} below, by some map $S$ which we now present.

Denote by $\partial':\g'\to \Hom(TM,\Ftemp)$ the map given by the restriction of $D'$ to its symbol space $\g'=\ker(\sigma')$; it is linear by equation \eqref{eq:Leibniz}. Condition (1) of Definition \ref{def:compatible} implies that the image of $\partial'$ lies inside $\Hom(TM,\g)$, $\g=\ker(\sigma)$, hence $\partial'$ takes the form
$$\partial':\g'\to \Hom(TM,\g),\quad \partial'=D'|_{{\g'}}.$$
By the very definition of the operators \eqref{eq:higher} we get that at higher order $\partial'(\omega\otimes s)=(-1)^k\omega\wedge\partial'(s)$, for any $\omega\in\Omega^k(M)$ and any section $s\in\Gamma(\g')$; hence, together with Lemma \ref{lemma:composition}, this implies that the composition
$$\wedge^kT^*M\otimes\g'\overset{\partial'}{\longrightarrow}\wedge^{k+1}T^*M\otimes\g\overset{\partial}{\longrightarrow}\wedge^{k+2}T^*M\otimes \Etemp$$ 
of vector bundles over $M$ is zero. Interpreting $\g'$ as the ``prolongation" of $\g$, we consider the following quotient
$$H^{0,1}(\g):=\frac{\ker\{\partial:T^*M\otimes\g\to\wedge^2T^*M\otimes \Etemp\}}{\mathrm{Im}\{\partial':\g'\to T^*M\otimes \g\}}.$$

\begin{lemma} 
The following map is well defined:
$$S:\Gamma(\Ftemp,D)\to H^{0,1}(\g),\ s\mapsto [D'(\bar s)],$$
where $\bar s$ is a section of $\Ftemp'$ such that $\sigma'(\bar s)=s.$
\end{lemma}

\begin{proof} If $s'\in\Gamma(\Ftemp')$ is another section with the same property as $\bar{s}$, then $\alpha:=\bar s-s'$ belongs to $\g'$ and $\partial'(\alpha)=D'(\bar s)-D'(\bar s')$. This means that $D'(\bar s)$ and $D'(s')$, which are a priori sections of $\Hom(TM,\g)$ (since $\sigma(D'(\bar s))=D(\sigma' (\bar s))=D(s)=0$, and the same for $s'$), represent the same class on the quotient by $\mathrm{Im}(\partial')$. Moreover, for vector fields $X,Y\in\mathfrak{X}(M)$, 
$$\partial(D'(\bar s)) (X,Y) =D_XD'_{Y}(\bar s)-D_YD'_{X}(\bar s)-\sigma D'_{[X,Y]}(\bar s),$$
which is zero by condition (2) of Definition \ref{def:compatible}. Hence, $S$ is indeed well defined. 
\end{proof}

\begin{proposition}\label{prop:surjective}
For compatible connections as in \eqref{eq:compatible}, the following sequence is exact
\begin{equation*}
\Gamma(\Ftemp',D')\overset{\sigma'}{\longrightarrow}\Gamma(\Ftemp,D)\overset{S}{\longrightarrow}H^{0,1}(\g). \qedhere
\end{equation*} 
\end{proposition}

\begin{proof}
If $\alpha$ is a holonomic section of $D'$, then $S(\sigma'(\alpha))$ is equal to the class of $D'(\alpha)=0$, so $S \circ \si' = 0$. Moreover, if $S(s)=[D'(\bar s)]=0$, then there is a section $\beta$ of $\g'$ so that $D'(\bar s)=\partial'(\beta)=D'(\beta)$. In particular, the section $s':=\bar s-\beta$ of $\Ftemp'$ is holonomic and is such that $\sigma'(s')=\sigma'(\bar s)=s$, so the sequence is exact.
\end{proof}

When looking at linear Pfaffian fibrations in terms of the linear Pfaffian forms, we realise that the definition of compatible connections coincides with the linear counterpart of normalised prolongations (see Remark \ref{rmk:normalised}). Let $\theta'$ and $\theta$ be linear forms, and let $D'$ associated to $\theta'$ as in \eqref{eq:relative_connection}:
$$D': \Gamma(\Ftemp')\to \Omega^1(M,\Ftemp),\quad s\mapsto s^*\theta',$$
and $D$ associated to $\theta$ in the same way: $D(u)=u^*\theta$, $u\in\Gamma(\Ftemp)$.

\begin{lemma}
Two relative connections $(D',\sigma')$ and $(D,\sigma)$ as in equation \eqref{eq:compatible} are compatible (Definition \ref{def:compatible}) if and only if $\sigma':(\Ftemp',\theta')\to(\Ftemp,\theta)$ is a normalised prolongation. Moreover, any other normalised prolongation $\phi:(\Ftemp',\theta')\to(\Ftemp,\theta)$ with $\phi$ linear is, up to automorphisms of $\Ftemp$, of the form $\phi=\sigma'=\theta'|_{\g'}$. 
\end{lemma}

\begin{proof}
First of all, as $\sigma'$ is by definition the restriction of $\theta'$ to $\g(\theta')=\ker(\theta') \cap T^{\pi'} E'$, and as $\sigma'$ is linear, its differential $d \si'$ coincides with $\sigma'$ when restricted to $T^\pi_v \Ftemp'= \Ftemp'_{\pi'(v)}$  for any $v\in \Ftemp'$ (we are using the canonical identification of these vector spaces). From this we get for free the condition that 
$$\g(\theta')=\pi'^*\g(D)=\pi'^*\ker(\sigma')=\ker d\sigma'.$$
It follows that the coefficient bundle of $\theta'$ (which is, up to isomorphism, the normal bundle $TP'/\ker(\te')\cong T^{\pi'}P/\g(\theta')$ by $\pi'$-regularity of $\theta'$) is precisely $\pi'^*\Ftemp$ (see also Remark \ref{rmk:normalised}).  

From the correspondence \eqref{eq:relative_connection}, the relation $\sigma'^*\theta=\sigma\circ\theta'$ between the Pfaffian forms is translated into the equivalent condition (1) of Definition \ref{def:compatible}, i.e.\ $D\circ\sigma'=\sigma\circ D'$ in terms of the relative connections.

To see that the condition on the curvatures of $\theta'$ and $\theta$ is the same as condition (2) of Definition \ref{def:compatible} for compatible connections, we write $\sigma\circ \kappa_{\theta'}$ as the restriction to $\ker(\theta')$ of the skew-symmetric bilinear map 
$$T\Ftemp'\times T\Ftemp'\to {\pi'}^*\Etemp, \quad (u,v)\mapsto -d_{D}\theta'(u,v).$$
Here $d_D\theta\in\Omega^2(\Ftemp',\pi'^*(\Etemp))$ at $U,V\in\mathfrak{X}(\Ftemp')$ is defined by the De-Rham-type formula 
$$d_D\theta(U,V)=D^{\pi'}_U(\theta'(V))-D^{\pi'}_V(\theta'(U))-\sigma(\theta'[U,V]),$$
with $D^{\pi'}:\mathfrak{X}(\Ftemp')\times \Gamma({\pi'}^*\Ftemp)\to \Gamma(\pi'^*\Etemp)$ the pullback of $D$ via $\pi':\Ftemp'\to M$; of course, when $U,V$ belong to $\ker(\theta')$, $-d_D\theta(U,V)$ coincides with $\sigma(\kappa_{\theta'}(U,V))$. As $\sigma\circ \kappa_{\theta'}=\sigma'^*\kappa_\theta$, and $d\sigma'$ is zero on the vertical part $\ker(\theta') \cap T^{\pi'} E'$ because it coincides with $\sigma'$ on $\g(D')=\ker\sigma'$, then a straightforward check shows that $\sigma\circ \kappa_{\theta'}$ is zero if and only if $s^*(\sigma\circ \kappa_{\theta'})_x=0$ for any $x\in M$ and any $s\in\Gamma(\Ftemp')$ such that $s^*(\theta')_x=0$. However,
\begin{equation*}\label{eq:linear_curvature}
s^*(\sigma\circ \kappa_{\theta'})_x(X,Y)=s^*(d_D\theta')_x(X,Y)=D_X\circ D'_{Y}(s)(x)-D_Y\circ D'_{X}(s)(x)-\sigma\circ D'_{[X,Y]}(s)(x),
\end{equation*}
so we conclude that $\sigma\circ \kappa_{\theta'}$ is zero if and only if condition (2) of Definition \ref{def:compatible} holds.\\

Last, consider a normalised prolongation $\phi:(\Ftemp',\theta')\to(\Ftemp,\theta)$ between linear Pfaffian fibrations and assume that $\phi$ is also linear; then, in view of Remark \ref{rmk:normalised}, 
we can assume that $\theta'$ takes values on $\phi^*T^{\pi}\Ftemp$, which, in turn, is isomorphic to $\phi^*\pi^*(\Ftemp)=\pi'^*\Ftemp$ (again we use the canonical isomorphism of $T^\pi \Ftemp$ with $\pi^*(\Ftemp)$). We also assume that, under these isomorphisms, $d\phi$ coincides with $\theta'$ on $T^{\pi'}\Ftemp'$.  Again, as $\phi$ is linear its differential $d\phi$ when restricted to the vertical vector bundle $T^{\pi'} \Ftemp'=\pi'^*\Ftemp'$ coincides with $\phi$; hence, on $\Ftemp'=T^{\pi'}\Ftemp'|_M$
\begin{equation*}
\phi=d\phi=\theta'=\sigma'. 	\qedhere
\end{equation*}
\end{proof}

\subsection{Partial and classical prolongations in the linear case}

Let us continue the discussion on prolongations for linear Pfaffian fibrations; we will find again that many objects, which were in general over $\Ftemp$, become linear objects over $M$ described in terms of relative connections.

\begin{definition}\label{partial_prolongation_linear_pfaffian_bundle}
The {\bf partial prolongation of a linear Pfaffian fibration} $(\Ftemp,D)$ is
$$J^1_D\Ftemp:=\{j^1_xs \in J^1 \Ftemp \mid D(s)(x)=0\}.$$
\end{definition}

Since the linear form $\te$ associated to $D$ is characterised by $s^*\theta=D(s)$ and $\theta|_\Ftemp=\sigma$, it is immediate to check that the partial prolongation of $(E,D)$ as a linear Pfaffian fibration from Definition \ref{partial_prolongation_linear_pfaffian_bundle} coincides with the partial prolongation of $(E,\te)$ as a Pfaffian fibration from Definition \ref{def_partial_prolongation}, i.e.\ $J^1_D\Ftemp = J^1_\theta \Ftemp$. Similarly to Theorem \ref{thm:partial_prolongation} (together with the fact that the $J^1 \Ftemp$ is a linear Pfaffian fibration), we can characterise $J^1_D\Ftemp$ as the largest vector subbundle of $J^1\Ftemp$ over $M$, with the property that the projection $\mathrm{pr}:J^1_D\Ftemp\to \Ftemp$ is a Pfaffian morphism. In this language, this means that $J^1_D\Ftemp$ is the largest subbundle so that condition (1) of Definition \ref{def:compatible},
$$\sigma\circ D^{(1)}=D\circ\mathrm{pr},$$
 holds for the restriction 
$D^{(1)}:\Gamma(J^1_D\Ftemp)\to \Omega^1(M,\Ftemp)$ of the classical Spencer operator from equation \eqref{eq:Spencer_operator}.

At the level of sections, the partial prolongation can be also described as follows
\begin{proposition}
Let $(P,D)$ be a linear Pfaffian fibration; then
\begin{equation}\label{eq:spencer_decomposition}
\Gamma(J^1_D\Ftemp)=\{(\alpha,\omega)\in \Gamma(\Ftemp)\oplus \Omega^1(M,\Ftemp)\mid D(\alpha)=\sigma\circ \omega\}.
\end{equation}
\end{proposition}

\begin{proof}
Using the decomposition \eqref{eq:Spencer_decomposition}, a section $(\alpha,\omega)$ of $J^1\Ftemp$ at $x$ is precisely the splitting 
\begin{equation}\label{eq:splitting}
d_x\alpha-\omega_x:T_xM\to T_{\alpha(x)}\Ftemp,
\end{equation}
 where $\omega_x$ is viewed as a map from $T_xM$ to $T_{\alpha(x)}^\pi \Ftemp$, when canonically identifying $T_{\alpha(x)}^\pi \Ftemp$ with $\Ftemp_x$. Therefore, the image of $(\alpha,\omega)_x$ belongs to $\ker(\theta)$ if and only if for all $X\in T_xM$
\begin{equation*}
0=\theta(d_x\alpha(X)-\omega(X))=\theta(d_x\alpha(X))-\theta(\omega(X))=\alpha^*\theta_x(X)-\sigma(\omega(X))=D_X(\alpha)-\sigma(\omega(X)). \qedhere
\end{equation*}
\end{proof}

Let us repeat the same discussion for the classical prolongation.

\begin{definition}\label{classical_prolongation_linear_pfaffian_bundle}
The {\bf classical prolongation of a linear Pfaffian fibration} $(\Ftemp, D)$ is
$$\mathrm{Prol}(\Ftemp,D):=\ker(K),$$ 
where $K$ is the vector bundle map
\begin{equation}\label{eq:K}K:J^1_D\Ftemp\to \Hom(\wedge^2 TM,\Etemp)\end{equation}
defined at the level of sections, for any $X,Y\in\mathfrak{X}(M)$, as
\begin{equation*}
K(\alpha,\omega)(X,Y)=D_X(\omega(Y))-D_Y(\omega(X))-\sigma(\omega[X,Y]). \qedhere
\end{equation*}
\end{definition}

As a consequence of the Lemma \ref{eq:curvaturas} below, one sees that the classical prolongation of $(E,D)$ as a linear Pfaffian fibration from Definition \ref{classical_prolongation_linear_pfaffian_bundle} coincides with the classical prolongation of $(E,\te)$ as a Pfaffian fibration from Definition \ref{def:classical_prolongation}, i.e.\ $\mathrm{Prol}(\Ftemp,D) = \mathrm{Prol}(\Ftemp,\theta)$. As the relative connection $D^{(1)}$ of $J^1_D\Ftemp$ is the projection to the second component of $\Gamma(J^1_D\Ftemp)\subset \Gamma(\Ftemp)\oplus \Omega^1(M,\Ftemp),$ the classical prolongation can be alternatively written as 
$$\mathrm{Prol}(\Ftemp,D)=\{j^1_xs \in J^1_D E \mid D_X\circ D^{(1)}_{Y}(s)(x)-D_Y\circ D^{(1)}_{X}(s)(x) = \sigma\circ D^{(1)}_{[X,Y]}(s)(x)\},$$
i.e.\ $\mathrm{Prol}(\Ftemp,D)$ is the largest bundle of vector spaces of $J^1_D\Ftemp$, where the condition (2) of Definition \ref{def:compatible} holds. 

\begin{lemma}\label{eq:curvaturas}
Let $(\Ftemp,\theta)$ be a linear Pfaffian fibration, with  $\theta\in\Omega^1(\Ftemp,\pi^*\Etemp)$, and let $(D,\sigma)$ be the associated relative connection. Then the map $\widetilde{\kappa}_H:J^1_D\Ftemp\to  \Hom(\pi^*\wedge^2 TM,\pi^*\Etemp)$ from equation \eqref{prol-c2} is precisely $-\pi^*K$, with $K$ as in \eqref{eq:K}.
\end{lemma}

\begin{proof}
Using the Spencer decomposition \eqref{eq:spencer_decomposition}, let $(\alpha,\omega)\in\Gamma(J^1_D\Ftemp)$; in terms of the form $\theta$, this means that $\alpha^*\theta=\theta\circ\omega$. Following \eqref{eq:splitting}, for $X,Y\in\X(M)$ we regard $d\alpha(X)-\omega(X)$ as a $\pi$-projectable vector field on $\ker(\theta)$, so that $\omega(X)$ is the vector field constant along the fibres of $\Ftemp$ and extending $\omega(X)$ (strictly speaking, we choose a $\pi$-projectable extension inside $\ker(\theta)$ so that it coincides with $d\alpha(X)-\omega(X)$ along $\alpha(M)\subset T\Ftemp$); we do the same for $d\alpha(Y)-\omega(Y)$. With this,
\begin{align*}
(\alpha,\omega)^*\kappa_{\theta}(X,Y)&=\theta([d\alpha(X)-\omega(X),d\alpha(Y)-\omega(Y)])\\
&=\theta[d\alpha(X),d\alpha(Y)]-\theta([d\alpha(X),\omega(Y)])-\theta([\omega(X),d\alpha(Y)-\omega(Y)])\\
&=\alpha^*\theta([X,Y])-\theta([d\alpha(X),\omega(Y)])+D_Y(\omega(X))\\
&=D_{[X,Y]}(\alpha)-\theta([d\alpha(X),\omega(Y)])+D_Y(\omega(X)),
\end{align*}
where in third line we use Remark \ref{rmk:operator_distribution} saying that $D_{Y}(\omega(X))$ is precisely $\theta([\omega(X),d\alpha(Y)-\omega(Y)])$; recall also that $(\alpha,\omega)$ belonging to $J^1_D\Ftemp=J^1_\theta \Ftemp$ means precisely that $d\alpha(X)-\omega(X)\in\ker(\theta)$ for all $X\in\X(M)$. Now, using the fact that vector fields constant along the fibres of $\Ftemp$ commute, we get that $[\omega(X),\omega(Y)]=0$, and therefore $\theta([d\alpha(X),\omega(Y)])$ can be computed as
\begin{align*}\theta([d\alpha(X),\omega(Y)])&=\theta([d\alpha(X),\omega(Y)])-\theta([\omega(X),\omega(Y)])+\theta([\omega(X),\omega(Y)])\\
&=\theta([d\alpha(X)-\omega(X),\omega(Y)])=-D_X(\omega).
\end{align*} 
Putting the two equations above together and using that $D(\alpha)=\sigma(\omega)$, we conclude the proof.
\end{proof}

As pointed out in the general discussion, $\mathrm{Prol}(\Ftemp,D)$ might fail to be a (smooth) fibre bundle over $\Ftemp$, the reasons being the lack of surjectivity of the projection $\mathrm{pr}:\mathrm{Prol}(\Ftemp,D)\to \Ftemp$, and that the rank over $M$ might vary. However, in this linear picture things simplify and the exact sequence \eqref{eq:exact_sec} for $J^1\Ftemp$ restricts to the exact sequence of vector bundles over $M$,
\begin{equation}\label{exact_sequence}
0\to\g(D)^{(1)}\to \mathrm{Prol}(\Ftemp,D)\overset{\mathrm{pr}}{\to}\Ftemp.
\end{equation}
Here $\g(D)^{(1)}$ is the first prolongation of the symbol space $\g(D)$, viewed as a tableau in the sense of equation \eqref{eq:prolongation_tableau}, with
$$\partial_D:\g(D)\to \Hom(TM, \Etemp),\quad \partial_D(v)(X)\mapsto D_X(v);$$
using $\g(D)=\ker(\sigma)$ and the Leibniz identity of $D$ w.r.t.\ $\sigma$, one can easily verify that $\partial_D$ is a well-defined linear map. One checks that the sequence \eqref{exact_sequence} is exact by considering a section of $J^1_D\Ftemp$ that belongs to $\mathrm{Prol}(\Ftemp,D)$, which lives inside $\ker(\mathrm{pr})$, i.e.\ its second component in the decomposition \eqref{eq:spencer_decomposition} is zero. 

Now, the surjectivity of $\mathrm{pr}:\mathrm{Prol}(\Ftemp,D)\to \Ftemp$ is of course related to the map $K$ of equation \eqref{eq:K}. Indeed, letting  
$$\delta_D:\mathrm{Hom}(TM,\g(D))\to\mathrm{Hom}(\wedge^2TM,\Etemp),$$
defined by $\delta_D(\eta)(X,Y)=\partial_D(\eta(X))(Y)-\partial_D(\eta(Y))(X)$, we see that $K$ descends to a vector bundle map 
$$T:\Ftemp\to\Hom(\wedge^2TM,\Etemp)/\mathrm{Im}(\delta_D),\ p\mapsto[K(\xi)],$$
where $\xi\in J^1_D\Ftemp$ is any element that projects to $p$; it is a straightforward computation using the decomposition \eqref{eq:spencer_decomposition} that $T$ is well defined. It is now a simple exercise to check that the zero-set of $T$ is precisely the image of $\mathrm{pr}:\mathrm{Prol}(\Ftemp,D)\to \Ftemp$. Thus, we have just proved the following:

\begin{proposition} The classical prolongation $\mathrm{Prol}(\Ftemp,D)$ is a (smooth) subbundle of $J^1\Ftemp\to \Ftemp$ if and only if $T=0$ and the prolongation $\g(D)^{(1)}$ has constant rank. In this case, the restriction of the Spencer operator
$$D^{(1)}:\Gamma(\mathrm{Prol}(\Ftemp,D))\to \Omega^1(M,\Ftemp),$$
is compatible with D.
\end{proposition}

As in Remark \ref{rmk:solutions}, even not assuming any smoothness condition on $\mathrm{Prol}(\Ftemp,D)$, the map 
$$\Gamma(\mathrm{Prol}(\Ftemp,D), D^{(1)})\to\Gamma(\Ftemp,D),\quad \xi\mapsto \mathrm{pr}\circ \xi$$
defines a bijection, with inverse $s\in \Gamma(\Ftemp,D)\mapsto j^1s$. Moreover, $D^{(1)}$ is universal among the connections compatible to $D$ in the following sense:

\begin{proposition} If $(\Ftemp',D')$ is a relative connection compatible with $(\Ftemp,D)$, then there exists a unique vector bundle map $j: \Ftemp'\to \mathrm{Prol}(\Ftemp,D)$ so that 
$$D'=D^{(1)}\circ j.$$
\end{proposition}

Of course the above proposition is consequence of Proposition \ref{prop:classical_universal} for non-linear prolongations. We only remark that, in this case, $j=\varphi$ is defined in terms of $D'$, and at the level 
of sections is given by 
$$j(s)=(\sigma'(s),D'(s))\in\Gamma(\Ftemp)\oplus \Omega^{1}(M,\Ftemp).$$
The conditions for compatible connections mean that $j(s)$ actually lands in $\mathrm{Prol}(\Ftemp,D).$

\begin{observation}
As we had remarked on \ref{rmk:operator_distribution}, in the linear case many of the objects associated to a Pfaffian fibration sit on top of $M$. Of course, for any linear distribution $H$, the symbol map $\partial_H$ of equation \eqref{eq:symbol_map}, the prolongation $\g^{(1)}(H):= \g^{(1)}(\partial)$ of equation \eqref{eq:prolongation_tableau}, and the torsion map $\tors$ of equation \eqref{eq:higher_curvature}, are just pullbacks of the analogous objects for the associated relative connection $D$. In fact, from Remark \ref{rmk:operator_distribution} we know that $\g(H)\cong\pi^*\g(D)$ and this isomorphism comes from the canonical identification of $T^\pi \Ftemp$ with $\pi^*\Ftemp$ by translating vertical vectors to the zero section, Therefore, using the description of $D$ in terms of $H$ as in Remark \ref{rmk:operator_distribution} we have
\begin{equation*}
\partial_H=\pi^*\partial_D,\quad \g^{(1)}(H)\cong\pi^*\g^{(1)}(D), \quad \tors=\pi^*T. \qedhere
\end{equation*}
\end{observation}

\begin{observation}[\bf Linearisation of Pfaffian prolongations along holonomic sections]
As we did for Pfaffian fibrations (Section \ref{sec:linearisation}), we can linearise Pfaffian normalised prolongations 
$$\phi:(P',\theta')\to (P,\theta)$$
along a holonomic section $\xi \in \Gamma(P',\theta')$ and its image $\phi(\xi)\in\Gamma(P,\theta)$, and obtain compatible connections
$$ \mathrm{Lin}_{\xi}(P', \theta') \overset{D'^\xi}{\longrightarrow} \mathrm{Lin}_{\phi(\xi)}(P, \theta)\overset{D^{\phi(\xi)}}{\longrightarrow}\phi(\xi)^*\mathcal{N}.$$

As a particular case, if $P' = \mathrm{Prol} (P,\te)$, $\phi = \mathrm{pr}$ and $\xi = j^1 \beta$, for $\beta$ a holonomic section of $(P,\te)$ (so that $\mathrm{pr}(\xi) = \beta$), the 
 functoriality of linearisation implies that 
$$\mathrm{Prol}(\mathrm{Lin}_{\beta}(P, \theta),D^\beta)=\mathrm{Lin}_{j^1\beta}(\mathrm{Prol}(P,\te)),\quad D^{(1)}=D^{j^1\beta}.$$

This linearisation becomes particularly nice when applied to Pfaffian groupoids along the unit section, where the multiplicativity allows us to translate properties of the linearisation to the analogous properties of the Pfaffian groupoid (see Remark \ref{obs_pfaffian_groupoids}).
\end{observation}

\section{Integrability of Pfaffian fibrations}

Informally speaking, when we prolong a Pfaffian fibration $(P,H)$, we are trying to determine if an element of $(P,H)$ comes from a section which is ``holonomic up to order 1''; if we prolong again then we are looking for sections which are ``holonomic up to order 2'', etc. If we can repeat this process indefinitely, we find a formal holonomic section of the Pfaffian fibration i.e.\ a Taylor series of a potential holonomic section of $(P,H)$.

Let us be more specific. To simplify the notation, denote by
$$ P^{(1)} := \mathrm{Prol}(P,H)$$
the classical prolongation of $(P,H)$ from Definition \ref{def:classical_prolongation}. Under the conditions of Theorem \ref{thm:classical_prolongation}, the projection $P^{(1)} \arr P$ is a fibration and the prolongation is in turn a smooth Pfaffian fibration over $M$. We could therefore build the classical prolongation of $P^{(1)}$ and denote it by $(P^{(2)}, H^{(2)})$; this sits inside a jet bundle, as $ P^{(2)} \subset J^1_{H^{(1)}} P^{(1)} \subset J^1 P^{(1)}$, but may not be a smooth submanifold, and the projection over $P^{(1)}$ may not be a fibration. However, if we apply again Theorem \ref{thm:classical_prolongation}, we find conditions under which also $P^{(2)}$ is a Pfaffian fibration over $M$. When this process can be carried out up to ``infinity'' we say that $(P,H)$ is {\it formally integrable}. The goal of this section is to formalise this procedure and describing precisely the obstructions to formal integrability.

\subsection{Integrability up to finite order}

\begin{definition}
A Pfaffian fibration $(P,H) = (P^{(0)}, H^{(0)})$ is called {\bf integrable up to order $k \geq 1$} when, for all $i = 1,\ldots,k$, the classical prolongations
$$ P^{(i)} := \mathrm{Prol}(P^{(i-1)}, H^{(i-1)}) \subset J^1_{H^{(i-1)}} P^{(i-1)}$$
are smooth submanifolds, and the projections $P^{(i)} \arr P^{(i-1)}$ are surjective submersions.
\end{definition}

In particular, if $(P,H)$ is integrable up to order $k$, it follows from Theorem \ref{thm:classical_prolongation} that each $P^{(i)}$ is a Pfaffian fibration over $M$, when endowed with the distribution $H^{(i)} := (H^{(i-1)})^{(1)}$, and $\mathrm{pr}:(P^{(i)},H^{(i)})\to (P^{(i-1)},H^{(i-1)})$ is precisely the classical prolongation of the Pfaffian fibration $(P^{(i-1)}, H^{(i-1)})$. We call $(P^{(i)},H^{(i)})$ the {\bf $i^{th}$ classical prolongation} of the Pfaffian fibration $(P,H)$, for $i = 1, \ldots, k$.

\begin{observation}\label{cor_integrability_order_k}
Let $(P,H)$ be a Pfaffian fibration integrable up to order $k$. Then, for every integers $i,l \leq k$ with $i+l \leq k$,
\begin{itemize}
 \item $(P,H)$ is also integrable up to order $i$.
 \item The Pfaffian fibration $(P^{(i)}, H^{(i)})$ is integrable up to order $l$, and its $l^{th}$-prolongation $(P^{(i)})^{(l)}$ coincide with the $(i+l)^{th}$-prolongation $P^{(i+l)}$ of $(P,H)$.
 \item The holonomic sections of $(P,H)$ are in bijections with the holonomic sections of $(P^{(i)}, H^{(i)})$.
\end{itemize}
Properties 1 and 3 are immediate from the definition and from Remark \ref{bijection_solutions_prolongations}. For the second property, note that $P^{(i)} \subset J^i P$ is a PDE, and recall from Proposition \ref{prop:PDE_prolongation} that prolongations of Pfaffian fibrations and PDEs coincide. Our claim becomes then precisely \cite[Theorem 7.2]{Gol67b}.
\end{observation}

\begin{example}
If $P \subset J^l R$ is a PDE, the notion of integrability up to order $k$ in the sense of Pfaffian fibrations coincides with the notion of integrability up to order $k$ in the sense of PDEs (see Section \ref{sec:PDE_prolongation}); this follows directly from Proposition \ref{prop:PDE_prolongation}. 
\end{example}

We describe now the main obstructions for integrability up to finite orders. The first step, which takes care of the first prolongation $P^{(1)}$, was already discussed in Theorem \ref{thm:classical_prolongation}. In particular, one needs two conditions:
\begin{enumerate}
\item the projection $\mathrm{pr}:P^{(1)}\to P$ is surjective, which, in turn, was shown to be equivalent to the vanishing of the torsion map \eqref{eq:higher_curvature}. 
\item the prolongation $\g^{(1)}= \g(H)^{(1)}$ of the symbol space $\g = \g(H)$ is of constant rank, where $\g^{(1)}$ is given by \eqref{eq:prolongation_tableau}, applied to $\partial_H:\g=\g(H)\to \Hom(\pi^*TM,\mathcal{N}_H)$.
\end{enumerate}
Under these conditions, $P^{(1)}$ becomes an affine bundle over $P$ modelled on $\g^{(1)}$, as well as a smooth Pfaffian fibration (over $M$). Moving one step upwards, we
unravel now these conditions 1 and 2 when applied to the prolongation of $P^{(1)}$, $\mathrm{pr}:P^{(2)}\to P^{(1)}$, and then we continue this analysis inductively.
First of all, the (higher) prolongations that are relevant in condition 2 will be precisely the ones from Section \ref{section:tableaux}: 
$$\mathfrak{g}^{(i)}= \pi^*S^iT^*M\otimes \g\cap\mathrm{Hom}(\pi^*TM, \g^{(i-1)})=\ker(\delta_i),\quad \textrm{for $i>1$},$$
with $\delta_i$ as in \eqref{eq:differential}. This can also be rewritten using the following inductive lemma (see also Lemma 6.3 of \cite{Gol67b}):
 
 \begin{lemma}\label{lemma:up_to_k} If a Pfaffian fibration $(P,H)$ is integrable up to order $k\geq 1$, then we  have the following canonical isomorphisms of bundles of vector spaces over $P^{(i)}$, $1\leq i\leq k$
 \begin{equation}\label{eq:equality_tableau} \mathrm{pr}^*\g^{(i+1)} \cong \mathrm{pr}^*\g(H^{(1)})^{(i)} \cong \ldots \cong \mathrm{pr}^*\g(H^{(i-1)})^{(2)} \cong \g(H^{(i)})^{(1)}.\end{equation}
 Moreover, for every $i \leq k-1$, $\g^{(i)}$ is a vector bundle, whose pullback $\mathrm{pr}^*\g^{(i)}$ over $P^{(i-1)}$ models the affine bundle $\mathrm{pr}:P^{(i)}\to P^{(i-1)}$.
\end{lemma}

\begin{proof} First of all, we regard $\g^{(i)}$ sitting inside of $\pi^*S^iT^*M \otimes \g\subset \pi^*(S^iT^*M)\otimes T^\pi P$. Having in mind the exact sequence \eqref{eq:jets} of vector bundles over $J^{i}P$, and recalling that the symbol space of $(J^{i}P,\mathcal{C})$ is precisely $\ker(d\mathrm{pr}:T^\pi J^kR\to TJ^{k-1}R)\cong \pi^*S^{i-1}T^*M\otimes\mathrm{pr}^*T^\pi P$, one can check that $\delta_i$ coincides with the restriction of the symbol map 
$$\delta_\mathcal{C}:\textrm{Hom}(\pi^*TM, \pi^*S^{i-1}T^*M\otimes\mathrm{pr}^*T^\pi P)\to \textrm{Hom}(\pi^*\wedge^2TM,  \pi^*S^{i-2}T^*M\otimes\mathrm{pr}^*T^\pi P)$$
 (see also the proof of Proposition \ref{prop:PDE_prolongation}, where we look at this $\partial_{\mathcal{C}}$).
Also, we can regard $(P^{(i)},H^{(i)})$, for $i=1,\ldots,k$, as a PDE endowed with the restriction $H^{(i)}$ of the Cartan distribution $\mathcal{C}\subset TJ^i P$. Having all these in mind, and using the equality of the prolongations from Proposition \ref{prop:PDE_prolongation}, we can prove inductively the canonical isomorphisms \eqref{eq:equality_tableau}. Moreover, $\mathrm{pr}:P^{(i)}\to P^{(i-1)}$ is an affine bundle modelled on the vector bundle $\mathrm{pr}^*\g^{(i)}=(\g^{(i-1)})^{(1)}$ (we set $\g^{(0)}=\g$). 
\end{proof}

We now move to the condition 1. For a Pfaffian fibration $(P,H)$ integrable up to order $k$, the discussion after  Definition \ref{def:classical_prolongation} tells us that the prolongation $(P^{(k)},H^{(k)})$ is the kernel of the map \eqref{prol-c2}
\begin{equation}\label{eq:curvature_prol} 
\begin{split}
\widetilde{\kappa}_{H^{(k)}}: J^1_{H^{(k)}} P^{(k)} \arr \Hom (\pi^*\wedge^2 TM, \mathrm{pr}^*T^{\pi}P^{(k-1)}) \\
j^1_x \si \mapsto (\si^*\kappa_{H^{(k)}})_x = (\kappa_{H^{(k)}})_x (d_x \si (\cdot), d_x \si (\cdot) ).
\end{split}
\end{equation}
In the last $\Hom$-space we have used the identification of the normal bundle $\mathcal{N}_{H^{(k)}}$ with $\mathrm{pr}^*T^{\pi}P^{(k-1)}$ (via the differential $d\mathrm{pr}$) because $\mathrm{pr}:(P^{(k)},H^{(k)})\to (P^{(k-1)},H^{(k-1)})$ is a normalised prolongation (see Remark \ref{rmk:normalised}). Also, $\widetilde{\kappa}_{H^{(k)}}$ is an affine map of affine bundles over $P^{(k)}$, where $J^1_{H^{(k)}} P^{(k)}\to P^{(k)}$ is modelled on $\Hom (\pi^*TM,\g(H^{(k)}))$, with
 $$\g(H^{(k)})=\mathrm{pr}^*\g(H^{(k-1)})^{(1)}\cong \mathrm{pr}^*\g^{(k)}$$ 
 where the first equality is by (part of) Theorem \ref{thm:classical_prolongation}, and the second by Lemma \ref{lemma:up_to_k}. Thus, the underlying vector bundle morphism of $ \widetilde{\kappa}_{H^{(k)}}$ is of the form
 $$\overrightarrow{\widetilde{\kappa}_{H^{(k)}}}:\Hom(\pi^*TM,\mathrm{pr}^*\g^{(k)})\to \Hom (\pi^* \wedge^2 TM, \mathrm{pr}^*T^{\pi}P^{(k-1)}),$$
and a computation reveals that it is precisely the pullback via $\mathrm{pr}$ of the Spencer differential $\delta_k$ from equation \eqref{eq:differential} (see the proofs of Lemma \ref{lemma:up_to_k} and Proposition \ref{prop:PDE_prolongation}). Thus, $P^{(k+1)}:=\mathrm{Prol}(P^{(k)},H^{(k)})$ is a smooth affine subbundle of $J^1_{H^{(k)}} P^{(k)}\to P^{(k)}$ if and only if 
\begin{itemize}
\item[1'.]  $P^{(k+1)}\to P^{(k)}$ is surjective;
\item [2'.] $\delta_k$ has constant rank, i.e.\ $\ker(\delta_k)=\g^{(k+1)}$ has constant rank.
\end{itemize}
 Related to 1', this discussion also implies that $\widetilde{\kappa}_{H^{(k)}}$ descends to the following map: 

\begin{definition}
 Let $(P,H)$ be a Pfaffian fibration integrable up to order $k\geq 1$.
 The {\bf torsion of order $k+1$} of $(P,H)$ is defined to be the torsion \eqref{eq:higher_curvature} of $(P^{(k)},H^{(k)})$, i.e.\ the map
 \begin{equation*} 
 \tors^{k+1} := \tors (P^{(k)}) : P^{(k)} \arr \frac{\Hom (\pi^* \wedge^2 TM, \mathrm{pr}^*T^{\pi}P^{(k-1)})}{\de(\Hom(\pi^*TM, \mathrm{pr}^*\g^{(k)}))}, \quad p \mapsto [ \sigma^*(\kappa_{H^{(k)}})_x ] = [ \widetilde{\kappa}_{H^{(k)}} (j^1_x \si) ],
 \end{equation*}
 where $j^1_x \si$ is any element of the partial prolongation $J^1_{H^{(k)}} P^{(k)}$ s.t. $\si(x) = p.$
By definition we set $P^{(0)} = P$ and $\tors^1=\tors$.
 \end{definition}

From the general discussion of the classical prolongation, we know already that the zero-set of $\tors^k$ is precisely the image of $P^{(k+1)}\to P^{(k)}$. Hence, from Theorem \ref{thm:classical_prolongation} we obtain:

 \begin{proposition}\label{integrability_one_step_higher}
Let $(P,H)$ be a Pfaffian fibration integrable up to order $k$. Then $(P,H)$ is integrable up to order $k+1$ if and only if
\begin{itemize}
 \item the torsion $\tors^{k+1}$ vanishes
 \item the prolongation $\g^{(k+1)}$ is smooth
\end{itemize}
Moreover, the classical prolongation $$\mathrm{pr}:(P^{(k+1)},H^{(k+1)})\to (P^{(k)},H^{(k)})$$ has symbol $\g (H^{(k+1)}) = \mathrm{pr}^*\g^{(k+1)}$, and it is an affine bundle over $P^{(k)}$ modelled on $\mathrm{pr}^*\g^{(k+1)}$.
\end{proposition}

\begin{obsx}[\bf Pfaffian fibrations and geometric structures]
The name \textit{torsion} originates from the theory of $G$-structures. More precisely, given a $G$-structure $P$, its torsions are objects defined recursively, whose vanishing are obstructions to the integrability of $P$. In particular, the torsion of $P$ are the same thing as the torsions of the Pfaffian fibration $\tilde{P}$ associated to $P$ (see Example \ref{example_G_structure}).



More generally, one can revise the theory of Pfaffian fibrations by taking into account the presence of a symmetry group(oid), in order to define more refined obstructions to integrability, called {\it intrinsic torsions}. These can be used to study (formal) integrability of a large class of geometric structures (which includes $G$-structures as a particular case), namely those described by any Lie pseudogroup: see \cite{Cat19}.
\end{obsx}

To understand better $\tors^{k+1}$ we look at its image; at the end of the section we will prove the following:

\begin{proposition}\label{prop:torsion}
 Let $(P,H)$ be a Pfaffian fibration integrable up to order $k \geq 1$. Then its torsion $\tors^{k+1}$ takes values in the Spencer cohomology groups \eqref{eq:Spencer_cohomology} of the tableau bundle $\g=\g^{(0)}=\g(H)$
 $$H^{k-1,2}(\g) = \frac{\ker (\de: \Hom (\pi^* \wedge^2 TM, \g^{(k-1)} ) \arr \Hom (\pi^* \wedge^3 TM, \g^{(k-2)}) ) }{\Ima(\de: \Hom(\pi^*TM, \g^{(k)}) \arr \Hom (\pi^* \wedge^2 TM, \g^{(k-1)}))} $$
 where we set $\g^{(-1)}=\mathcal{N}_H$, and we regard the prolongations $\g^{(i)}$ sitting on top of $P^{(k)}$ via the pullback by $\mathrm{pr}.$
 \end{proposition}

If we assume that some prolongation $\g^{(i)}$ of the symbol space has rank 0, then the Spencer cohomology group $H^{i,2}(\g)$ vanishes. In particular, by Proposition \ref{prop:torsion}, the torsion $\tors^{i+2}$ is zero; this suggests that for certain types of Pfaffian fibrations, Proposition \ref{integrability_one_step_higher} becomes simpler.

This leads us to the following definition:

\begin{definition}\label{def:finite}
 A Pfaffian fibration $(P,H)$ is of {\bf finite type $l$} if $l$ is the smallest integer $l\geq 0$ such that $\g^{(l)} = 0$. 
 We say that $(P,H)$ is of {\bf infinite type} if $\g^{(l)} \neq 0$ $\forall l$.
 \end{definition}
 
With this, it follows from Proposition \ref{prop:torsion} that

\begin{corollary}\label{integrability_finite_type_case}
 Let $(P,H)$ be a Pfaffian fibration of finite type $l$. If $(P,H)$ is integrable up to order $k$ and $l < k$, then it is integrable up to order $k+i$, $i\geq 0$. Moreover, $\mathrm{pr}:P^{(j)}\to P^{(j-1)}$ is a bijection for all $j\geq l$.
\end{corollary}

\begin{proof} Because $i\geq 0$, then the finite type condition says that $\g^{(k+i-1)}=0$ (as $k+i-1\geq l$), and therefore $\tors^{k+i+1}$ vanishes (see the discussion before Definition \ref{def:finite}). Also $\g^{(k+i+1)}$ has obviously constant rank equal to 0, and we can apply Proposition \ref{integrability_one_step_higher} inductively on $i$ to conclude that $(P,H)$ is integrable up to order $k+i$. Now, Lemma \ref{lemma:up_to_k} tells us that $P^{(j)}\to P^{(j-1)}$ is an affine bundle modelled on $\mathrm{pr}^*\g^{(j)}$, so if $j\geq l$, then $\g^{(j)}=0$, and therefore $P^{(j)}\to P^{(j-1)}$ is a bijection.
\end{proof}







 \begin{proof}[Proof of Proposition \ref{prop:torsion}]
 We check the case $k=1$, using the Pfaffian form $\theta$ associated to $H\subset TP$, and the Pfaffian form $\theta^{(1)}$ associated to $H^{(1)}\subset TP^{(1)}$. The general case $k\geq 1$ follows similarly.  
 
  First of all, we check that the map $\widetilde{\kappa}_{H^{(1)}}=\widetilde{\kappa}_{\te^{(1)}}$ of equation \eqref{eq:curvature_prol} takes values in
  $$\Hom (\pi^*\wedge^2TM, \g) \subset \Hom (\pi^* \wedge^2 TM, \mathrm{pr}^* T^\pi P).$$
  Indeed, an element $j^1_x\sigma$ belongs to $ J^1_{H^{(1)}}P^{(1)}$ if $d_x\sigma(T_xM)\subset H^{(1)}_{\si(x)}$; thus, since the classical prolongation $\mathrm{pr}:(P^{(1)},H^{(1)})\to (P,H)$ is normalised (Theorem\ref{thm:classical_prolongation}), we have
 $$\theta(\kappa_{\theta^{(1)}}(d_x\sigma(X),d_x\sigma(Y)))=0,$$ 
 for any $X,Y\in T_xM$ (see Remark \ref{rmk:normalised}). In conclusion, $\widetilde{\kappa}_{H^{(1)}}(j^1_x\si)(X,Y)\in\ker(\theta)$, therefore it is in $\g = \ker(\te) \cap T^\pi P$.
  
  \
 
 Now, we check that $\widetilde{\kappa}_{\te^{(1)}}$ takes values in the kernel of $$\de_\theta=\de_H: \Hom (\pi^* \wedge^2 TM, \g = \g^{(0)} ) \arr \Hom (\pi^* \wedge^3 TM, \mathcal{N}_H = \g^{(-1)} ).$$
 
 In order to do that, let $j^1_x\sigma\in J^1_{H^{(1)}}P^{(1)}$ and $X,Y,Z$ vector field on $M$; we need to compute
\begin{multline}\label{eq:beginning}
 \del_H (\kappa_{H^{(1)}} (j^1_x \si) (X, Y) ) (Z) =\del_H ( \kappa_{H^{(1)}} (d_x \si (X), d_x \si (Y) ) ) (Z) \\= \kappa_H ( \kappa_{H^{(1)}} (d_x \si (X), d_x \si (Y) ), \si(X)(Z) ).
\end{multline}

 First, we extend $d\si(X),d\si(Y),d\si(Z)\in TP^{(1)}$ to local vector fields $\bar X,\bar Y,\bar Z$ on $P^{(1)}$ which are simultaneously $\pi$- and $\mathrm{pr}$-projectable; in particular, this means $d\pi(\bar X)=X$, and similarly for $\bar Y$ and $\bar Z$. These extensions are always possible as $\mathrm{pr}$ is a submersion and a fibre bundle map over $M$, hence one can simultaneously trivialise $P^{(1)}$ around $\si(x)$ as $\mathbb{R}^{k+n+m}$, $P$ around $\mathrm{pr}(\si(x))$ as $\mathbb{R}^{n+m}$, and $M$ around $x$ as $\mathbb{R}^n$, so that $\mathrm{pr}$ and the two maps to $M$ become standard projections.

 Moreover, consider the pullback via $\mathrm{pr}:P^{(1)}\to P$ of some torsion-free linear connection $\nabla:\X(P)\times\X(P)\to\X(P)$ (e.g.\ the Levi-Civita connection of some fixed Riemannian metric on $P$); in the following we will use the same notation $\nabla$ also for the pullback connection on $\mathrm{pr}^*TP$. We can now compute the term $\kappa_{H^{(1)}} (d_x \si (X), d_x \si (Y) )$ in equation \eqref{eq:beginning} using $\na$ (see the discussion after equation \eqref{eq:curvature}):
 \begin{multline}\label{eq:d}
 \kappa_{\te^{(1)}} (d_x \si (X), d_x \si (Y) ) = d_\nabla\theta^{(1)}(d_x\si(X),d_x\si(Y)) = \\ (\nabla_{\bar X}\theta^{(1)}(\bar Y))_{\si(x)} - (\nabla_{\bar Y}\theta^{(1)}(\bar X))_{\si(x)} - \theta^{(1)} ( [\bar X,\bar Y]_{\si(x)}).
 \end{multline}

From the the definition \eqref{eq:Cartan_form} of $\te^{(1)}$ as Cartan form, we see that the last term vanishes:
\begin{equation*}\label{eq:eq}
\begin{aligned}
\theta^{(1)}_{\si(x)} ([\bar X,\bar Y]_{\si(x)}) = \theta^{(1)}_{\sigma(x)}(d_x\sigma([X,Y])) = 0.
\end{aligned}
\end{equation*}

Note that we use $\si(x)$ also to denote the splitting $\si(x):T_xM\to T_{\mathrm{pr}(\si(x))}P$. In the second equality we also used that $[\bar X,\bar Y]_{\si(x)}=d\si ([X, Y]_{x})$ because $\bar X,\bar Y$ are $\pi$-projectable and $\sigma$ is a section of $\pi$. In the second equality we used the fact that $j^1_x\si$ is an element of $ J^1_{H^{(1)}}P^{(1)}$, therefore $\theta^{(1)}\circ d_x\sigma=0$.



On the other hand, in order to rewrite the other two terms of \eqref{eq:d} we use
$$\theta^{(1)}_p(\bar X)=d_p\mathrm{pr}(\bar X)-p(X) = d_p\mathrm{pr}(\bar X)- \bullet(X)(p) ,\quad \forall p\in P^{(1)}$$
where we write again $p$ for the induced splitting $p:T_{\pi(p)}M\to T_{\mathrm{pr}(p)}P$, and we denote by $\bullet (\bar X)$ the section of $\mathrm{pr}^*(TP)\to P^{(1)}$ defined by $\bullet (X)(p)=p(X)$. We have therefore written $\theta^{(1)}(\bar X)$ as the sum of two sections of $\mathrm{pr}^*(TP)$; doing the same also for $Y$ we get
\begin{equation}\label{eq:uff}
\begin{aligned}
\nabla_{\bar X}(\theta^{(1)}&(\bar Y)) -\nabla_{\bar Y} (\theta^{(1)}(\bar X)) = \nabla_{d\mathrm{pr}(\bar X)}(d\mathrm{pr}(\bar Y))-\nabla_{d\mathrm{pr}(\bar Y)}(d\mathrm{pr}(\bar X))-\nabla_{\bar X}(\bullet(Y))+\nabla_{\bar Y}(\bullet( X)) \\
&=[d\mathrm{pr}(\bar X), d\mathrm{pr}(\bar Y)]-\nabla_{\bar X}(\bullet(Y))+\nabla_{\bar Y}(\bullet (X))
=d\mathrm{pr}[\bar X, \bar Y]-\nabla_{\bar X}(\bullet(Y))+\nabla_{\bar Y}(\bullet (X)).
\end{aligned}
\end{equation}

Here we used in the first line the definition of pullback connection via $\mathrm{pr}$, i.e.\ $\nabla_{\bar X}(d\mathrm{pr}(Y))=\nabla_{d\mathrm{pr}(\bar X)}(d\mathrm{pr}(\bar Y))$, because the section $d\mathrm{pr}(\bar Y)\in\Gamma(\mathrm{pr}^*TP)$ is already the pullback of the section $\mathrm{pr}^*(d\mathrm{pr}(\bar Y))\in \X(P)$ (recall that they are $\mathrm{pr}$-projectable vector fields). The first equality of the second line follows from the fact that $\nabla$ is torsion-free. For the last equality, as $d_x\si$ takes values in $H^{(1)}_{\si(x)}$, we have $d\mathrm{pr}(\bar X_{\si(x)})=\sigma(x)(X)$; in particular, $d\mathrm{pr}[\bar X, \bar Y]_{\si(x)}=d\mathrm{pr}d_x\si[X, Y]=\si(x)[X, Y]$.

\

We compute the last two terms of \eqref{eq:uff} at $\sigma(x)\in P^{(1)}$: since $\bullet(X)_{\si(x)}=\si(x)(X)=\bar X_{\sigma(x)}$, and similarly for $Y$, we have
\begin{equation}\label{eq:calculo}
-(\nabla_{\bar X}(\bullet(Y)))_{\si(x)} + (\nabla_{\bar Y}(\bullet (X)))_{\si(x)} = -\nabla_{\bullet(X)_{\si(x)}}(\bullet(Y)) + \nabla_{\bullet (Y)_{\si(x)}}(\bullet (X)).
\end{equation} 

Now, choose a local Cartan-Ehresmann connection $C\subset H$ extending $\si(x)(T_xM)=d\mathrm{pr}(H^{(1)}_{\si(x)})\subset H_{\mathrm{pr}(\si(x))}$ (see Remark \ref{rmk:Cartan-Ehresmann}). As $p:T_{\pi(p)}M\to T_{\mathrm{p}}P$ denotes an integral element of $(P,H)$ for $p\in P^{(1)}$, then locally $p(X)=C_p (X)+\eta_p(X)$ for every $X\in X(M)$, with $\eta_p$ some element in $\g^{(1)}_{\mathrm{pr}(p)}$. It follows that, locally,
$$\bullet(X)=\mathrm{pr}^*C(X)+S,$$
where $S$ is a finite sum of terms of the form $f\mathrm{pr}^*(\eta)(X)$, for $\eta\in\Gamma(\g^{(1)})$ and $f\in C^{\infty}(P^{(1)})$ such that$f(\si(x))=0$ (as $C_{\si(x)}=d_x\si$, and $\si(x)(X)=d_x\si(X)$). To simplify notation, we assume that locally $S$ is given by a single term, i.e.
$$\bullet(X)=\mathrm{pr}^*C(X)+f\mathrm{pr}^*(\eta)(X), \quad \eta \in \Ga(\g^{(1)}), f\in C^{\infty}(P^{(1)}), f(\si(x))=0, \forall X\in\X(M) .$$

A direct calculation shows that the right-hand side of \eqref{eq:calculo} is (up to pullbacks and coefficients) a $\mathcal{C}^{\infty}(P^{(1)})$-linear combinations of five kinds of terms (the first three come from $\nabla$ being torsion-free, and the last two from its Leibniz property):
\begin{equation}\label{eq:ya}
\textrm{(i) }[C(X),C(Y)],\quad \textrm{(ii) }[\eta(X),\eta(Y)],\quad \textrm{(iii) }[C(X)+\eta(X),C(Y)+\eta(Y)],\quad  \textrm{(iv) }\eta(X),\quad  \textrm{(v) }\eta(Y).
\end{equation}

In conclusion, we plug our results in equation \eqref{eq:beginning} to get 
\begin{equation}\label{eq:uff!}
\begin{aligned}
\partial_\theta(\widetilde{\kappa}_{H^{(1)}}(j^1_x\si)(X,Y))(Z)&=\kappa_{\theta}(d_\nabla\theta^{(1)}(d_x\si(X),d_x\si(Y)),\si(x)(Z))\\
&=\kappa_\theta(\sigma(x)[X,Y], \sigma(x)(Z))+\kappa_\theta(t_1 \text{(iv)} +t_2 \text{(v)} ,\si(x)(Z))\\
& \ +\kappa_{\theta,\sigma(x)} (r_1\text{(i)}+r_2\text{(ii)}+r_3\text{(iii)}, C(Z)+f\eta(Z))
\end{aligned}
\end{equation}
where the enumeration indicates terms as in \eqref{eq:ya}, $t_1,t_2\in\mathbb{R}$, and $r_1, r_2, r_3\in C^{\infty}(P^{(1)})$. Now, the theorem is proved once we show that
$$\delta_\theta(\widetilde{\kappa}_{H^{(1)}}(j^1_x\si))(X,Y,Z)=\partial_\theta(\widetilde{\kappa}_{H^{(1)}}(j^1_x\si)(X,Y),Z)+\textrm{ cyclic permutations of }(X,Y,Z) = 0.$$ 
Indeed, terms like the first one in the second line of \eqref{eq:uff!} are zero because $\sigma(x)$ is an integral element, i.e.\ $\sigma(x)^*\kappa_\theta=0$. Terms involving $\eta(\cdot)$ and $\si(x)(\cdot)$, such as the second one in the second line of $\eqref{eq:uff!}$, vanish as well, since $\eta\in\g^{(1)}$.

Last, all the terms inside $\kappa_\theta$ in the third line of \eqref{eq:uff!} are vector fields taking values in $H$: indeed, $[C(X),C(Y)]$ and $[C(X)+\eta(X),C(Y)+\eta(Y)]$ are in $H$ because $C$ is a Cartan-Ehresmann connection, and the same holds for $C+\eta$, since $\eta\in\g^{(1)}$ and $\eta(X),\eta(Y) \in \g \subset H$. Therefore, $\kappa_\theta$ evaluated in these terms can be computed as $\theta([\cdot,\cdot])$; we can use the Jacobi identity to show that the part of $\delta_\theta$ involving these terms vanishes.
\end{proof}

\subsection{Formal integrability}

\begin{definition}
A Pfaffian fibration is called {\bf formally integrable} when it is integrable up to any order.
\end{definition}

When a Pfaffian fibration $(P,H)$ is a PDE, it follows from Corollary \ref{cor_integrability_order_k} that the definition of formal integrability coincides with the homonymous one, introduced in Section \ref{sec:PDE_prolongation}. In particular, formal integrability is not always a sufficient condition for PDE-integrability. However, as for PDEs, the situation is nicer in the analytic setting, where we can use Theorem \ref{existence_analytic_solutions_PDE}, to prove the following result:

\begin{theorem}[\bf Existence of analytic local holonomic sections]\label{thm:integrability}
If $(P,H)$ is an analytic formally integrable Pfaffian fibration, then for every $p \in P^{(k)} \subset J^k P$ over $x \in M$ there is an analytic local holonomic section $\be$ of $(P,H)$ such that $j^k_x \be = p$ on a neighbourhood of $x \in \dom(\be)$. In particular, $(P,H)$ is PDE-integrable.
\end{theorem}
\begin{proof}
If $(P,H)$ is formally integrable, its classical prolongation $P^{(1)} \subset J^1 P$ is a formally integrable PDE. 
Moreover, since $P$ is an analytic manifold, $J^1_H P$ is analytic as well, being the kernel of the analytic bundle map $e$
of equation \eqref{prol-c1}. Similarly, $P^{(1)} \subset J^1_H P$ is analytic because it is the kernel of $\widetilde{\kappa}_H$, which is also an analytic bundle map. 
We conclude that $P^{(1)}$ is an analytic formally integrable PDE, so we can apply Theorem \ref{existence_analytic_solutions_PDE}, which gives precisely the first part of our statement.

In particular, for every $p \in P^{(k)}=(P^{(1)})^{(k-1)}$ over $x$, there exists a solution $\be$ of the PDE $P^{(1)}$ such that $ j^k_x\beta=p$. This means that $\alpha=j^1\beta$ sits inside $P^{(1)}$, i.e.\ $\alpha$ is a holonomic section of $(P^{(1)},H^{(1)})$, and therefore $\mathrm{pr}(\alpha)=\beta$ is a holonomic section of $(P,H)$. The PDE-integrability of $(P,H)$ follows from the PDE-integrability of $P^{(1)}$ and the fact that $\mathrm{pr}:P^{(1)}\to P$ is surjective.
\end{proof}

We look now for sufficient conditions for formal integrability. An immediate one follows from Corollary \ref{integrability_finite_type_case}:

\begin{proposition}
Let $(P,H)$ be a Pfaffian fibration of finite type $l$. If $P$ is integrable up to order $k > l$, then it is formally integrable.
\end{proposition}

This proposition follows also as a corollary from a straightforward generalisation of the cohomological integrability criterion of Goldschmidt (Theorem \ref{Goldschmidt_criterion_PDE}):

\begin{theorem}
Let $(P,H)$ be a Pfaffian fibration such that
\begin{itemize}
\item The symbol space $\g$ is 2-acyclic, i.e.\ $H^{l, 2}(\g) = 0$ $\forall l \geq 0$
\item $\g^{(1)}$ is smooth and $P^{(1)} \arr P$ is surjective
\end{itemize}

Then $P$ is formally integrable.
\end{theorem}

\begin{proof}
From the fact that $\g$ is 2-acyclic and $\g^{(1)}$ is smooth, it follows from Lemma \ref{lemma:2-acyclic} and Remark \ref{rmk:lemma} that $\g^{(l)}$ is smooth also for $l \geq 1$.
Moreover, thanks to our hypotheses, $P$ is already integrable up to order 1 by Theorem \ref{thm:classical_prolongation}. Assume now that $P$ is integrable up to order $l \geq 1$: then the torsion $\tors^{l+1}: P^{(l)} \arr H^{l-1,2}(\g) = 0$ must vanish, hence $P$ is integrable up to order $l+1$ by Proposition \ref{integrability_one_step_higher}. By induction we find that $P$ is formally integrable.
\end{proof}

\end{document}